\documentclass[12pt]{article}

\usepackage{settings}

\title{The Schwartz space for the $ (k, a) $-generalized Fourier transform and the minimal representation of the conformal group}
\author{%
  Tatsuro Hikawa%
  \thanks{Graduate School of Mathematical Sciences, the University of Tokyo, 3-8-1 Komaba, Meguro-ku, Tokyo 153-8914, Japan}%
}
\date{}

\begin{document}

\maketitle
\begin{abstract}
  This paper studies an analog of the classical Schwartz space $ \mscrS(\setR^N) $
  in the framework of $ (k, a) $-deformed harmonic analysis
  associated with the $ (k, a) $-generalized Fourier transform $ \mscrF_{k, a} $.
  Motivated by the observation that $ \mscrS(\setR^N) $ coincides with
  the space of smooth vectors for the Segal--Shale--Weil representation,
  we define the $ (k, a) $-generalized Schwartz space
  $ \mscrS_{k, a}(\setR^N) $ as the space of smooth vectors
  for the unitary representation associated with $ \mscrF_{k, a} $.
  Since this definition is intrinsic to the representation,
  it follows immediately that $ \mscrS_{k, a}(\setR^N) $ is preserved by $ \mscrF_{k, a} $.
  As main results, we explicitly determine $ \mscrS_{k, a}(\setR^N) $ for $ N = 1 $,
  as well as for general $ N $ when $ k = 0 $ and $ a $ is rational.
  We also explicitly determine the space of smooth vectors
  for the $ L^2 $-model of the minimal representation of the conformal group
  $ \widetilde{\mathit{SO}}_0(N + 1, 2) $ studied by Kobayashi--Mano.
\end{abstract}

\tableofcontents

\section{Introduction}

The Schwartz space $ \mscrS(\setR^N) $ plays a fundamental role in harmonic analysis on Euclidean space.
A key property is that the Fourier transform maps $ \mscrS(\setR^N) $ onto itself as a topological linear isomorphism.
This paper studies an analog of the Schwartz space
in the framework of $ (k, a) $-deformed harmonic analysis
associated with the \emph{$ (k, a) $-generalized Fourier transform} $ \mscrF_{k, a} $
introduced by Ben~Saïd--Kobayashi--Ørsted~\cite{MR2566988,MR2956043}.
Here, $ k $ is the multiplicity parameter arising in Dunkl theory,
and $ a > 0 $ is a parameter that provides a continuous interpolation between
the minimal representations of the two simple Lie groups
$ \mathit{Mp}(N, \mathbb{R}) $ and $ \widetilde{\mathit{SO}}_0(N + 1, 2) $.

The transform $ \mscrF_{k, a} $ is a unitary operator
that intertwines the multiplication operator $ \enorm{x}^a $ with
the differential-difference operator $ -\enorm{x}^{2 - a} \Laplacian_k $.
Note that the $ (0, 2) $-generalized Fourier transform $ \mscrF_{0, 2} $ coincides with
the classical Fourier transform, which intertwines $ \enorm{x}^2 $ with $ -\Laplacian $.

\subsection{Background from representation theory}
\label{ssec:background}

Our approach is motivated by the observation
that the Schwartz space $\mscrS(\setR^N)$ coincides with the space of smooth vectors
for the Segal--Shale--Weil representation (also known as the oscillator representation).
We adopt this as a guiding principle and focus on spaces of smooth vectors (see \cref{def:smooth-vector}) throughout this paper.

Ben~Saïd--Kobayashi--Ørsted introduced
a family of $ \mathfrak{sl}_2 $-triples of differential-difference operators 
on $ \setR^N \setminus \setenum{0} $:
\[
  \DiffH{k, a}  = \frac{2}{a} E_x + \frac{2 \dindex{k} + a + N - 2}{a}, \qquad
  \DiffEp{k, a} = \frac{i}{a} \enorm{x}^a, \qquad
  \DiffEm{k, a} = \frac{i}{a} \enorm{x}^{2 - a} \Laplacian_k,
\]
where $ E_x = \sum_{j = 1}^{N} x_j \Pdif{x_j} $ denotes the Euler operator,
and $ \Laplacian_k $ denotes the Dunkl Laplacian, which reduces to
the classical Laplacian $ \Laplacian $ when $ k = 0 $.

This $ \mathfrak{sl}_2 $-triple lifts to a unitary representation
\[
  \map{\Omega_{k, a}}{\widetilde{\mathit{SL}}(2, \setR)}{\mscrU(L^2(\setR^N, w_{k, a}(x) \,dx))}
\]
of the universal covering group $ \widetilde{\mathit{SL}}(2, \setR) $ of $ \mathit{SL}(2, \setR) $.
The $ (k, a) $-generalized Fourier transform is defined
as the unitary inversion operator arising in $ \Omega_{k, a} $:
\[
  \mscrF_{k, a}
  = c_{k, a}
    \Omega_{k, a}\Paren*{
      \exp_{\widetilde{\mathit{SL}}(2, \setR)}\Paren*{
        \frac{\pi}{2} \begin{pmatrix} 0 & -1 \\ 1 & 0 \end{pmatrix}
      }
    },
\]
where the phase factor is given by $ c_{k, a} = e^{\frac{i\pi}{2} \frac{2 \dindex{k} + a + N - 2}{a}} $.
It includes several classical transforms as special cases:
\begin{itemize}
  \item The $ (0, 2) $-generalized Fourier transform $ \mscrF_{0, 2} $ is
        the classical Fourier transform.
  \item The $ (0, 1) $-generalized Fourier transform $ \mscrF_{0, 1} $ is
        a Hankel-type transform on the light cone~\cite{MR2134314} (see the next paragraph).
  \item The $ (k, 2) $-generalized Fourier transform $ \mscrF_{k, 2} $ is
        the Dunkl transform~\cite{MR1199124}.
\end{itemize}

Let us explain the meaning of the parameter $ a $ from the perspective of minimal representations.
The classical Fourier transform on the Euclidean space $ \setR^N $ can be
interpreted as the unitary inversion operator arising in the Segal--Shale--Weil representation,
which is a unitary representation of the metaplectic group $ \mathit{Mp}(N, \setR) $
on the Hilbert space $ L^2(\setR^N) $ (see \cite[Chapter~4]{MR983366} for details) and which
decomposes into two irreducible components,
each of which is a minimal representation.
This interpretation naturally leads to analogs to the Fourier transform
in harmonic analysis on manifolds beyond the Euclidean setting.
Motivated by this idea, the \emph{Fourier transform on the light cone} was introduced
by Kobayashi--Mano~\cite{MR2134314,MR2317306,MR2401813,MR2858535}
as the unitary inversion operator arising in the $ L^2 $-model of the minimal representation of the conformal group.

When $ k = 0 $, the representation $ \Omega_{0, a} $ commutes with the natural action of $ O(N) $,
and hence it defines a unitary representation of
$ O(N) \times \widetilde{\mathit{SL}}(2, \setR) $.
This coincides with the restriction of the Segal--Shale--Weil representation to the subgroup when $ (k, a) = (0, 2) $,
and with the restriction of the minimal representation of
$ \widetilde{\mathit{SO}}_0(N + 1, 2) $ to the subgroup when $ (k, a) = (0, 1) $.
This explains why $ \mscrF_{0, 2} $ is the classical Fourier transform and
$ \mscrF_{0, 1} $ is the Fourier transform on the light cone.
The parameter $ a $ therefore provides a continuous interpolation between
the minimal representations of the two simple Lie groups
$ \mathit{Mp}(N, \mathbb{R}) $ and $ \widetilde{\mathit{SO}}_0(N + 1, 2) $
at the level of the subgroup $ O(N) \times \widetilde{\mathit{SL}}(2, \setR) $.

\subsection{Results of the paper}

From the perspective of representation theory,
the Schwartz space $ \mscrS(\setR^N) $ coincides with the space of smooth vectors
for the Segal--Shale--Weil representation of $ \mathit{Mp}(N, \setR) $
(or equivalently, for its restriction $ \Omega_{0, 2} $ to $ \widetilde{\mathit{SL}}(2, \setR) $).
Since the Fourier transform arises as the unitary inversion operator
in the Segal--Shale--Weil representation,
this viewpoint gives a representation-theoretic explanation
why the Fourier transform preserves $ \mscrS(\setR^N) $
(see \cref{thm:properties-of-smooth-vectors}~(2)).

Motivated by this observation, we define the \emph{$ (k, a) $-generalized Schwartz space} as
\[
  \mscrS_{k, a}(\setR^N)
  = L^2(\setR^N, w_{k, a}(x) \,dx)_{\Omega_{k, a}}^\infty,
\]
that is, the space of smooth vectors for $ \Omega_{k, a} $ (\cref{def:schwartz-space}).
Since the $ (k, a) $-generalized Fourier transform $ \mscrF_{k, a} $ is defined as
the unitary inversion operator arising in $ \Omega_{k, a} $,
it preserves $ \mscrS_{k, a}(\setR^N) $ (\cref{thm:schwartz-space-is-ft-stable}).

In this way, we obtain a natural $ \mscrF_{k, a} $-stable subspace $ \mscrS_{k, a}(\setR^N) $,
intrinsic to the unitary representation $ \Omega_{k, a} $,
which generalizes the classical Schwartz space $ \mscrS(\setR^N) = \mscrS_{0, 2}(\setR^N) $.

In this paper, we address the following problem.

\begin{problem*}
  Determine the $ (k, a) $-generalized Schwartz space $ \mscrS_{k, a}(\setR^N) $
  explicitly.
\end{problem*}

In the case $ N = 1 $, we determine $ \mscrS_{k, a}(\setR) $
for general parameters $ k $ and $ a $.

\begin{theoremAlph}[\cref{thm:schwartz-space-of-rank-one}]\label{thmalph:schwartz-space-of-rank-one}
  Let $ k $ be a non-negative multiplicity function and $ a > 0 $
  such that $ \lambda_{k, a, 0} > -1 $, that is, $ \frac{2\dindex{k} - 1}{a} > -1 $.
  Then, we have
  \[
    \mscrS_{k, a}(\setR)
    = \set*{x \mapsto u(\abs{x}^a) + x v(\abs{x}^a)}{\text{%
      $ u $, $ v \in \mscrS(\setRzp) $%
    }}.
  \]
\end{theoremAlph}

For general $ N $, we determine $ \mscrS_{0, a}(\setR^N) $ when $ a $ is rational.
We note that the rationality of $ a $ is equivalent to
the absence of accumulation points in the $ \mathfrak{so}(2) $-spectrum of $ \Omega_{k, a} $~\cite[Corollary~3.22~3)]{MR2956043},
as well as to $ \mscrF_{k, a} $ having finite order~\cite[Corollary~5.2]{MR2956043},
making this case particularly interesting
from both representation-theoretic and harmonic-analytic perspectives.
On the other hand, when $ a $ is irrational,
it was pointed out in Gorbachev--Ivanov--Tikhonov~\cite[Proposition~5.8]{MR4629458} that
$ \mscrF_{k, a} $ ``drastically deforms even very smooth functions''.

\begin{theoremAlph}[\cref{thm:schwartz-space}]
  Let $ a = 2p/q \in \setQp $ with $ p $, $ q \in \setNp $,
  and further assume that $ a > 1 $ when $ N = 1 $.
  Then, we have
  \[
    \mscrS_{0, a}(\setR^N)
    = \sum_{j = 0}^{q - 1} \enorm{x}^{ja} S_{2p}(\setR^N),
  \]
  where
  \[
    S_{2p}(\setR^N)
    = \set*{F \in \mscrS(\setR^N)}{
        \begin{aligned}
          & \text{The formal Taylor series of $ F $ at the origin belongs to} \\
          & \textstyle\prod_{m, n = 0}^{\infty} \enorm{x}^{2np} \mcalH^m(\setR^N)
        \end{aligned}
    }.
  \]
\end{theoremAlph}

So far, we have focused on the space $ \mscrS_{k, a}(\setR^N) $ of smooth vectors
for the representation $ \Omega_{k, a} $ of $ \widetilde{\mathit{SL}}(2, \setR) $.
As explained in \cref{ssec:background}, for the special parameter $ (k, a) = (0, 1) $,
the representation $ \Omega_{0, 1} $ arises as the restriction of
the minimal representation of the larger group $ \widetilde{\mathit{SO}}_0(N + 1, 2) $.
We also determine the space of smooth vectors for this larger symmetry,
and show that it coincides with $ \mscrS_{0, 1}(\setR^N) $.

In what follows, we assume that $ N \geq 2 $.
Let $ \Pi_1 $ denote the $ L^2 $-model for the minimal representation of
the conformal group $ \widetilde{\mathit{SO}}_0(N + 1, 2) $.
We define the \emph{Schwartz space for this minimal representation} as
\[
  \mscrS_{\mathrm{conf}}(\setR^N)
  = L^2(\setR^N, \enorm{x}^{-1} \,dx)_{\Pi_1}^\infty,
\]
that is, the space of smooth vectors for the minimal representation $ \Pi_1 $ (\cref{def:schwartz-space-conformal}).
By the same reasoning as for the $ (k, a) $-generalized Schwartz space,
$ \mscrS_{\mathrm{conf}}(\setR^N) $ is preserved by $ \mscrF_{0, 1} $.

Since $ \Omega_{0, 1} $ coincides with the restriction of $ \Pi_1 $ to a subgroup,
it is clear that $ \mscrS_{\mathrm{conf}}(\setR^N) $ is contained in $ \mscrS_{0, 1}(\setR^N) $.
We prove that these $ \mscrF_{0, 1} $-stable subspaces coincide.

\begin{theoremAlph}[\cref{thm:schwartz-space-conformal}]
  Let $ N \geq 2 $.
  The following spaces coincide.
  \begin{enumalphp}
    \item $ \mscrS_{\mathrm{conf}}(\setR^N) $,
          or the space of smooth vectors for the minimal representation $ \Pi_1 $
    \item $ \mscrS_{0, 1}(\setR^N) $,
          or the space of smooth vectors for $ \Omega_{0, 1} $
    \item the space of smooth vectors for $ \restr{\Omega_{0, 1}}{\widetilde{\mathit{SO}}(2)} $
    \item $ \mscrS(\setR^N) + \enorm{x} \mscrS(\setR^N) $
  \end{enumalphp}
\end{theoremAlph}

Recently, several works related to Schwartz-type spaces
for the $ (k, a) $-generalized Fourier transform have appeared.
Gorbachev--Ivanov--Tikhonov~\cite{MR4629458} investigated
the image $ \mscrF_{k, a}(\mscrS(\setR^N)) $ of the classical Schwartz space under $ \mscrF_{k, a} $.
Ivanov~\cite{MR4684153} defined an $ \mscrF_{k, a} $-stable Schwartz-type space in the case $ N = 1 $,
which, by \cref{thmalph:schwartz-space-of-rank-one}, coincides with $ \mscrS_{k, a}(\setR) $.
Note that this space is defined by an explicit description of the function space,
without reference to the theory of smooth vectors.
There is also a related preprint by Faustino--Negzaoui~\cite{arXiv2507-04064}.
See \cref{ssec:previous-work} for more details.

\subsection{Organization of the paper}

In \cref{sec:preliminaries}, we briefly review basic properties of
the differential-difference operators $ \DiffH{k, a} $, $ \DiffEp{k, a} $ and $ \DiffEm{k, a} $
introduced by Ben~Saïd--Kobayashi--Ørsted.
In \cref{sec:infinitesimal}, after recalling some general definitions and facts,
we consider the infinitesimal representation $ d\Omega_{k, a} $
of the unitary representation $ \Omega_{k, a} $,
and investigate certain related self-adjoint and skew-adjoint operators.
\cref{sec:schwartz-space,sec:conformal-group} constitute the main part of this paper.
In these sections, we determine the $ (k, a) $-generalized Schwartz space and
the Schwartz space for the minimal representation of the conformal group, respectively.
The proofs rely on the results obtained in \cref{sec:infinitesimal}
as well as technical results collected
in \cref{sec:laguerre-polynomial,sec:pg,sec:harmonic-polynomial}.

\subsection{Notation and terminology}
 
\begin{itemize}
  \item $ \setN = \setenum{0, 1, 2, \dots} $.
  \item We write $ \innprod{\blank}{\blank} $ for the Euclidean inner product,
        and $ \enorm{\blank} $ for the Euclidean norm.
  \item $ \Sphere{N - 1} = \set{x \in \setR^N}{\enorm{x} = 1} $.
  \item Function spaces, such as $ C^\infty $ spaces and $ L^2 $ spaces,
        are understood to consist of complex-valued functions.
  \item We write $ \mcalP(\setR^N) $ for the space of polynomials on $ \setR^N $ and
        $ \mcalP^m(\setR^N) $ for its subspace of homogeneous polynomials of degree $ m $.
        We write $ \mcalH^m(\setR^N) $
        for the space of harmonic polynomials of degree $ m $ on $ \setR^N $ and
        $ \mcalH^m(\Sphere{N - 1}) $
        for the space of spherical harmonics of degree $ m $ on $ \Sphere{N - 1} $.
  \item We write $ E_x = \sum_{j = 1}^{N} x_j \Pdif{x_j} $ for the Euler operator
        on $ \setR^N $, and $ E_r = r \Odif{r} $ for the Euler operator on $ \setRp $.
  \item We write the rising and falling factorials as
        \begin{align*}
          x^{\overline{n}} &= x (x + 1) \dotsm (x + n - 1), \\
          x^{\underline{n}} &= x (x - 1) \dotsm (x - n + 1).
        \end{align*}
  \item Let $ \faml{c_\alpha}{\alpha \in \setN^n} $ be a family of complex numbers.
        We say that $ \faml{c_\alpha}{\alpha \in \setN^n} $ is \emph{rapidly decreasing}
        if $ \sup_{\alpha \in \setN^n} \abs{P(\alpha) c_\alpha} < \infty $ for any polynomial $ P $.
        We say that $ \faml{c_\alpha}{\alpha \in \setN^n} $ is of \emph{polynomial growth}
        if $ \abs{c_\alpha} \leq P(\alpha) $ for some polynomial $ P $.
\end{itemize}

\section{Preliminaries}
\label{sec:preliminaries}

In this section, we briefly review basic properties of
the differential-difference operators $ \DiffH{k, a} $, $ \DiffEp{k, a} $ and $ \DiffEm{k, a} $
introduced by Ben~Saïd--Kobayashi--Ørsted
to the extent necessary for later use.
This section contains no new results.

\subsection{Dunkl theory}

Throughout this section, we fix a reduced root system $ \mscrR $ on $ \setR^N $.
That is, we assume that $ \mscrR $ satisfies the following conditions:
\begin{itemize}
  \item $ \mscrR $ is a finite subset of $ \setR^N \setminus \setenum{0} $,
  \item $ \mscrR $ is stable under the orthogonal reflection $ r_\alpha $
        with respect to the hyperplane $ (\setR \alpha)^\perp $ for all
        $ \alpha \in \mscrR $, and
  \item $ \mscrR \cap \setR \alpha = \setenum{\alpha, -\alpha} $ for all
        $ \alpha \in \mscrR $.
\end{itemize}
Note that we do not impose crystallographic conditions on roots and do not
require that $ \mscrR $ spans $ \setR^N $.

The subgroup of $ O(N) $ generated by all the reflections $ r_\alpha $
is called the \emph{reflection group associated with $ \mscrR $}.
We say that a function $ \map{k}{\mscrR}{\setC} $ is a \emph{multiplicity
function} if it is invariant under the natural action of the reflection group.
We usually write $ k_\alpha $ instead of $ k(\alpha) $. We say that
a multiplicity function $ k $ is \emph{non-negative} if $ k_\alpha \geq 0 $
for all $ \alpha \in \mscrR $. The \emph{index} of a multiplicity function
$ k $ is defined as
\[
  \dindex{k}
  = \frac{1}{2} \sum_{\alpha \in \mscrR} k_\alpha
  = \sum_{\alpha \in \mscrR^+} k_\alpha,
\]
where $ \mscrR^+ $ is any positive system of $ \mscrR $.

For details on Dunkl theory, we refer the reader to the original papers
\cite{MR917849,MR951883,MR1145585,MR1199124} and the lecture notes \cite{MR2022853}.

\subsection{The differential-difference operators \texorpdfstring{$ \DiffH{k, a} $}{Hk,a}, \texorpdfstring{$ \DiffEp{k, a} $}{E+k,a} and \texorpdfstring{$ \DiffEm{k, a} $}{E-k,a}}

Let $ k $ be a multiplicity function and $ a \in \setC \setminus \setenum{0} $.
We recall the definition of the differential-difference operators
$ \DiffH{k, a} $, $ \DiffEp{k, a} $ and $ \DiffEm{k, a} $
on $ \setR^N \setminus \setenum{0} $ from \cite[(3.3)]{MR2956043}:
\[
  \DiffH{k, a}
  = \frac{2}{a} E_x + \frac{2 \dindex{k} + a + N - 2}{a}, \qquad
  \DiffEp{k, a}
  = \frac{i}{a} \enorm{x}^a, \qquad
  \DiffEm{k, a}
  = \frac{i}{a} \enorm{x}^{2 - a} \Laplacian_k.
\]
Here, $ \Laplacian_k $ denotes the Dunkl Laplacian
(see \cite{MR917849}, \cite[Definition~1.1]{MR951883} and \cite[(2.2)]{MR2022853}),
which reduces to the classical Laplacian $ \Laplacian $ when $ k = 0 $.

Additionally, for $ m \in \setN $, we consider the following differential
operators on $ \setRp $:
\begin{align*}
  \DiffH{k, a}[m]
  &= \frac{2}{a} E_r + \frac{2 \dindex{k} + a + N - 2}{a}, \\
  \DiffEp{k, a}[m]
  &= \frac{i}{a} r^a, \\
  \DiffEm{k, a}[m]
  &= \frac{i}{a} r^{-a} (E_r - m) (E_r + m + 2\dindex{k} + N - 2).
\end{align*}
These are the radial parts of $ \DiffH{k, a} $, $ \DiffEp{k, a} $ and
$ \DiffEm{k, a} $ respectively in the following sense.

\begin{proposition}\label{thm:radial-parts}
  Let $ k $ be a multiplicity function, $ a \in \setC \setminus \setenum{0} $,
  and $ m \in \setN $.
  For $ p \in \mcalH_k^m(\Sphere{N - 1}) $ and $ f \in C^\infty(\setRp) $, we have
  \begin{align*}
    \DiffH{k, a} (p \otimes f) &= p \otimes \DiffH{k, a}[m] f, \\
    \DiffEp{k, a} (p \otimes f) &= p \otimes \DiffEp{k, a}[m] f, \\
    \DiffEm{k, a} (p \otimes f) &= p \otimes \DiffEm{k, a}[m] f,
  \end{align*}
  where $ p \otimes f $ denotes the function $ r\omega \mapsto p(\omega) f(r) $
  on $ \setR^N \setminus \setenum{0} $.
\end{proposition}

\begin{proof}
  It follows from a direct computation.
  See also \cite[Lemma~3.6]{MR2956043}.
\end{proof}

\begin{proposition}\label{thm:sl2-triple}
  Let $ k $ be a multiplicity function and $ a \in \setC \setminus \setenum{0} $.
  \begin{enumarabicp}
    \item The differential-difference operators
          $ \DiffH{k, a} $, $ \DiffEp{k, a} $ and $ \DiffEm{k, a} $
          form an $ \mathfrak{sl}_2 $-triple. That is,
          \[
            [\DiffH{k, a}, \DiffEp{k, a}] = 2 \DiffEp{k, a}, \qquad
            [\DiffH{k, a}, \DiffEm{k, a}] = -2 \DiffEm{k, a}, \qquad
            [\DiffEp{k, a}, \DiffEm{k, a}] = \DiffH{k, a}.
          \]
    \item For any $ m \in \setN $, the differential operators
          $ \DiffH{k, a}[m] $, $ \DiffEp{k, a}[m] $ and $ \DiffEm{k, a}[m] $
          form an $ \mathfrak{sl}_2 $-triple. That is,
          \[
            [\DiffH{k, a}[m], \DiffEp{k, a}[m]] = 2 \DiffEp{k, a}, \qquad
            [\DiffH{k, a}[m], \DiffEm{k, a}[m]] = -2 \DiffEm{k, a}, \qquad
            [\DiffEp{k, a}[m], \DiffEm{k, a}[m]] = \DiffH{k, a}.
          \]
  \end{enumarabicp}
\end{proposition}

\begin{proof}
  \begin{subproof}{(1)}
    It is \cite[Theorem~3.2]{MR2956043}.
  \end{subproof}

  \begin{subproof}{(2)}
    It follows from (1) and \cref{thm:radial-parts}.
  \end{subproof}
\end{proof}

Let
\[
  \mbfh   = \begin{pmatrix} 1 & 0 \\ 0 & -1 \end{pmatrix}, \qquad
  \mbfe^+ = \begin{pmatrix} 0 & 1 \\ 0 & 0 \end{pmatrix}, \qquad
  \mbfe^- = \begin{pmatrix} 0 & 0 \\ 1 & 0 \end{pmatrix}
\]
be the standard basis of $ \mathfrak{sl}(2, \setR) $.
By \cref{thm:sl2-triple}~(1), the differential-difference operators
$ \DiffH{k, a} $, $ \DiffEp{k, a} $ and $ \DiffEm{k, a} $ define
a Lie algebra representation
\begin{align*}
  & \map{\omega_{k, a}}{\mathfrak{sl}(2, \setR)}{\End_{\setC}(C^\infty(\setR^N \setminus \setenum{0}))}, \\
  & \omega_{k, a}(\mbfh) = \DiffH{k, a}, \qquad
    \omega_{k, a}(\mbfe^+) = \DiffEp{k, a}, \qquad
    \omega_{k, a}(\mbfe^-) = \DiffEm{k, a}.
\end{align*}
Similarly, by \cref{thm:sl2-triple}~(2), the differential operators
$ \DiffH{k, a}[m] $, $ \DiffEp{k, a}[m] $ and $ \DiffEm{k, a}[m] $ define
a Lie algebra representation
\begin{align*}
  & \map{\omega_{k, a}}{\mathfrak{sl}(2, \setR)}{\End_{\setC}(C^\infty(\setR^N \setminus \setenum{0}))}, \\
  & \omega_{k, a, m}(\mbfh) = \DiffH{k, a}[m], \qquad
    \omega_{k, a, m}(\mbfe^+) = \DiffEp{k, a}[m], \qquad
    \omega_{k, a, m}(\mbfe^-) = \DiffEm{k, a}[m].
\end{align*}
We again write $ \omega_{k ,a} $ and $ \omega_{k, a, m} $ for their complexification.

\subsection{The orthonormal basis \texorpdfstring{$ \faml{f_{k, a, m; l}}{l \in \setN} $}{(fk,a,m;l)l∈N}}
\label{ssec:onb}

Following Ben~Saïd--Kobayashi--Ørsted~\cite[(3.11), (3.32)]{MR2956043}, we set
\[
  \lambda_{k, a, m}
  = \frac{2m + 2\dindex{k} + N - 2}{a}
\]
and consider the function on $ \setRp $ defined by
\[
  f_{k, a, m; l}(r)
  = \Paren*{\frac{2^{\lambda_{k, a, m} + 1} \Gamma(l + 1)}{a^{\lambda_{k, a, m}} \Gamma(\lambda_{k, a, m} + l + 1)}}^{1/2}
    r^m
    L_l^{\lambda_{k, a, m}} \Paren*{\frac{2}{a} r^a}
    \exp \Paren*{-\frac{1}{a} r^a}
\]
for $ l \in \setN $.
Here, $ L_l^\lambda $ denotes the Laguerre polynomial
(see \cref{sec:laguerre-polynomial} for the definition).

\begin{proposition}[{\cite[Proposition~3.15]{MR2956043}}]\label{thm:onb}
  Let $ k $ be a non-negative multiplicity function, $ a > 0 $, and $ m \in \setN $
  such that $ \lambda_{k, a, m} = \frac{2m + 2\dindex{k} + N - 2}{a} > -1 $.
  Then, $ \faml{f_{k, a, m; l}}{l \in \setN} $ is
  an orthonormal basis for $ L^2(\setRp, r^{2 \dindex{k} + a + N - 3} \,dr) $.
  \qed
\end{proposition}

Let $ \mbfk $, $ \mbfn^+ $ and $ \mbfn^- $ be the Cayley transforms of
the standard basis $ \mbfh $, $ \mbfe^+ $ and $ \mbfe^- $. That is,
\begin{align*}
  \mbfk
  &= \Ad\Paren*{\frac{1}{\sqrt{2}} \begin{pmatrix} -i & -1 \\ 1 & i \end{pmatrix}} \mbfh
  = \begin{pmatrix} 0 & -i \\ i & 0 \end{pmatrix}, \\
  \mbfn^+
  &= \Ad\Paren*{\frac{1}{\sqrt{2}} \begin{pmatrix} -i & -1 \\ 1 & i \end{pmatrix}} \mbfe^+
  = \frac{1}{2} \begin{pmatrix} i & -1 \\ -1 & -i \end{pmatrix}, \\
  \mbfn^-
  &= \Ad\Paren*{\frac{1}{\sqrt{2}} \begin{pmatrix} -i & -1 \\ 1 & i \end{pmatrix}} \mbfe^-
  = \frac{1}{2} \begin{pmatrix} -i & -1 \\ -1 & i \end{pmatrix}.
\end{align*}

\begin{proposition}\label{thm:action}
  Let $ k $ be a non-negative multiplicity function, $ a > 0 $, and $ m \in \setN $
  such that $ \lambda_{k, a, m} = \frac{2m + 2\dindex{k} + N - 2}{a} > -1 $.
  For $ l \in \setN $, we have
  \begin{align*}
    \omega_{k, a, m}(\mbfk)   f_{k, a, m; l} &= (\lambda_{k, a, m} + 2l + 1) f_{k, a, m; l}, \\
    \omega_{k, a, m}(\mbfn^+) f_{k, a, m; l} &= i \sqrt{(l + 1)(\lambda_{k, a, m} + l + 1)} \, f_{k, a, m; l + 1}, \\
    \omega_{k, a, m}(\mbfn^-) f_{k, a, m; l} &= i \sqrt{l (\lambda_{k, a, m} + l)} \, f_{k, a, m; l - 1},
  \end{align*}
  where we regard $ f_{k, a, m; -1} = 0 $.
\end{proposition}

\begin{proof}
  By \cref{thm:radial-parts}, the assertion is essentially the same as
  \cite[Theorem~3.21]{MR2956043}.
\end{proof}

\begin{remark}
  Since $ k $ is non-negative and $ a > 0 $, the condition
  $ \lambda_{k, a, m} = \frac{2m + 2\dindex{k} + N - 2}{a} > -1 $
  is always satisfied when $ N \geq 2 $.
  In the case $ N = 1 $, the condition $ \lambda_{k, a, m} > -1 $
  is equivalent to $ 2m + 2\dindex{k} + a > 1 $.
\end{remark}

\subsection{The unitary representation \texorpdfstring{$ \Omega_{k, a} $}{Ωk,a}}

We write $ \widetilde{\mathit{SL}}(2, \setR) $
for the universal covering group of $ \mathit{SL}(2, \setR) $.

\begin{theorem}\label{thm:unitary-representation-rad}
  Let $ k $ be a non-negative multiplicity function, $ a > 0 $, and $ m \in \setN $
  such that $ \lambda_{k, a, m} = \frac{2m + 2\dindex{k} + N - 2}{a} > -1 $.
  We set
  \[
    W_{k, a, m}
    = \lspan_{\setC} \set{f_{k, a, m; l}}{l \in \setN}.
  \]
  Then, $ (\omega_{k, a, m}, W_{k, a, m}) $ is
  an $ (\mathfrak{sl}(2, \setR), \mathfrak{so}(2)) $-module.
  Moreover, it lifts to a unique unitary representation $ \Omega_{k, a, m} $
  of $ \widetilde{\mathit{SL}}(2, \setR) $ on the Hilbert space
  $ L^2(\setRp, r^{2\dindex{k} + a + N - 3} \,dr) $,
  which is an infinite-dimensional irreducible unitary representation
  with a lowest weight vector of weight $ \lambda_{k, a, m} + 1 $.
\end{theorem}

Let us consider the weight functions
\[
  w_k(\omega)
  = \prod_{\alpha \in \mscrR} \abs{\innprod{\alpha}{\omega}}^{k_\alpha}
  = \prod_{\alpha \in \mscrR^+} \abs{\innprod{\alpha}{\omega}}^{2k_\alpha}
  \quad (\omega \in \Sphere{N - 1})
\]
and
\[
  w_{k, a}(x)
  = \enorm{x}^{a - 2} \prod_{\alpha \in \mscrR} \abs{\innprod{\alpha}{x}}^{k_\alpha}
  = \enorm{x}^{a - 2} \prod_{\alpha \in \mscrR^+} \abs{\innprod{\alpha}{x}}^{2k_\alpha}
  \quad (x \in \setR^N)
\]
(see \cite[(1.2)]{MR2956043}; $ \vartheta_{k, a} $ in their notation).
By the polar decomposition
$ w_{k, a}(x) \,dx = w_k(\omega) \,d\omega \otimes r^{2 \dindex{k} + a + N - 3} \,dr $ and
the fact that $ L^2(\Sphere{N - 1}, w_k(\omega) \,d\omega) $ decomposes into
the spaces $ \mcalH_k^m(\Sphere{N - 1}) $ of $ k $-spherical harmonics
(see \cite[pp.\,37--39]{MR917849}), we have the orthogonal decomposition
\begin{align*}
  L^2(\setR^N, w_{k, a}(x) \,dx)
  &= L^2(\Sphere{N - 1}, w_k(\omega) \,d\omega) \otimeshat L^2(\setRp, r^{2\dindex{k} + a + N - 3} \,dr) \\
  &= \sumoplus_{m \in \setN} \mcalH_k^m(\Sphere{N - 1}) \otimes L^2(\setRp, r^{2\dindex{k} + a + N - 3} \,dr).
\end{align*}
Note that, when $ k = 0 $, $ w_{k, a}(x) $ reduces to $ w_{0, a}(x) = \enorm{x}^{a - 2} $
and $ \mcalH_k^m(\Sphere{N - 1}) $ reduces to the space $ \mcalH^m(\Sphere{N - 1}) $
of classical spherical harmonics.

\begin{theorem}\label{thm:unitary-representation}
  Let $ k $ be a non-negative multiplicity function and $ a > 0 $
  such that $ \lambda_{k, a, 0} = \frac{2\dindex{k} + N - 2}{a} > -1 $.
  We set
  \[
    W_{k, a}(\setR^N)
    = \bigoplus_{m \in \setN} \mcalH_k^m(\Sphere{N - 1}) \otimes W_{k, a, m}.
  \]
  Then, $ (\omega_{k, a}, W_{k, a}(\setR^N)) $ is
  an $ (\mathfrak{sl}(2, \setR), \mathfrak{so}(2)) $-module.
  Moreover, it lifts to a unique unitary representation
  $ \Omega_{k, a} $ of $ \widetilde{\mathit{SL}}(2, \setR) $
  on the Hilbert space $ L^2(\setR, w_{k, a}(x) \,dx) $, which decomposes as
  \[
    L^2(\setR^N, w_{k, a}(x) \,dx)
    = \sumoplus_{m \in \setN} \mcalH_k^m(\Sphere{N - 1})
      \otimes L^2(\setRp, r^{2\dindex{k} + a + N - 3} \,dr).
  \]
  Here, $ \widetilde{\mathit{SL}}(2, \setR) $ acts trivially on $ \mcalH_k^m(\Sphere{N - 1}) $,
  and acts $ L^2(\setRp, r^{2\dindex{k} + a + N - 3} \,dr) $
  via the representation $ \Omega_{k, a, m} $.
\end{theorem}

\begin{proof}[Proof of \cref{thm:unitary-representation-rad,thm:unitary-representation}]
  \cref{thm:unitary-representation} is nothing but \cite[Theorems~3.28, 3.30, 3.31]{MR2956043},
  in the proof of which \cref{thm:unitary-representation-rad} is essentially proven by using
  \cite[Fact~3.27, 7)]{MR2956043}.
\end{proof}

\begin{remark}\label{rem:unitary-equivalence}
  Let us consider the case $ N = 1 $,
  $ k \neq 0 $ is a non-negative multiplicity function, $ a = 1 $ and $ m = 0 $.
  We set $ \lambda = \lambda_{k, 1, 0} = 2\dindex{k} - 1 $,
  which ranges over $ (-1, \infty) $.
  In this case, the $ \mathfrak{sl}_2 $-triple is
  \[
    \DiffH{k, 1}[0]  = 2E_r + \lambda + 1, \qquad
    \DiffEp{k, 1}[0] = ir, \qquad
    \DiffEm{k, 1}[0] = i r^{-1} E_r (E_r + \lambda)
  \]
  and the function $ f_{k, 1, 0; l} $ defined in \cref{ssec:onb} is
  \[
    f_{k, 1, 0; l}(r)
    = \Paren*{\frac{2^{\lambda + 1} \Gamma(l + 1)}{\Gamma(\lambda + l + 1)}}^{1/2}
      L_l^\lambda(2r) e^{-r}.
  \]
  Moreover, \cref{thm:unitary-representation-rad} states that
  $ \omega_{k, 1, 0} $ defines an $ (\mathfrak{sl}(2, \setR), \mathfrak{so}(2)) $-module and
  lifts to a unique unitary representation $ \Omega_{k, 1, 0} $
  on the Hilbert space $ L^2(\setRp, r^\lambda \,dr) $.
  In this remark, for simplicity,
  let us write $ \DiffH{\lambda} $, $ \DiffEp{\lambda} $, $ \DiffEm{\lambda} $,
  $ f_{\lambda; l} $, $ \omega_\lambda $ and $ \Omega_\lambda $
  instead of $ \DiffH{k, 1}[0] $, $ \DiffEp{k, 1}[0] $, $ \DiffEm{k, 1}[0] $,
  $ f_{k, 1, 0; l} $, $ \omega_{k, 1, 0} $ and $ \Omega_{k, 1, 0} $, respectively.

  Now let $ N \in \setNp $, $ k $ be a non-negative multiplicity function,
  $ a > 0 $ and $ m \in \setN $
  such that $ \lambda_{k, a, m} = \frac{2m + 2\dindex{k} + N - 2}{a} > -1 $.
  Then, the unitary operator
  \begin{align*}
    & \map{\Phi}{L^2(\setRp, r^{2\dindex{k} + a + N - 3} \,dr)}{L^2(\setRp, r^{\lambda_{k, a, m}} \,dr)}, \\
    & \Phi f(r) = a^{\lambda_{k, a, m}/2} r^{-m/a} f(a^{1/a} r^{1/a})
  \end{align*}
  transfers $ \DiffH{k, a}[m] $, $ \DiffEp{k, a}[m] $ and $ \DiffEm{k, a}[m] $
  to $ \DiffH{\lambda} $, $ \DiffEp{\lambda} $ and $ \DiffEm{\lambda} $, respectively,
  and maps $ f_{k, a, m; l} $ to $ f_{\lambda_{k, a, m}; l} $.
  Therefore, by \cref{thm:action},
  $ \Phi $ gives an $ (\mathfrak{sl}(2, \setR), \mathfrak{so}(2)) $-isomorphism
  from $ \omega_{k, a, m} $ onto $ \omega_{\lambda_{k, a, m}} $
  and a unitary equivalence
  from $ \Omega_{k, a, m} $ onto $ \Omega_{\lambda_{k, a, m}} $.
\end{remark}

\begin{remark}\label{rem:kostants-model}
  $ L^2 $-models of the (relative) discrete series representations of
  $ \mathit{SL}(2, \setR) $ or its universal covering group
  $ \widetilde{\mathit{SL}}(2, \setR) $ have appeared in various contexts
  in the works of Faraut--Korányi~\cite{MR1446489}, Kostant~\cite{MR1755901},
  Vershik--Graev~\cite{MR2328257}, Hilgert--Kobayashi--Möllers~\cite{MR3201818}, and others.
  Among these, Kostant considered an $ \mathfrak{sl}_2 $-triple
  of differential operators
  \[
    2r \Odif{r} + 1, \quad
    ir, \quad
    i \Paren*{r \Paren*{\Odif{r}}^2 + \Odif{r} - \frac{\lambda^2}{4r}}
  \]
  on $ \setRp $, and proved that, for $ \lambda > -1 $,
  it lifts to a unique unitary representation $ \pi_\lambda $
  of $ \widetilde{\mathit{SL}}(2, \setR) $ on the Hilbert space $ L^2(\setRp, dr) $,
  which is an infinite-dimensional irreducible unitary representation
  with a lowest weight vector of weight $ \lambda + 1 $.

  In the notation of \cref{rem:unitary-equivalence}, the unitary operator
  \begin{align*}
    & \map{\Psi}{L^2(\setRp, r^\lambda \,dr)}{L^2(\setRp, dr)}, \\
    & \Psi f(r) = r^{\lambda/2} f(r)
  \end{align*}
  gives a unitary equivalence from $ \Omega_\lambda $ onto $ \pi_\lambda $.
  See also \cite[Remark~3.32]{MR2956043} and \cref{rem:smooth-vectors-for-kostants-model}.
\end{remark}

\section{The infinitesimal representation \texorpdfstring{$ d\Omega_{k, a} $}{dΩk,a}}
\label{sec:infinitesimal}

\subsection{Infinitesimal representation of unitary representations}

In this subsection, we recall some general definitions and facts about
infinitesimal representation of unitary representations and
smooth vectors for unitary representations.

We fix some notation and terminology related to operators on a Hilbert space.
The domain of a linear operator $ T $ on a Hilbert space is denoted by $ \Dom(T) $.
For linear operators $ S $ and $ T $ on a Hilbert space,
their composition $ TS $ is understood to be defined on $ \Dom(S) \cap S^{-1}(\Dom(T)) $.
A densely defined operator $ T $ on a Hilbert space is called
\emph{symmetric} (resp.\ \emph{skew-symmetric}) if its adjoint $ T^* $
is an extension of $ T $ (resp.\ $ iT $).
A densely defined operator $ T $ on a Hilbert space is called
\emph{self-adjoint} (resp.\ \emph{skew-adjoint}) if its adjoint $ T^* $
is equal to $ T $ (resp.\ $ iT $);
that is, they have the same domain and coincide on it.

\begin{definition}\label{def:infinitesimal-representation}
  Let $ G $ be a Lie group with Lie algebra $ \mfrakg $
  and $ (\pi, \mcalH) $ be a unitary representation of $ G $.
  For $ X \in \mfrakg $, we define the operator $ d\pi(X) $ on $ \mcalH $ by
  \begin{align*}
    \Dom(d\pi(X)) &= \set*{v \in \mcalH}{\text{$ \lim_{t \to 0} \frac{\pi(e^{tX}) v - v}{t} $ exists in $ \mcalH $}}, \\
    d\pi(X) v &= \lim_{t \to 0} \frac{\pi(e^{tX}) v - v}{t},
  \end{align*}
  which is a skew-adjoint operator on $ \mcalH $ (see \cite[Proposition~6.1]{MR2953553}).
  For $ X \in \mfrakg $, we also write $ d\pi(iX) = i\, d\pi(X) $,
  which is a self-adjoint operator on $ \mcalH $.
  We call $ d\pi $ the \emph{infinitesimal representation} of $ \pi $.
\end{definition}

\begin{definition}\label{def:smooth-vector}
  Let $ G $ be a Lie group and $ (\pi, \mcalH) $ be a unitary representation of $ G $.
  A vector $ v \in \mcalH $ is said to be \emph{smooth for $ \pi $} if the map
  $ g \mapsto \pi(g) v $ from $ G $ into $ \mcalH $ is smooth.
  We write $ \mcalH^\infty $,
  or $ \mcalH_\pi^\infty $ when we want to make the representation $ \pi $ explicit,
  for the space of smooth vectors for $ \pi $.

  We equip $ \mcalH^\infty $ with the Fréchet topology (see \cite[1.6.3]{MR929683}).
\end{definition}

We collect some basic properties
about the infinitesimal representation and the space of smooth vectors.

\begin{proposition}\label{thm:properties-of-smooth-vectors}
  Let $ G $ be a Lie group with Lie algebra $ \mfrakg $
  and $ (\pi, \mcalH) $ be a unitary representation of $ G $.
  \begin{enumarabicp}
    \item $ \mcalH^\infty $ is a dense subspace of $ \mcalH $.
    \item $ \mcalH^\infty $ is $ \pi(G) $-stable and
          $ \pi $ defines a smooth representation of $ G $ on $ \mcalH^\infty $.
    \item For any $ X \in \mfrakg $,
          $ \mcalH^\infty $ is contained in $ \Dom(d\pi(X)) $ and $ d\pi(X) $-stable.
    \item For any $ X \in \mfrakg $,
          $ \mcalH^\infty $ is a core of the skew-adjoint operator $ d\pi(X) $;
          that is, $ d\pi(X) = \closure{\restr{d\pi(X)}{\mcalH^\infty}} $.
  \end{enumarabicp}
\end{proposition}

\begin{proof}
  \begin{subproof}{(1), (2), (3)}
    See \cite[1.6.2--1.6.4]{MR929683}.
  \end{subproof}
  
  \begin{subproof}{(4)}
    It follows from (2) and \cite[Proposition~6.3]{MR2953553}.
  \end{subproof}
\end{proof}

\begin{proposition}\label{thm:infinitesimal-lie-algebra-representation}
  Let $ G $ be a Lie group with Lie algebra $ \mfrakg $
  and $ (\pi, \mcalH) $ be a unitary representation of $ G $.
  Then, for $ X $, $ Y \in \mfrakg $, we have
  \[
    d\pi([X, Y])
    = \closure{d\pi(X) d\pi(Y) - d\pi(Y) d\pi(X)}
  \]
  where the operator $ d\pi(X) d\pi(Y) - d\pi(Y) d\pi(X) $ is defined
  on $ \Dom(d\pi(X) d\pi(Y)) \cap \Dom(d\pi(Y) d\pi(X)) $.
\end{proposition}

\begin{proof}
  The operator $ d\pi(X) d\pi(Y) - d\pi(Y) d\pi(X) $ is skew-symmetric,
  and its domain contains $ \mcalH^\infty $ (\cref{thm:properties-of-smooth-vectors}~(3)),
  which is a core of $ d\pi([X, Y]) $ (\cref{thm:properties-of-smooth-vectors}~(4)).
  Hence, $ \closure{d\pi(X) d\pi(Y) - d\pi(Y) d\pi(X)} $ is
  a skew-symmetric extension of the skew-adjoint operator $ d\pi([X, Y]) $.
  Now the assertion follows because a skew-adjoint operator is maximally skew-symmetric.
\end{proof}

\begin{proposition}\label{thm:condition-to-be-a-smooth-vector}
  Let $ G $ be a Lie group with Lie algebra $ \mfrakg $
  and $ (\pi, \mcalH) $ be a unitary representation of $ G $.
  Then, for $ v \in \mcalH $, the following are equivalent.
  \begin{enumalphp}
    \item $ v $ is smooth for $ \pi $.
    \item For any finite sequence $ X_1 $, $ \dots $, $ X_n \in \mfrakg $,
          $ v \in \Dom(d\pi(X_1) \dotsm d\pi(X_n)) $.
    \item There exists a spanning subset $ S $ of a vector space $ \mfrakg $ such that,
          for any finite sequence $ X_1 $, $ \dots $, $ X_n \in S $,
          $ v \in \Dom(d\pi(X_1) \dotsm d\pi(X_n)) $.
    \item There exists a generating subset $ S $ of a Lie algebra $ \mfrakg $ such that,
          for any finite sequence $ X_1 $, $ \dots $, $ X_n \in S $,
          $ v \in \Dom(d\pi(X_1) \dotsm d\pi(X_n)) $.
  \end{enumalphp}
\end{proposition}

\begin{proof}
  \begin{subproof}{$ \text{(a)} \Longleftrightarrow \text{(b)} \Longleftrightarrow \text{(c)} $}
    It is clear from the definition of smooth vectors.
  \end{subproof}

  \begin{subproof}{$ \text{(c)} \Longleftrightarrow \text{(d)} $}
    It follows from \cref{thm:infinitesimal-lie-algebra-representation}.
  \end{subproof}
\end{proof}

\begin{fact}[{\cite[Theorem~3.46~(2)]{MR2279709}}]\label{thm:density-of-space-of-k-finite-vectors}
  Let $ K $ be a compact Hausdorff group and
  $ (\pi, V) $ be a continuous representation of $ K $ on a complete Hausdorff locally convex.
  Then, the space $ V_K $ of $ K $-finite vectors is dense in $ V $.
\end{fact}

\begin{proposition}\label{thm:smooth-and-k-finite-vectors}
  Let $ \widetilde{G} $ be a Lie group with Lie algebra $ \mfrakg $,
  $ \widetilde{K} $ be a Lie subgroup of $ G $,
  and $ (\pi, \mcalH) $ be a unitary representation of $ G $.
  Assume that there exists a character $ \map{\chi}{\widetilde{K}}{\setC} $
  such that the representation $ \restr{\pi}{\widetilde{K}} \otimes \chi $ descends to
  a compact quotient group $ K $ of $ \widetilde{K} $.
  \begin{enumarabicp}
    \item The space $ \mcalH_{\widetilde{K}} \cap \mcalH^\infty $ of
          $ \widetilde{K} $-finite smooth vectors is dense in $ \mcalH^\infty $
          with respect to its Fréchet topology.
    \item For $ X \in \mfrakg $,
          $ \mcalH_{\widetilde{K}} \cap \mcalH^\infty $ is
          a core of the skew-adjoint operator $ d\pi(X) $;
          that is, $ d\pi(X) = \closure{\restr{d\pi(X)}{\mcalH_{\widetilde{K}} \cap \mcalH^\infty}} $.
  \end{enumarabicp}
\end{proposition}

\begin{proof}
  \begin{subproof}{(1)}
    Let $ \pi_1 $ denote the representation of $ K $ on $ \mcalH $ defined by
    $ \restr{\pi}{\widetilde{K}} \otimes \chi $.
    It is clear that the space $ \mcalH_K $ of $ K $-finite vectors for $ \pi_1 $
    coincides with the space $ \mcalH_{\widetilde{K}} $ of $ \widetilde{K} $-finite vectors for $ \pi $.
    The assertion follows since $ \mcalH_K \cap \mcalH^\infty $ is dense in $ \mcalH^\infty $
    by \cref{thm:density-of-space-of-k-finite-vectors}.
  \end{subproof}

  \begin{subproof}{(2)}
    The assertion holds
    since $ \mcalH_{\widetilde{K}} \cap \mcalH^\infty $ is dense in $ \mcalH^\infty $ by (1),
    which is in turn a core of $ d\pi(X) $ by \cref{thm:properties-of-smooth-vectors}~(4).
  \end{subproof}
\end{proof}

\subsection{The infinitesimal representation \texorpdfstring{$ d\Omega_{k, a} $}{dΩk,a}}

Recall from \cref{thm:unitary-representation} that
$ \Omega_{k, a} $ is the unique unitary representation
whose Harish-Chandra module is $ (\omega_{k, a}, W_{k, a}(\setR^N)) $.
In particular, $ W_{k, a}(\setR^N) $ is the space of
$ \widetilde{\mathit{SO}}(2) $-finite vectors for $ \Omega_{k, a} $
(here, $ \widetilde{\mathit{SO}}(2) $ denotes the universal covering group of $ \mathit{SO}(2) $)
and the infinitesimal representation $ d\Omega_{k, a} $ satisfies
\[
  \restr{d\Omega_{k, a}(X)}{W_{k, a}(\setR^N)}
  = \restr{\omega_{k, a}(X)}{W_{k, a}(\setR^N)}
\]
for all $ X \in \mathfrak{sl}(2, \setR) $.
Similarly (see \cref{thm:unitary-representation-rad}),
$ W_{k, a, m} $ is the space of
$ \widetilde{\mathit{SO}}(2) $-finite vectors for $ \Omega_{k, a, m} $
and the infinitesimal representation $ d\Omega_{k, a, m} $ satisfies
\[
  \restr{d\Omega_{k, a, m}(X)}{W_{k, a, m}}
  = \restr{\omega_{k, a, m}(X)}{W_{k, a, m}}
\]
for all $ X \in \mathfrak{sl}(2, \setR) $.

We emphasize the distinction between $ \omega_{k, a} $ and $ d\Omega_{k, a} $;
for $ X \in \mathfrak{sl}(2, \setR) $,
while $ \omega_{k, a}(X) $ is a differential-difference operator on $ \setR^N \setminus \setenum{0} $,
$ d\Omega_{k, a}(X) $ is a skew-adjoint operator on $ L^2(\setR^N, w_{k, a}(x) \,dx) $
defined by \cref{def:infinitesimal-representation}.
The same remark applies to $ \omega_{k, a, m} $ and $ d\Omega_{k, a, m} $.

\begin{proposition}\label{thm:infinitesimal-representation-rad}
  Let $ k $ be a non-negative multiplicity function, $ a > 0 $, and $ m \in \setN $
  such that $ \lambda_{k, a, m} = \frac{2m + 2\dindex{k} + N - 2}{a} > -1 $.
  Then, for $ X \in \mathfrak{sl}(2, \setR) $, we have
  \[
    d\Omega_{k, a, m}(X)
    = \closure{\restr{\omega_{k, a, m}(X)}{W_{k, a, m}}}.
  \]
\end{proposition}

\begin{proof}
  By \cref{thm:unitary-representation-rad}, we have
  \[
    \Omega_{k, a, m}\Paren*{\exp_{\widetilde{\mathit{SO}}(2)}\Paren*{\begin{pmatrix} 0 & t \\ -t & 0 \end{pmatrix}}} f_{k, a, m; l}
    = e^{it (\lambda_{k, a, m} + 2l + 1)} f_{k, a, m; l}.
  \]
  Hence, the representation $ \restr{\Omega_{k, a, m}}{\widetilde{\mathit{SO}}(2)} \otimes \chi $
  descends to $ \mathit{SO}(2) $,
  where $ \chi $ denotes the character
  $ \exp_{\widetilde{\mathit{SO}}(2)}(\begin{psmallmatrix} 0 & t \\ -t & 0 \end{psmallmatrix})
  \mapsto e^{-it\lambda_{k, a, m}} $.
  Therefore, $ W_{k, a, m} $ is a core of $ d\Omega_{k, a, m}(X) $
  by \cref{thm:smooth-and-k-finite-vectors}~(2),
  and hence
  \[
    d\Omega_{k, a, m}(X)
    = \closure{\restr{d\Omega_{k, a, m}(X)}{W_{k, a, m}}}
    = \closure{\restr{\omega_{k, a, m}(X)}{W_{k, a, m}}}.
      \qedhere
  \]
\end{proof}

\begin{proposition}\label{thm:infinitesimal-representation}
  Let $ k $ be a non-negative multiplicity function and $ a > 0 $
  such that $ \lambda_{k, a, 0} = \frac{2\dindex{k} + N - 2}{a} > -1 $.
  Then, for $ X \in \mathfrak{sl}(2, \setR) $, we have
  \[
    d\Omega_{k, a}(X)
    = \closure{\restr{\omega_{k, a}(X)}{W_{k, a}(\setR^N)}}.
  \]
\end{proposition}

\begin{proof}
  Since $ \Omega_{k, a} $ decomposes as
  $ \sum_{m \in \setN}^\oplus 1_{\mcalH_k^m(\Sphere{N - 1})} \otimes \Omega_{k, a, m} $
  (\cref{thm:unitary-representation}),
  $ d\Omega_{k, a}(X) $ coincides with the direct sum of
  $ \id_{\mcalH_k^m(\Sphere{N - 1})} \otimes d\Omega_{k, a, m}(X) $.
  Therefore, $ W_{k, a}(\setR^N) = \bigoplus_{m \in \setN} \mcalH_k^m(\Sphere{N - 1}) \otimes W_{k, a, m} $ is
  a core of $ d\Omega_{k, a}(X) $ by \cref{thm:infinitesimal-representation-rad},
  and hence
  \[
    d\Omega_{k, a}(X)
    = \closure{\restr{d\Omega_{k, a}(X)}{W_{k, a}(\setR^N)}}
    = \closure{\restr{\omega_{k, a}(X)}{W_{k, a}(\setR^N)}}.
      \qedhere
  \]
\end{proof}

The following lemmas are used throughout the rest of this section.

\begin{lemma}\label{thm:criterion-rad}
  Let $ k $ be a non-negative multiplicity function, $ a > 0 $, and $ m \in \setN $
  such that $ \lambda_{k, a, m} = \frac{2m + 2\dindex{k} + N - 2}{a} > -1 $.
  Let $ X \in \mathfrak{sl}(2, \setR) $.
  Assume that $ \mcalD \subseteq L^2(\setRp, r^{2\dindex{k} + a + N - 3} \,dr) $ is a subspace
  such that
  \begin{itemize}
    \item $ W_{k, a, m} \subseteq \mcalD
          \subseteq L^2(\setRp, r^{2\dindex{k} + a + N - 3} \,dr) \cap C^\infty(\setRp) $,
          and
    \item $ \omega_{k, a, m}(X)(\mcalD) \subseteq L^2(\setRp, r^{2\dindex{k} + a + N - 3} \,dr) $ and
          $ \restr{\omega_{k, a, m}(X)}{\mcalD} $ is a skew-symmetric operator
          on $ L^2(\setRp, r^{2\dindex{k} + a + N - 3} \,dr) $.
  \end{itemize}
  Then, we have
  \[
    d\Omega_{k, a, m}(X)
    = \closure{\restr{\omega_{k, a, m}(X)}{\mcalD}}.
  \]
\end{lemma}

\begin{lemma}\label{thm:criterion}
  Let $ k $ be a non-negative multiplicity function and $ a > 0 $
  such that $ \lambda_{k, a, 0} = \frac{2\dindex{k} + N - 2}{a} > -1 $.
  Let $ X \in \mathfrak{sl}(2, \setR) $.
  Assume that $ \mcalD \subseteq L^2(\setR^N, w_{k, a}(x) \,dx) $ is a subspace
  such that
  \begin{itemize}
    \item $ W_{k, a}(\setR^N) \subseteq \mcalD
          \subseteq L^2(\setR^N, w_{k, a}(x) \,dx) \cap C^\infty(\setR^N \setminus \setenum{0}) $,
          and
    \item $ \omega_{k, a}(X)(\mcalD) \subseteq L^2(\setR^N, w_{k, a}(x) \,dx) $ and
          $ \restr{\omega_{k, a}(X)}{\mcalD} $ is a skew-symmetric operator
          on $ L^2(\setR^N, w_{k, a}(x) \,dx) $.
  \end{itemize}
  Then, we have
  \[
    d\Omega_{k, a}(X)
    = \closure{\restr{\omega_{k, a}(X)}{\mcalD}}.
  \]
\end{lemma}

Since \cref{thm:criterion-rad} is proved in the same way as \cref{thm:criterion},
we focus on the latter.

\begin{proof}[Proof of \cref{thm:criterion}]
  Since $ W_{k, a}(\setR^N) $ is a core of $ d\Omega_{k, a}(X) $ and
  $ \restr{d\Omega_{k, a}(X)}{W_{k, a}(\setR^N)} = \restr{\omega_{k, a}(X)}{W_{k, a}(\setR^N)} $
  (\cref{thm:infinitesimal-representation}),
  $ \closure{\restr{\omega_{k, a}(X)}{\mcalD}} $ is a skew-symmetric extension of $ d\Omega_{k, a}(X) $.
  Now the assertion follows because a skew-adjoint operator is maximally skew-symmetric.
\end{proof}


\subsection{The self-adjoint operator \texorpdfstring{$ d\Omega_{k, a}(\mbfk) $}{dΩk,a(k)}}

Recall from \cref{ssec:onb} that
$ \mbfk = \begin{psmallmatrix} 0 & -i \\ i & 0 \end{psmallmatrix} \in i\, \mathfrak{so}(2) $.
In this subsection, we find the self-adjoint operators
$ d\Omega_{k, a, m}(\mbfk) $ and $ d\Omega_{k, a}(\mbfk) $ explicitly.
These results will be used to describe the $ (k, a) $-generalized Schwartz space
in terms of infinite sum
(see \cref{thm:schwartz-space-as-infinite-sum-rad,thm:schwartz-space-as-infinite-sum}).

\begin{proposition}\label{thm:operator-k-rad}
  Let $ k $ be a non-negative multiplicity function, $ a > 0 $, and $ m \in \setN $
  such that $ \lambda_{k, a, m} = \frac{2m + 2\dindex{k} + N - 2}{a} > -1 $.
  Then, we have
  \[
    \Dom(d\Omega_{k, a, m}(\mbfk))
    = \set*{\sum_{l \in \setN} c_l f_{k, a, m; l}}{\text{%
      $ c_l \in \setC $,
      $ \sum_{l \in \setN} (\lambda_{k, a, m} + 2l + 1)^2 \abs{c_l}^2 < \infty $%
    }}
  \]
  and
  \[
    d\Omega_{k, a, m}(\mbfk) \Paren*{\sum_{l \in \setN} c_l f_{k, a, m; l}}
    = \sum_{l \in \setN} (\lambda_{k, a, m} + 2l + 1) c_l f_{k, a, m; l},
  \]
  where the infinite sums are understood to converge in $ L^2(\setRp, r^{2\dindex{k} + a + N - 3} \,dr) $.
\end{proposition}

\begin{proof}
  The diagonal operator with entries $ \lambda_{k, a, m} + 2l + 1 $
  with respect to the orthonormal basis $ \faml{f_{k, a, m; l}}{l \in \setN} $
  (\cref{thm:onb}),
  defined on the domain given in the assertion, is self-adjoint
  by general theory about operators (see \cite[Example~3.8]{MR2953553}).
  Now the assertion follows from \cref{thm:criterion-rad}.
\end{proof}

\begin{proposition}\label{thm:operator-k}
  Let $ k $ be a non-negative multiplicity function and $ a > 0 $
  such that $ \lambda_{k, a, 0} = \frac{2\dindex{k} + N - 2}{a} > -1 $.
  Then, we have
  \begin{align*}
    & \Dom(d\Omega_{k, a}(\mbfk)) \\
    &= \set*{\sum_{m, l \in \setN} p_{m, l} \otimes f_{k, a, m; l}}{
      \begin{aligned}
        & p_{m, l} \in \mcalH_k^m(\Sphere{N - 1}), \\
        & \textstyle\sum_{m, l \in \setN} (\lambda_{k, a, m} + 2l + 1)^2 \norm{p_{m, l}}^2_{L^2(\setRp, r^{2\dindex{k} + a + N - 3} \,dr)} < \infty
      \end{aligned}
    }
  \end{align*}
  and
  \[
    d\Omega_{k, a}(\mbfk) \Paren*{\sum_{m, l \in \setN} p_{m, l} \otimes f_{k, a, m; l}}
    = \sum_{m, l \in \setN} (\lambda_{k, a, m} + 2l + 1) (p_{m, l} \otimes f_{k, a, m; l}),
  \]
  where the infinite sums are understood to converge in $ L^2(\setR^N, w_{k, a}(x) \,dx) $.
\end{proposition}

\begin{proof}
  Since $ \Omega_{k, a} $ decomposes as
  $ \sum_{m \in \setN}^\oplus 1_{\mcalH_k^m(\Sphere{N - 1})} \otimes \Omega_{k, a, m} $
  (\cref{thm:unitary-representation}),
  $ d\Omega_{k, a}(\mbfk) $ coincides with the direct sum of
  $ \id_{\mcalH_k^m(\Sphere{N - 1})} \otimes d\Omega_{k, a, m}(\mbfk) $.
  That is, we have
  \begin{align*}
    & \Dom(d\Omega_{k, a}(\mbfk)) \\
    &= \set*{\sum_{m \in \setN} F_m}{
      \begin{aligned}
        & F_m \in \mcalH_k^m(\Sphere{N - 1}) \otimes L^2(\setRp, r^{2\dindex{k} + a + N - 3} \,dr), \\
        & \text{$ \textstyle\sum_{m \in \setN} (\id_{\mcalH_k^m(\Sphere{N - 1})} \otimes d\Omega_{k, a, m}(\mbfk)) F_m $ exists in $ L^2(\setR^N, w_{k, a}(x) \,dx) $}
      \end{aligned}
    }
  \end{align*}
  and
  \[
    d\Omega_{k, a}(\mbfk) \Paren*{\sum_{m \in \setN} F_m}
    = \sum_{m \in \setN} (\id_{\mcalH_k^m(\Sphere{N - 1})} \otimes d\Omega_{k, a, m}) F_m.
  \]
  Therefore, the assertion follows from \cref{thm:operator-k-rad}.
\end{proof}

\subsection{The skew-adjoint operator \texorpdfstring{$ d\Omega_{k, a}(\mbfe^+) $}{dΩk,a(e+)}}

Recall from \cref{ssec:onb} that $ (\mbfh, \mbfe^+, \mbfe^-) $ denotes
the standard basis of $ \mathfrak{sl}(2, \setR) $.
In this and the next subsection, we investigate the skew-adjoint operators
$ d\Omega_{k, a, m}(\mbfe^+) $, $ d\Omega_{k, a}(\mbfe^+) $,
$ d\Omega_{k, a, m}(\mbfe^-) $ and $ d\Omega_{k, a}(\mbfe^-) $.
These results will be used in the proof of \cref{thm:schwartz-space-rad,thm:schwartz-space}.

\begin{proposition}\label{thm:operator-ep-rad}
  Let $ k $ be a non-negative multiplicity function, $ a > 0 $, and $ m \in \setN $
  such that $ \lambda_{k, a, m} = \frac{2m + 2\dindex{k} + N - 2}{a} > -1 $.
  Then, we have
  \[
    \Dom(d\Omega_{k, a, m}(\mbfe^+))
    = \set{f \in L^2(\setRp, r^{2\dindex{k} + a + N - 3} \,dr)}{r^a f \in L^2(\setRp, r^{2\dindex{k} + a + N - 3} \,dr)}
  \]
  and
  \[
    d\Omega_{k, a, m}(\mbfe^+) f(r)
    = \DiffEp{k, a}[m] f(r)
    = \frac{i}{a} r^a f(r).
  \]
\end{proposition}

\begin{proposition}\label{thm:operator-ep}
  Let $ k $ be a non-negative multiplicity function and $ a > 0 $
  such that $ \lambda_{k, a, 0} = \frac{2\dindex{k} + N - 2}{a} > -1 $.
  Then, we have
  \[
    \Dom(d\Omega_{k, a}(\mbfe^+))
    = \set{F \in L^2(\setR^N, w_{k, a}(x) \,dx)}{\enorm{x}^a F \in L^2(\setR^N, w_{k, a}(x) \,dx)}
  \]
  and
  \[
    d\Omega_{k, a}(\mbfe^+) F(x)
    = \DiffEp{k, a} F(x)
    = \frac{i}{a} \enorm{x}^a F(x).
  \]
\end{proposition}

Since \cref{thm:operator-ep-rad} is proved in the same way as \cref{thm:operator-ep},
we focus on the latter.

\begin{proof}[Proof of \cref{thm:operator-ep}]
  The multiplication operator $ (i/a) \enorm{x}^a $,
  defined on the domain given in the assertion, is skew-adjoint
  by general theory about operators (see \cite[Example~3.8]{MR2953553}).
  Now the assertion follows from \cref{thm:criterion}.
\end{proof}

\subsection{The skew-adjoint operator \texorpdfstring{$ d\Omega_{k, a}(\mbfe^-) $}{dΩk,a(e-)}}

\begin{proposition}\label{thm:operator-em-rad}
  Let $ k $ be a non-negative multiplicity function, $ a > 0 $, and $ m \in \setN $
  such that $ \lambda_{k, a, m} = \frac{2m + 2\dindex{k} + N - 2}{a} > -1 $.
  Then, we have
  \[
    d\Omega_{k, a, m}(\mbfe^-)
    = \closure{\restr{\DiffEm{k, a}[m]}{\mcalD_{k, a, m}}}.
  \]
  Here, $ \mcalD_{k, a, m} $ consists of functions $ f $ satisfying the following conditions.
  \begin{enumromanp}
    \item $ f \in L^2(\setRp, r^{2\dindex{k} + a + N - 3} \,dr) \cap C^\infty(\setRp) $.
    \item $ \DiffEm{k, a}[m] f \in L^2(\setRp, r^{2\dindex{k} + a + N - 3} \,dr) $.
    \item \label{item:operator-em-rad-at-infinity}
          $ f $ is rapidly decreasing at infinity.
    \item \label{item:operator-em-rad-at-origin}
          $ f(r) = O(r^m) $ and $ (E_r - m) f(r) = O(r^{m + a}) $ as $ r \to 0 $.
  \end{enumromanp}
\end{proposition}

\begin{proof}
  Since the unitary equivalence in \cref{rem:unitary-equivalence} maps
  $ \mcalD_{k, a, m} $ (which implicitly depends on $ N $) onto
  $ \mcalD_{(\lambda_{k, a, m} + 1)/2, 1, 0} $ (which corresponds to the case $ N = 1 $),
  it suffices to consider the case $ N = 1 $,
  $ k \neq 0 $ is a non-negative multiplicity function, $ a = 1 $ and $ m = 0 $.
  We use the notation in \cref{rem:unitary-equivalence};
  that is, we set $ \lambda = \lambda_{k, 1, 0} = 2\dindex{k} - 1 > -1 $ and
  write $ \DiffEm{\lambda} $ and $ \Omega_\lambda $
  instead of $ \DiffEm{k, 1}[0] $ and $ \Omega_{k, 1, 0} $, respectively.
  In addition, we write $ W_\lambda $ instead of $ W_{k, 1, 0} $ and
  $ \mcalD_\lambda $ instead of $ \mcalD_{k, 1, 0} $.
  We need to prove that
  \[
    d\Omega_\lambda(\mbfe^-)
    = \closure{\restr{\DiffEm{\lambda}}{\mcalD_\lambda}}.
  \]
  Notice that
  \[
    W_\lambda
    \subseteq \mcalD_\lambda
    \subseteq L^2(\setRp, r^\lambda \,dr) \cap C^\infty(\setRp)
  \]
  and $ \mcalD_\lambda $ is stable under $ \DiffEm{\lambda} = ir^{-1} E_r (E_r + \lambda) $.
  Hence, by \cref{thm:criterion-rad}, it suffices to show that
  $ \restr{\DiffEm{\lambda}}{\mcalD_\lambda} $ is
  a skew-symmetric operator on $ L^2(\setRp, r^\lambda \,dr) $.

  Let $ f $, $ g \in \mcalD_\lambda $.
  By integrating by parts twice, we have
  \begin{align*}
    & \int_{\epsilon}^{R} r^{-1} E_r (E_r + \lambda) f(r) \cdot \conj{g(r)} \cdot r^\lambda \,dr \\
    &= \Brack[\bigg]{(f'(r) \conj{g(r)} - f(r) \conj{g'(r)}) r^{\lambda + 1}}_{\epsilon}^{R}
      + \int_{\epsilon}^{R} f(r) \cdot \conj{r^{-1} E_r (E_r + \lambda) g(r)} \cdot r^\lambda \,dr.
  \end{align*}
  The boundary term for $ r = R $ vanishes as $ R \to \infty $
  by the condition \cref{item:operator-em-rad-at-infinity}.
  On the other hand, the boundary term for $ r = \epsilon $ vanishes as $ \epsilon \to 0 $
  by the condition \cref{item:operator-em-rad-at-origin} and $ \lambda > -1 $.
  Therefore, we obtain
  \[
    \innprod{\DiffEm{\lambda} f}{g}_{L^2(\setRp, r^\lambda \,dr)}
    = -\innprod{f}{\DiffEm{\lambda} g}_{L^2(\setRp, r^\lambda \,dr)}.
      \qedhere
  \]
\end{proof}

Next, we investigate the skew-adjoint operator $ d\Omega_{k, a}(\mbfe^-) $.
We restrict our attention to the case $ k = 0 $.

\begin{proposition}\label{thm:operator-em}
  Let $ a > 0 $, and further assume that $ a > 1 $ when $ N = 1 $.
  Then, we have
  \[
    d\Omega_{0, a}(\mbfe^-)
    = \closure{\restr{\DiffEm{0, a}}{\mcalD_a}}.
  \]
  Here, $ \mcalD_a $ consists of functions $ F $ satisfying the following conditions.
  \begin{enumromanp}
    \item $ F \in L^2(\setR^N, \enorm{x}^{a - 2} \,dx) \cap C^\infty(\setR^N \setminus \setenum{0}) $.
    \item $ \DiffEm{0, a} F \in L^2(\setR^N, \enorm{x}^{a - 2} \,dx) $.
    \item \label{item:operator-em-at-infinity}
          $ F $ is rapidly decreasing at infinity.
    \item \label{item:operator-em-at-origin}
          In the case $ N \geq 2 $, $ F(x) = O(1) $ and $ E_x F(x) = o(1) $
          as $ x \to 0 $.
          In the case $ N = 1 $, $ F $ and $ F' $ are continuous at the origin.
  \end{enumromanp}
\end{proposition}

\begin{proof}
  Notice that $ \mcalD_a $ contains $ W_{0, a}(\setR^N) $.
  Hence, by \cref{thm:criterion}, it suffices to show that
  $ \restr{\DiffEm{0, a}}{\mcalD_a} $ is
  a skew-symmetric operator on $ L^2(\setR^N, \enorm{x}^{a - 2} \,dx) $.
  Moreover, since the factor $ \enorm{x}^{2 - a} $
  in $ \DiffEm{0, a} = (i/a) \enorm{x}^{2 - a} \Laplacian $
  cancels with the weight function $ \enorm{x}^{a - 2} $,
  what we want to prove is that $ \restr{\Laplacian}{\mcalD_a} $ is
  a symmetric operator on $ L^2(\setR^N) $.

  Let $ F $, $ G \in \mcalD_a $.
  By Green's second identity, we have
  \begin{align*}
    & \innprod{\Laplacian F}{G}_{L^2(\setR^N)} - \innprod{F}{\Laplacian G}_{L^2(\setR^N)} \\
    &= \lim_{\epsilon \to 0,\, R \to \infty}
      \int_{\epsilon \leq \enorm{x} \leq R} (\Laplacian F(x) \conj{G(x)} - F(x) \conj{\Laplacian G(x)}) \,dx \\
    &= \lim_{\epsilon \to 0,\, R \to \infty}
      \Brack*{
        \frac{1}{r}
        \int_{\enorm{x} = r} (E_x F(x) \conj{G(x)} - F(x) \conj{E_x G(x)}) \,dx
      }_{\epsilon}^{R} \\
    &= \lim_{\epsilon \to 0,\, R \to \infty}
      \Brack*{
        r^{N - 2}
        \int_{\Sphere{N - 1}} (E_x F(r\omega) \conj{G(r\omega)} - F(r\omega) \conj{E_x G(r\omega)}) \,d\omega
      }_{\epsilon}^{R}.
  \end{align*}
  The boundary term for $ r = R $ vanishes as $ R \to \infty $
  by the condition \cref{item:operator-em-at-infinity}.
  Next, let us consider the boundary term
  \begin{equation}
    \epsilon^{N - 2}
    \int_{\Sphere{N - 1}} (E_x F(\epsilon\omega) \conj{G(\epsilon\omega)} - F(\epsilon\omega) \conj{E_x G(\epsilon\omega)}) \,d\omega
    \label{eq:operator-em}
  \end{equation}
  for $ r = \epsilon $.
  In the case $ N \geq 2 $, \cref{eq:operator-em} vanishes as $ \epsilon \to 0 $
  by the condition \cref{item:operator-em-at-origin}.
  On the other hand, in the case $ N = 1 $, \cref{eq:operator-em} reduces to
  \[
    (F'(\epsilon) \conj{G(\epsilon)} - F(\epsilon) \conj{G'(\epsilon)})
    - (F'(-\epsilon) \conj{G(-\epsilon)} - F(-\epsilon) \conj{G'(-\epsilon)}),
  \]
  which vanishes as $ \epsilon \to 0 $
  by the condition \cref{item:operator-em-at-origin}.
  This completes the proof.
\end{proof}

\begin{remark}
  As can be seen from the proof,
  the conditions (iii) and (iv) in \cref{thm:operator-em-rad,thm:operator-em} can be relaxed.
  However, since they suffice for use in the proof of
  \cref{thm:schwartz-space-rad,thm:schwartz-space},
  we do not pursue this direction here.
\end{remark}

\section{The \texorpdfstring{$ (k, a) $}{(k, a)}-generalized Schwartz space}
\label{sec:schwartz-space}

\subsection{The \texorpdfstring{$ (k, a) $}{(k, a)}-generalized Schwartz space}

\begin{definition}
  Let $ k $ be a non-negative multiplicity function, $ a > 0 $, and $ m \in \setN $
  such that $ \lambda_{k, a, m} = \frac{2m + 2\dindex{k} + N - 2}{a} > -1 $.
  We define
  \[
    \mscrS_{k, a, m}
    = L^2(\setRp, r^{2\dindex{k} + a + N - 3})_{\Omega_{k, a, m}}^\infty,
  \]
  or the space of smooth vectors for the unitary representation $ \Omega_{k, a, m} $.
\end{definition}

Note that the definition of $ \mscrS_{k, a, m} $ implicitly depends on $ N $.

\begin{definition}\label{def:schwartz-space}
  Let $ k $ be a non-negative multiplicity function and $ a > 0 $
  such that $ \lambda_{k, a, 0} = \frac{2\dindex{k} + N - 2}{a} > -1 $.
  We define the \emph{$ (k, a) $-generalized Schwartz space} as
  \[
    \mscrS_{k, a}(\setR^N)
    = L^2(\setR^N, w_{k, a}(x) \,dx)_{\Omega_{k, a}}^\infty,
  \]
  or the space of smooth vectors for the unitary representation $ \Omega_{k, a} $.
\end{definition}

\begin{proposition}\label{thm:properties-of-schwartz-space}
  Let $ k $ be a non-negative multiplicity function and $ a > 0 $
  such that $ \lambda_{k, a, 0} = \frac{2\dindex{k} + N - 2}{a} > -1 $.
  \begin{enumarabicp}
    \item $ \mscrS_{k, a}(\setR^N) $ is a dense subspace of
          $ L^2(\setR^N, w_{k, a}(x) \,dx) $.
    \item $ \mscrS_{k, a}(\setR^N) $ is
          $ \Omega_{k, a}(\widetilde{\mathit{SL}}(2, \setR)) $-stable.
    \item For any $ X \in \mathfrak{sl}(2, \setR) $,
          $ \mscrS_{k, a}(\setR^N) $ is contained in $ \Dom(d\Omega_{k, a}(X)) $
          and $ d\Omega_{k, a}(X) $-stable.
    \item For any $ X \in \mathfrak{sl}(2, \setR) $,
          $ \mscrS_{k, a}(\setR^N) $ is a core of the skew-adjoint operator $ d\Omega_{k, a}(X) $;
          that is, $ d\Omega_{k, a}(X) = \restr{d\Omega_{k, a}(X)}{\mscrS_{k, a}(\setR^N)} $.
  \end{enumarabicp}
\end{proposition}

\begin{proof}
  It follows from \cref{thm:properties-of-smooth-vectors}.
\end{proof}

\begin{corollary}\label{thm:schwartz-space-is-ft-stable}
  Let $ k $ be a non-negative multiplicity function and $ a > 0 $
  such that $ \lambda_{k, a, 0} = \frac{2\dindex{k} + N - 2}{a} > -1 $.
  Then, the $ (k, a) $-generalized Schwartz space $ \mscrS_{k, a}(\setR^N) $ is
  stable under the $ (k, a) $-generalized Fourier transform $ \mscrF_{k, a} $.
\end{corollary}

\begin{proof}
  Since
  \[
    \mscrF_{k, a}
    = e^{\frac{i\pi}{2} (\lambda_{k, a, 0} + 1)}
      \Omega_{k, a}\Paren*{
        \exp_{\widetilde{\mathit{SL}}(2, \setR)}\Paren*{
          \frac{\pi}{2} \begin{pmatrix} 0 & -1 \\ 1 & 0 \end{pmatrix}
        }
      }
  \]
  (see \cite[(5.2)]{MR2956043}),
  the assertion is a special case of \cref{thm:properties-of-schwartz-space}~(2).
\end{proof}


\subsection{Description in terms of infinite sum}

In this subsection, we write $ \widetilde{\mathit{SO}}(2) $
for the Lie subgroup of $ \widetilde{\mathit{SL}}(2, \setR) $
with Lie algebra $ \mathfrak{so}(2) $.

\begin{proposition}\label{thm:schwartz-space-as-infinite-sum-rad}
  Let $ k $ be a non-negative multiplicity function, $ a > 0 $, and $ m \in \setN $
  such that $ \lambda_{k, a, m} = \frac{2m + 2\dindex{k} + N - 2}{a} > -1 $.
  Then, the following spaces coincide.
  \begin{enumalphp}
    \item $ \mscrS_{k, a, m} $, or the space of smooth vectors
          for $ \Omega_{k, a, m} $
    \item the space of smooth vectors
          for $ \restr{\Omega_{k, a, m}}{\widetilde{\mathit{SO}}(2)} $
    \item the space
          \[
            \set*{\sum_{l \in \setN} c_l f_{k, a, m; l}}{\text{%
              $ c_l \in \setC $, $ \faml{c_l}{l \in \setN} $ is rapidly decreasing%
            }},
          \]
          where the infinite sum is understood to converge in $ L^2(\setRp, r^{2\dindex{k} + a + N - 3}) $
  \end{enumalphp}
\end{proposition}

\begin{proposition}\label{thm:schwartz-space-as-infinite-sum}
  Let $ k $ be a non-negative multiplicity function and $ a > 0 $
  such that $ \lambda_{k, a, 0} = \frac{2\dindex{k} + N - 2}{a} > -1 $.
  Then, the following spaces coincide.
  \begin{enumalphp}
    \item $ \mscrS_{k, a}(\setR^N) $, or the space of smooth vectors
          for $ \Omega_{k, a} $
    \item the space of smooth vectors
          for $ \restr{\Omega_{k, a}}{\widetilde{\mathit{SO}}(2)} $
    \item the space
          \[
            \set*{\sum_{m, l \in \setN} p_{m, l} \otimes f_{k, a, m; l}}{
              \begin{aligned}
                & p_{m, l} \in \mcalH_k^m(\Sphere{N - 1}), \\
                & \text{$ \faml{\norm{p_{m, l}}^2_{L^2(\Sphere{N - 1}, w_k(\omega) \,d\omega)}}{m, l \in \setN} $ is rapidly decreasing}
              \end{aligned}
            },
          \]
          where the infinite sum is understood to converge in $ L^2(\setR^N, w_{k, a}(x) \,dx) $
  \end{enumalphp}
\end{proposition}

Since \cref{thm:schwartz-space-as-infinite-sum-rad} is proved in the same way as
\cref{thm:schwartz-space-as-infinite-sum}, we focus on the latter.

\begin{proof}[Proof of \cref{thm:schwartz-space-as-infinite-sum}]
  It is clear that (a) is contained in (b).
  By \cref{thm:condition-to-be-a-smooth-vector} and
  \cref{thm:operator-k}, we have
  \begin{align*}
    \text{(b)}
    &= \bigcap_{n \in \setN} \Dom(d\Omega_{k, a}(\mbfk)^n) \\
    &= \set*{\sum_{m, l \in \setN} p_{m, l} \otimes f_{k, a, m; l}}{
      \begin{aligned}
        & p_{m, l} \in \mcalH_k^m(\Sphere{N - 1}), \\
        & \textstyle\sum_{m, l \in \setN} (\lambda_{k, a, m} + 2l + 1)^{2n} \norm{p_{m, l}}_{L^2(\Sphere{N - 1}, w_k(\omega) \,d\omega)}^2 < \infty \\
        & \text{for all $ n \in \setN $}
      \end{aligned}
    } \\
    &= \text{(c)}.
  \end{align*}

  We next prove that (c) is contained in (a).
  Recall from \cref{ssec:onb} that $ (\mbfk, \mbfn^+, \mbfn^-) $ denotes
  the Cayley transform of the standard basis $ (\mbfh, \mbfe^+, \mbfe^-) $
  of $ \mathfrak{sl}(2, \setR) $.
  Since $ (i\mbfk, \mbfn^+ + \mbfn^-, i (\mbfn^+ - \mbfn^-)) $ is a basis of
  $ \mathfrak{sl}(2, \setR) $,  by \cref{thm:condition-to-be-a-smooth-vector},
  it suffices to show that (c) is contained in the domains
  of $ d\Omega_{k, a}(\mbfk) $ and $ d\Omega_{k, a}(\mbfn^+ \pm \mbfn^-) $ and
  stable under these operators.
  The assertion concerning $ d\Omega_{k, a}(\mbfk) $ holds
  since (c) coincides with $ \bigcap_{n \in \setN} \Dom(d\Omega_{k, a}(\mbfk)^n) $
  as we see in the previous paragraph.
  In what follows, we prove the assertion concerning $ d\Omega_{k, a}(\mbfn^+ \pm \mbfn^-) $.
  Take any element $ F = \sum_{m, l \in \setN} p_{m, l} \otimes f_{k, a, m; l} $ of (c)
  and set $ F_n = \sum_{m, l = 0}^{n} p_{m, l} \otimes f_{k, a, m; l} $ for $ n \in \setN $.
  Then, for each $ n \in \setN $, we have
  \[
    F_n
    \in W_{k, a}(\setR^N)
    \subseteq \Dom(d\Omega_{k, a}(\mbfn^+ \pm \mbfn^-))
  \]
  and
  \begin{align*}
    & d\Omega_{k, a}(\mbfn^+ \pm \mbfn^-) F_n \\
    &= \sum_{m, l = 0}^{n} p_{m, l} \otimes \Paren*{
      i \sqrt{(l + 1) (\lambda_{k, a, m} + l + 1)} \, f_{k, a, m; l + 1}
      \pm i \sqrt{l (\lambda_{k, a, m} + l)} \, f_{k, a, m; l - 1}
    }
  \end{align*}
  by \cref{thm:action}.
  It is clear that $ \lim_{n \to \infty} F_n = F $ in $ L^2(\setR^N, w_{k, a}(x) \,dx) $.
  Moreover, since $ \faml{\norm{p_{m, l}}_{L^2(\Sphere{N - 1}, w_k(\omega) \,d\omega)}^2}{m, l \in \setN} $
  is rapidly decreasing, we have
  \begin{align*}
    & \lim_{n \to \infty} d\Omega_{k, a}(\mbfn^+ + \mbfn^-) F_n \\
    &= \sum_{m, l \in \setN} p_{m, l} \otimes \Paren*{
      i \sqrt{(l + 1) (\lambda_{k, a, m} + l + 1)} \, f_{k, a, m; l + 1}
      \pm i \sqrt{l (\lambda_{k, a, m} + l)} \, f_{k, a, m; l - 1}
    }
  \end{align*}
  in $ L^2(\setR^N, w_{k, a}(x) \,dx) $, and this limit again belongs to (c).
  Since $ d\Omega_{k, a}(\mbfn^+ \pm \mbfn^-) $ is a (skew-adjoint or self-adjoint,
  and hence) closed operator, we have
  \begin{align*}
    & F \in \Dom(d\Omega_{k, a}(\mbfn^+ \pm \mbfn^-)), \\
    & d\Omega_{k, a}(\mbfn^+ \pm \mbfn^-) F
      = \lim_{n \to \infty} d\Omega_{k, a}(\mbfn^+ \pm \mbfn^-) F_n
      \in \text{(c)}.
  \end{align*}
  Therefore, (c) is contained in the domains of $ d\Omega_{k, a}(\mbfn^+ \pm \mbfn^-) $
  and stable under these operators.
\end{proof}

\subsection{Determination of \texorpdfstring{$ \mscrS_{k, a, m} $}{Sk,a,m}}

In the following theorem, we write
\[
  \mscrS(\setRzp)
  = \set{\restr{u}{\setRzp}}{u \in \mscrS(\setR)},
\]
which we equip with the Fréchet topology defined by the seminorms
\[
  u \mapsto
  \sup_{t \geq 0} \abs*{t^\beta \Paren*{\Odif{t}}^\gamma u(t)}
\]
for $ \beta $, $ \gamma \in \setN $.

\begin{theorem}\label{thm:schwartz-space-rad}
  Let $ k $ be a non-negative multiplicity function, $ a > 0 $, and $ m \in \setN $
  such that $ \lambda_{k, a, m} = \frac{2m + 2\dindex{k} + N - 2}{a} > -1 $.
  Then, we have
  \[
    \mscrS_{k, a, m}
    = \set*{r \mapsto r^m u(r^a)}{u \in \mscrS(\setRzp)}.
  \]
\end{theorem}

\begin{proof}[Proof that $ \set*{r \mapsto r^m u(r^a)}{u \in \mscrS(\setRzp)} \subseteq \mscrS_{k, a, m} $]
  By \cref{thm:infinitesimal-lie-algebra-representation,thm:condition-to-be-a-smooth-vector},
  it suffices to show that $ \set*{r \mapsto r^m u(r^a)}{u \in \mscrS(\setRzp)} $ is
  contained in the domains of $ d\Omega_{k, a, m}(\mbfe^+) $ and $ d\Omega_{k, a, m}(\mbfe^-) $ and
  stable under these operators.

  A simple computation shows that $ \set*{r \mapsto r^m u(r^a)}{u \in \mscrS(\setRzp)} $ is
  contained in\linebreak
  $ L^2(\setRp, r^{2\dindex{k} + a + N - 3} \,dr) \cap C^\infty(\setRp) $ and
  stable under $ \DiffEp{k, a}[m] $ and $ \DiffEm{k, a}[m] $.
  Moreover, for any $ u \in \mscrS(\setRzp) $, the function $ r \mapsto r^m u(r^a) $ satisfies
  the conditions \cref{item:operator-em-rad-at-infinity} and \cref{item:operator-em-rad-at-origin}
  in \cref{thm:operator-em-rad}.
  This, together with \cref{thm:operator-ep-rad,thm:operator-em-rad}, proves the assertion.
\end{proof}

\begin{proof}[Proof that $ \mscrS_{k, a, m} \subseteq \set*{r \mapsto r^m u(r^a)}{u \in \mscrS(\setRzp)} $]
  By \cref{thm:schwartz-space-as-infinite-sum-rad}, we have
  \begin{align*}
    \mscrS_{k, a, m}
    &= \set*{\sum_{l \in \setN} c_l f_{k, a, m; l}}{\text{%
        $ c_l \in \setC $, $ \faml{c_l}{l \in \setN} $ is rapidly decreasing%
      }} \\
    &= \set*{
        \sum_{l \in \setN} c_l
        r^m L_l^{\lambda_{k, a, m}}\Paren*{\frac{2}{a} r^a} \exp\Paren*{-\frac{1}{a} r^a}
      }{\text{%
        $ c_l \in \setC $, $ \faml{c_l}{l \in \setN} $ is rapidly decreasing%
      }},
  \end{align*}
  where the last equality holds because
  \[
    f_{k, a, m; l}(r)
    = \Paren*{\frac{2^{\lambda_{k, a, m} + 1} \Gamma(l + 1)}{a^{\lambda_{k, a, m}} \Gamma(\lambda_{k, a, m} + l + 1)}}^{1/2}
      r^m
      L_l^{\lambda_{k, a, m}} \Paren*{\frac{2}{a} r^a}
      \exp \Paren*{-\frac{1}{a} r^a}
  \]
  and the constant multiple is of polynomial growth with respect to $ l $.
  Hence, it suffices to show that, for any rapidly decreasing family
  $ \faml{c_l}{l \in \setN} $ of complex numbers, the infinite sum
  \begin{equation}
    \sum_{l \in \setN} c_l L_l^{\lambda_{k, a, m}}(t) e^{-t/2}
      \label{eq:schwartz-space-rad}
  \end{equation}
  of functions of $ t \in \setRzp $ converges in $ \mscrS(\setRzp) $.

  Let us take any $ \beta $, $ \gamma \in \setN $.
  By \cref{thm:estimate-1}, for each $ l \in \setN $, we have
  \begin{align*}
    & \sup_{t \geq 0} \abs*{t^\beta \Paren*{\Odif{t}}^\gamma (L_l^{\lambda_{k, a, m}}(t) e^{-t/2})} \\
    &\leq 2^{\beta + \max\setenum{\beta - \lambda_{k, a, m}, 0}} (l + \beta)^{\underline{\beta}}
      \binom{\max\setenum{\lambda_{k, a, m} - \beta, 0} + l + \gamma}{l}.
  \end{align*}
  Since the right-hand side of the above inequality is of polynomial growth
  with respect to $ l $, it follows that
  \[
    \sum_{l \in \setN} \sup_{t \geq 0} \abs*{t^\beta \Paren*{\Odif{t}}^\gamma (L_l^{\lambda_{k, a, m}}(t) e^{-t/2})}
    < \infty.
  \]
  Therefore, the infinite sum \cref{eq:schwartz-space-rad} converges in $ \mscrS(\setRzp) $.
  This completes the proof.
\end{proof}

\subsection{Determination of \texorpdfstring{$ \mscrS_{k, a}(\setR) $}{Sk,a(R)}}

In this subsection,
we determine the one-dimensional $ (k, a) $-generalized Schwartz space.
This is one of the main results of this paper.

\begin{theorem}\label{thm:schwartz-space-of-rank-one}
  Let $ k $ be a non-negative multiplicity function and $ a > 0 $
  such that $ \lambda_{k, a, 0} = \frac{2\dindex{k} - 1}{a} > -1 $.
  Then, we have
  \[
    \mscrS_{k, a}(\setR)
    = \set*{x \mapsto u(\abs{x}^a) + x v(\abs{x}^a)}{\text{%
      $ u $, $ v \in \mscrS(\setRzp) $%
    }}.
  \]
\end{theorem}

\begin{proof}
  Since $ \Sphere{0} = \setenum{\pm 1} $ and
  \[
    \mcalH_k^m(\Sphere{0})
    = \begin{cases}
        \setC 1    & (m = 0) \\
        \setC \sgn & (m = 1) \\
        0          & (m \geq 2),
      \end{cases}
  \]
  \cref{thm:unitary-representation} gives the unitary equivalence
  \begin{align*}
    & \map{\Phi}{L^2(\setRp, r^{2\dindex{k} + a - 2} \,dr) \oplus L^2(\setRp, r^{2\dindex{k} + a - 2} \,dr)}{L^2(\setR, w_{k, a}(x) \,dx)}, \\
    & \Phi(f, g)(x) = f(\abs{x}) + \sgn(x) g(\abs{x})
  \end{align*}
  from $ \Omega_{k, a, 0} \oplus \Omega_{k, a, 1} $ onto $ \Omega_{k, a} $.
  Therefore, we have
  \begin{align*}
    \mscrS_{k, a}(\setR)
    &= \Phi(\mscrS_{k, a, 0} \oplus \mscrS_{k, a, 1}) \\
    &= \set{x \mapsto f(\abs{x}) + \sgn(x) g(\abs{x})}{\text{%
      $ f \in \mscrS_{k, a, 0} $, $ g \in \mscrS_{k, a, 1} $%
    }} \\
    &= \set*{x \mapsto u(\abs{x}^a) + x v(\abs{x}^a)}{\text{%
      $ u $, $ v \in \mscrS(\setRzp) $%
    }}
  \end{align*}
  by \cref{thm:schwartz-space-rad}.
\end{proof}

\subsection{Determination of \texorpdfstring{$ \mscrS_{0, a}(\setR^N) $}{S0,a(RN)} for rational \texorpdfstring{$ a $}{a}}

In this subsection,
we determine the $ (0, a) $-generalized Schwartz space $ \mscrS_{0, a}(\setR^N) $
for rational $ a $.
This is one of the main results of this paper.
For this purpose, we use propositions in \cref{sec:pg,sec:harmonic-polynomial}.

Recall~\cite[Proof of Theorem~2.1]{MR304972} that
the space $ \mcalP^m(\setR^N) $ of homogeneous polynomials of degree $ m $ decomposes as
\[
  \mcalP^m(\setR^N)
  = \bigoplus_{n = 0}^{\lfloor m/2 \rfloor} \enorm{x}^{2n} \mcalH^{m - 2n}(\setR^N),
\]
and hence the space $ \setC[[x]] = \setC[[x_1, \dots, x_N]] $ of formal power series
decomposes as
\[
  \setC[[x]]
  = \prod_{m \in \setN} \mcalP^m(\setR^N)
  = \prod_{m \in \setN} \bigoplus_{n = 0}^{\lfloor m/2 \rfloor} \enorm{x}^{2n} \mcalH^{m - 2n}(\setR^N)
  = \prod_{m, n \in \setN} \enorm{x}^{2n} \mcalH^{m}(\setR^N).
\]

In what follows, we write
\[
  S_{2p}(\setR^N)
  = \set*{F \in \mscrS(\setR^N)}{
      \tau(F) \in \prod_{m, n = 0}^{\infty} \enorm{x}^{2np} \mcalH^m(\setR^N)
    },
\]
where $ \tau(F) = \sum_{\alpha \in \setN^N} \frac{\partial^\alpha F(0)}{\alpha!} x^\alpha
\in \setC[[x]] $ denotes the formal Taylor series of $ F $ at the origin.
The above decomposition of $ \setC[[x]] $ yields that
$ S_2(\setR^N) = \mscrS(\setR^N) $.

\begin{lemma}\label{thm:stability}
  Let $ a = 2p/q \in \setQp $ with $ p $, $ q \in \setNp $,
  and further assume that $ a > 1 $ when $ N = 1 $.
  Then, the space
  \[
    \sum_{j = 0}^{q - 1} \enorm{x}^{ja} S_{2p}(\setR^N)
  \]
  is stable under $ \omega_{0, a}(\mathfrak{sl}(2, \setR))
  = \lspan_{\setR} \setenum{\DiffH{0, a}, \DiffEp{0, a}, \DiffEm{0, a}} $.
\end{lemma}

\begin{proof}
  It suffices to show the stability under $ \DiffEp{0, a} = (i/a) \enorm{x}^a $ and
  $ \DiffEm{0, a} = (i/a) \enorm{x}^{2 - a} \Laplacian $.
  The former is clear, so we focus on the latter.

  Let $ j \in \setenum{0, \dots, q - 1} $ and $ F \in S_{2p}(\setR^N) $.
  By the definition of $ S_{2p}(\setR^N) $,
  the formal Taylor series of $ F $ at the origin can be written as
  \[
    \tau(F)
    = \sum_{m, n = 0}^{\infty} \enorm{x}^{2np} p_{m, n}(x),
  \]
  where $ p_{m, n} \in \mcalH^m(\setR^N) $.
  Hence, for any $ n_0 \in \setN $, there exists $ G_{n_0} \in C^\infty(\setR^N) $
  such that
  \[
    F(x)
    = \sum_{m + 2np < 2n_0 p} \enorm{x}^{2np} p_{m, n}(x) + \enorm{x}^{2n_0 p} G_{n_0}(x).
  \]
  Since $ \Laplacian p_{m, n} = 0 $ and $ E_x p_{m, n} = mp_{m, n} $, we have
  \begin{align*}
    & \enorm{x}^{2 - a} \Laplacian (\enorm{x}^{ja} \enorm{x}^{2np} p_{m, n}(x)) \\
    &= \enorm{x}^{2 - a} \Paren*{
      \Laplacian (\enorm{x}^{2np + ja}) p_{m, n}(x)
      + 2 \sum_{l = 1}^{N} \Pdiff{\enorm{x}^{2np + ja}}{x_l} \Pdiff{p_{m, n}}{x_l}(x)
    } \\
    &= \enorm{x}^{2 - a} (
      (2np + ja)(2np + ja + N - 2) \enorm{x}^{2np + ja - 2} p_{m, n}(x)
      + 2(2np + ja) \enorm{x}^{2np + ja - 2} E_x p_{m, n}(x)
    ) \\
    &= (2np + ja)(2np + ja + 2m + N - 2) \enorm{x}^{2np + (j - 1)a} p_{m, n}(x).
  \end{align*}
  On the other hand, we have
  \begin{align*}
    & \enorm{x}^{2 - a} \Laplacian (\enorm{x}^{ja} \enorm{x}^{2n_0 p} G_{n_0}(x)) \\
    &= \enorm{x}^{2 - a} \Paren*{
      \Laplacian (\enorm{x}^{2n_0 p + ja}) G_{n_0}(x)
      + 2 \sum_{l = 1}^{N} \Pdiff{\enorm{x}^{2n_0 p + ja}}{x_l} \Pdiff{G_{n_0}}{x_l}(x)
      + \enorm{x}^{2n_0 p + ja} \Laplacian G_{n_0}(x)
    } \\
    &= \enorm{x}^{2 - a} (
      (2n_0 p + ja)(2n_0 p + ja + N - 2) \enorm{x}^{2n_0 p + ja - 2} G_{n_0}(x) \\
    &\qquad + 2(2n_0 p + ja) \enorm{x}^{2n_0 p + ja - 2} E_x G_{n_0}(x)
      + \enorm{x}^{2n_0 p + ja} \Laplacian G_{n_0}(x)
    )\\
    &= \enorm{x}^{2n_0 p + (j - 1)a} (
      (2n_0 p + ja)(2n_0 p + ja + N - 2) G_{n_0}(x) \\
    &\qquad + 2(2n_0 p + ja) E_x G_{n_0}(x)
      + \enorm{x}^2 \Laplacian G_{n_0}(x)
    ).
  \end{align*}
  From these equalities, we obtain
  \begin{align*}
    & \enorm{x}^{2 - a} \Laplacian (\enorm{x}^{ja} F(x)) \\
    &\in \sum_{m + 2np < 2n_0 p} (2np + ja)(2np + ja + 2m + N - 2) \enorm{x}^{2np + (j - 1)a} p_{m, n}(x) \\
    &\qquad + \enorm{x}^{2n_0 p + (j - 1)a} C^\infty(\setR^N).
  \end{align*}
  Observe that $ 2np + (j - 1)a = (2nq + j - 1)a $.
  Moreover, $ 2nq + j - 1 < 0 $ only when $ n = j = 0 $, in which case
  the corresponding term in the above sum vanishes.
  Since the above inclusion holds for any $ n_0 $,
  it follows that $ \enorm{x}^{2 - a} \Laplacian (\enorm{x}^{ja} F(x))
  \in \enorm{x}^{j'a} S_{2p}(\setR^N) $,
  where $ j' = j - 1 $ for $ j \in \setenum{1, \dots, q - 1} $ and
  $ j' = q - 1 $ for $ j = 0 $.
  This completes the proof.
\end{proof}

\begin{theorem}\label{thm:schwartz-space}
  Let $ a = 2p/q \in \setQp $ with $ p $, $ q \in \setNp $,
  and further assume that $ a > 1 $ when $ N = 1 $.
  Then, we have
  \[
    \mscrS_{0, a}(\setR^N)
    = \sum_{j = 0}^{q - 1} \enorm{x}^{ja} S_{2p}(\setR^N).
  \]
\end{theorem}

\begin{example}
  We list some special cases of \cref{thm:schwartz-space}.
  \begin{itemize}
    \item In the case $ p = q = 1 $ so that $ a = 2 $, we have
          \[
            \mscrS_{0, 2}(\setR^N)
            = \mscrS(\setR^N),
          \]
          which recovers the classical result.
    \item In the case $ q = 1 $ so that $ a = 2p $, we have
          \[
            \mscrS_{0, 2p}(\setR^N)
            = S_{2p}(\setR^N)
            = \set*{F \in \mscrS(\setR^N)}{
                \tau(F) \in \prod_{m, n = 0}^{\infty} \enorm{x}^{2np} \mcalH^m(\setR^N)
            }.
          \]
    \item In the case $ p = 1 $ so that $ a = 2/q $, we have
          \[
            \mscrS_{0, 2/q}(\setR^N)
            = S_{2p}(\setR^N)
            = \sum_{j = 0}^{q - 1} \enorm{x}^{2j/q} \mscrS(\setR^N).
          \]
  \end{itemize}
\end{example}

\begin{proof}[Proof that $ \sum_{j = 0}^{q - 1} \enorm{x}^{ja} S_{2p}(\setR^N) \subseteq \mscrS_{0, a}(\setR^N) $]
  By \cref{thm:infinitesimal-lie-algebra-representation,thm:condition-to-be-a-smooth-vector},
  it suffices to show that $ \sum_{j = 0}^{q - 1} \enorm{x}^{ja} S_{2p}(\setR^N) $ is
  contained in the domains of $ d\Omega_{0, a}(\mbfe^+) $ and $ d\Omega_{0, a}(\mbfe^-) $ and
  stable under these operators.

  The space $ \sum_{j = 0}^{q - 1} \enorm{x}^{ja} S_{2p}(\setR^N) $ is clearly
  contained in $ L^2(\setR^N, \enorm{x}^{a - 2}) \cap C^\infty(\setR^N \setminus \setenum{0}) $ and,
  by \cref{thm:stability}, is stable under $ \DiffEp{0, a} $ and $ \DiffEm{0, a} $.
  Moreover, any $ F \in \sum_{j = 0}^{q - 1} \enorm{x}^{ja} S_{2p}(\setR^N) $ satisfies
  the conditions \cref{item:operator-em-at-infinity} and \cref{item:operator-em-at-origin}
  in \cref{thm:operator-em}.
  This, together with \cref{thm:operator-ep,thm:operator-em}, proves the assertion.
\end{proof}

\begin{proof}[Proof that $ \mscrS_{0, a}(\setR^N) \subseteq \sum_{j = 0}^{q - 1} \enorm{x}^{ja} S_{2p}(\setR^N) $]
  Recall from \cref{thm:schwartz-space-as-infinite-sum} that
  \begin{equation}
    \mscrS_{0, a}(\setR^N)
    = \set*{\sum_{m, l \in \setN} p_{m, l} \otimes f_{0, a, m; l}}{
      \begin{aligned}
        & p_{m, l} \in \mcalH^m(\Sphere{N - 1}), \\
        & \text{$ \faml{\norm{p_{m, l}}_{L^2(\Sphere{N - 1})}^2}{m, l \in \setN} $ is rapidly decreasing}
      \end{aligned}
    }.
    \label{eq:schwartz-space-1}
  \end{equation}
  We first fix $ m $, $ l \in \setN $ and $ p_{m, l} \in \mcalH^m(\Sphere{N - 1}) $,
  and consider the function $ p_{m, l} \otimes f_{0, a, m; l} $.
  We extend $ p_{m, l} $ to a harmonic polynomial of degree $ m $ on $ \setR^N $,
  for which we again write $ p_{m, l} $.
  Then, we have
  \begin{align*}
    & (p_{m, l} \otimes f_{0, a, m; l})(x) \\
    &= \Paren*{\frac{2^{\lambda_{0, a, m} + 1} \Gamma(l + 1)}{a^{\lambda_{0, a, m}} \Gamma(\lambda_{0, a, m} + l + 1)}}^{1/2}
      p_{m, l}\Paren*{\frac{x}{\enorm{x}}}
      \enorm{x}^m L_l^{\lambda_{0, a, m}}\Paren*{\frac{2}{a} \enorm{x}^a} \exp\Paren*{-\frac{1}{a} \enorm{x}^a} \\
    &= \Paren*{\frac{2^{\lambda_{0, a, m} + 1} \Gamma(l + 1)}{a^{\lambda_{0, a, m}} \Gamma(\lambda_{0, a, m} + l + 1)}}^{1/2}
      p_{m, l}(x)
      L_l^{\lambda_{0, a, m}}\Paren*{\frac{2}{a} \enorm{x}^a} \exp\Paren*{-\frac{1}{a} \enorm{x}^a}.
  \end{align*}
  If we take smooth functions $ h_{m, l; 0} $, $ \dots $, $ h_{m, l; q - 1} $ on $ \setRzp $
  such that
  \[
    \Paren*{\frac{\Gamma(l + 1)}{\Gamma(\lambda_{0, a, m} + l + 1)}}^{1/2}
      L_l^{\lambda_{0, a, m}}(t) e^{-t/2}
    = \sum_{j = 0}^{q - 1} t^j h_{m, l; j}(t^q),
  \]
  the above identity becomes
  \begin{equation}
    (p_{m, l} \otimes f_{0, a, m; l})(x)
    = \sum_{j = 0}^{q - 1} 
      \sqrt{2}\, \Paren*{\frac{2}{a}}^j \enorm{x}^{ja} \cdot
      \Paren*{\frac{2}{a}}^{\lambda_{0, a, m} /2}
      p_{m, l}(x) h_{m, l; j}\Paren*{\Paren*{\frac{2}{a}}^q \enorm{x}^{2p}}.
    \label{eq:schwartz-space-2}
  \end{equation}
  Here, we note that
  \begin{equation}
    \Paren*{x \mapsto p_{m, l}(x) h_{m, l; j} \Paren*{\Paren*{\frac{2}{a}}^q \enorm{x}^{2p}}}
    \in S_{2p}(\setR^N)
    \label{eq:schwartz-space-3}
  \end{equation}
  for all $ j $.

  We now consider $ \faml{p_{m, l}}{m, l \in \setN} $
  such that $ p_{m, l} \in \mcalH^m(\Sphere{N - 1}) $ and
  $ \faml{\norm{p_{m, l}}_{L^2(\Sphere{N - 1})}^2}{m, l \in \setN} $ is rapidly decreasing.
  By \cref{thm:laguerre-satisfies-pg} and \cref{thm:pg-decomposition},
  we can take $ h_{m, l; 0} $, $ \dots $, $ h_{m, l; q - 1} $
  in the previous paragraph so that $ \faml{h_{m, l; 0}}{m, l \in \setN} $,
  $ \dots $, $ \faml{h_{m, l; q - 1}}{m, l \in \setN} $ satisfy ($ 2p $-PG) (\cref{def:pg}).
  In what follows, under this situation, we prove that the infinite sum
  \begin{equation}
    \sum_{m, l \in \setN}
      \Paren*{\frac{2}{a}}^{\lambda_{0, a, m} /2}
      p_{m, l}(x) h_{m, l; j}\Paren*{\Paren*{\frac{2}{a}}^q \enorm{x}^{2p}}
      \label{eq:schwartz-space-4}
  \end{equation}
  of functions of $ x \in \setR^N $ converges in $ \mscrS(\setR^N) $.
  Once this has been established,
  by \cref{eq:schwartz-space-1}, \cref{eq:schwartz-space-2}, \cref{eq:schwartz-space-3} and
  the fact that $ S_{2p}(\setR^N) $ is closed in $ \mscrS(\setR^N) $,
  the assertion follows.

  Let $ \beta $, $ \gamma \in \setN^N $. Then, we have
  \begin{align*}
    & \partial^\gamma \Paren*{
      \Paren*{\frac{2}{a}}^{\lambda_{0, a, m} /2}
      p_{m, l}(x) h_{m, l; j}\Paren*{\Paren*{\frac{2}{a}}^q \enorm{x}^{2p}}
    } \\
    &= \Paren*{\frac{2}{a}}^{\lambda_{0, a, m} /2}
      \sum_{\delta \leq \gamma}
      \binom{\gamma}{\delta} \cdot
      \partial^\delta p_{m, l}(x) \cdot
      \partial^{\gamma - \delta} \Paren*{h_{m, l; j}\Paren*{\Paren*{\frac{2}{a}}^q \enorm{x}^{2p}}} \\
    &= \Paren*{\frac{2}{a}}^{\lambda_{0, a, m} /2}
      \sum_{\delta \leq \gamma,\; \abs{\delta} \leq m}
      \binom{\gamma}{\delta}
      \partial^\delta p_{m, l}(x)
      \sum_{n = 0}^{\abs{\gamma - \delta}}
      P_{\gamma - \delta, n}(x) h_{m, l; j}^{(n)}\Paren*{\Paren*{\frac{2}{a}}^q \enorm{x}^{2p}},
  \end{align*}
  where $ P_{\gamma - \delta, n} $ is a polynomial depending
  only on $ \gamma - \delta $ and $ n $
  besides the fixed parameters $ N $, $ p $ and $ q $,
  and hence
  \begin{align}
    & \abs*{
        x^\beta \partial^\gamma \Paren*{
          \Paren*{\frac{2}{a}}^{\lambda_{0, a, m} /2}
          p_{m, l}(x) h_{m, l; j}\Paren*{\Paren*{\frac{2}{a}}^q \enorm{x}^{2p}}
        }
      } \notag \\
    &\leq \Paren*{\frac{2}{a}}^{\lambda_{0, a, m} /2}
      \enorm{x}^{\abs{\beta}}
      \sum_{\delta \leq \gamma,\; \abs{\gamma - \delta} \leq m}
      \binom{\gamma}{\delta}
      \abs{\partial^\delta p_{m, l}(x)}
      \sum_{n = 0}^{\abs{\gamma - \delta}}
      \abs{P_{\gamma - \delta, n}(x)} \abs*{h_{m, l; j}^{(n)}\Paren*{\Paren*{\frac{2}{a}}^q \enorm{x}^{2p}}}.
      \label{eq:schwartz-space-5}
  \end{align}
  By \cref{thm:pointwise-estimates-for-derivatives-of-harmonic-polynomials},
  we have
  \begin{equation}
    \abs{\partial^\delta p_{m, l}(x)}
    \leq \Paren*{
        (2m)^{\abs{\delta}}
        \Paren*{m + \frac{N}{2} - 1}^{\underline{\abs{\delta}}} \,
        \frac{\dim \mcalH^{m - \abs{\delta}}(\setR^N)}{\vol(\Sphere{N - 1})}
      }^{1/2}
      \norm{p_{m, l}}_{L^2(\Sphere{N - 1})}
      \enorm{x}^{m - \abs{\delta}},
    \label{eq:schwartz-space-6}
  \end{equation}
  and notice that $ \dim \mcalH^{m - \abs{\delta}}(\setR^N)
  = \binom{m - \abs{\delta} + N - 1}{N - 1} - \binom{m - \abs{\delta} + N - 3}{N - 1}
  = O(m^{N - 2}) $.
  Hence, the right-hand side of \cref{eq:schwartz-space-5} is bounded above
  by a finite sum of terms of the form
  \[
    \norm{p_{m, l}}_{L^2(\Sphere{N - 1})} \times
    (\text{polynomial of $ m $}) \times
    \Paren*{\frac{2}{a}}^{\lambda_{0, a, m} /2}
    \enorm{x}^{m + b}
    \abs*{h_{m, l; j}^{(n)}\Paren*{\Paren*{\frac{2}{a}}^q \enorm{x}^{2p}}}.
  \]
  Here,
  \begin{itemize}
    \item the degree and coefficients of the ``polynomial of $ m $''
          depend only on $ \gamma $ besides the fixed parameter $ N $, and
    \item $ b \in \setZ $ ranges over those values
          satisfying $ m + b \geq 0 $ and $ \abs{b} \leq M_{\beta, \gamma} $,
          where $ M_{\beta, \gamma} $ depends only on $ \beta $ and $ \gamma $
          besides the fixed parameters $ N $, $ p $ and $ q $.
  \end{itemize}
  We estimate the last factor of \cref{eq:schwartz-space-6}.
  Setting $ t = (2/a)^q \enorm{x}^{2p} $, we have
  \begin{align*}
    \Paren*{\frac{2}{a}}^{\lambda_{0, a, m} /2}
      \enorm{x}^{m + b}
      \abs*{h_{m, l; j}^{(n)}\Paren*{\Paren*{\frac{2}{a}}^q \enorm{x}^{2p}}}
    &= \Paren*{\frac{2}{a}}^{\lambda_{0, a, m} /2}
      \Paren*{\Paren*{\frac{a}{2}}^{1/a} t^{1/2p}}^{m + b}
      \abs{h_{m, l; j}^{(n)}(t)} \\
    &= \Paren*{\frac{a}{2}}^{(2b - N + 2)/2a}
      t^{(m + b)/2p} \abs{h_{m, l; j}^{(n)}(t)},
  \end{align*}
  which is bounded by a polynomial of $ m $ and $ l $
  since $ \faml{h_{m, l; j}}{m, l \in \setN} $ satisfies ($ 2p $-PG).
  Therefore, we obtain
  \begin{align*}
    & \sup_{x \in \setR^N} \abs*{
        x^\beta \partial^\gamma \Paren*{
          \Paren*{\frac{2}{a}}^{\lambda_{0, a, m} /2}
          p_{m, l}(x) h_{m, l; j}\Paren*{\Paren*{\frac{2}{a}}^q \enorm{x}^{2p}}
        }
      } \\
    &\leq \norm{p_{m, l}}_{L^2(\Sphere{N - 1})} \times (\text{polynomial of $ m $ and $ l $}).
  \end{align*}
  Since $ \faml{\norm{p_{m, l}}_{L^2(\Sphere{N - 1})}}{m, l \in \setN} $ is rapidly decreasing,
  it follows that
  \[
    \sum_{m, l \in \setN} \sup_{x \in \setR^N} \abs*{
        x^\beta \partial^\gamma \Paren*{
          \Paren*{\frac{2}{a}}^{\lambda_{0, a, m} /2}
          p_{m, l}(x) h_{m, l; j}\Paren*{\Paren*{\frac{2}{a}}^q \enorm{x}^{2p}}
        }
      }
    < \infty.
  \]
  Therefore, the infinite sum \cref{eq:schwartz-space-4} converges in $ \mscrS(\setR^N) $.
  This completes the proof.
\end{proof}

\subsection{Connection with previous work}
\label{ssec:previous-work}

In this subsection, we discuss the connection between our results and previous work.

\begin{remark}\label{rem:smooth-vectors-for-kostants-model}
  In \cite[Section~4]{MR1755901}, Kostant studies the properties of the space
  of the smooth vectors $ L^2(\setRp, dr)_{\pi_\lambda}^\infty $
  for the unitary representation $ \pi_\lambda $ (see \cref{rem:kostants-model}),
  but explicit determination of this space is not addressed there.

  By \cref{thm:schwartz-space-rad} and the unitary equivalence in \cref{rem:kostants-model},
  we now obtain the explicit formula
  \[
    L^2(\setRp, dr)_{\pi_\lambda}^\infty
    = r^{\lambda/2} \mscrS(\setRzp).
  \]
\end{remark}

\begin{remark}
  Gorbachev--Ivanov--Tikhonov~\cite[Propositions~5.2~(ii) and 5.5~(ii)]{MR4629458}
  proved that the condition $ a \in \set{2/q}{q \in \setNp} $ is necessary
  for the image $ \mscrF_{k, a}(\mscrS(\setR^N)) $ to consist of
  rapidly decreasing functions at infinity,
  and that this condition is also sufficient in the case $ N = 1 $.

  Our results show that this condition is sufficient for general $ N $ in the case $ k = 0 $.
  Indeed, since $ \mscrS_{0, 2/q}(\setR^N) = \sum_{j = 0}^{q - 1} \enorm{x}^{2j/q} \mscrS(\setR^N) $
  is stable under $ \mscrF_{0, 2/q} $ (\cref{thm:schwartz-space,thm:schwartz-space-is-ft-stable}),
  we obtain
  \[
    \mscrF_{0, 2/q}(\mscrS(\setR^N))
    \subseteq \mscrF_{0, 2/q}\Paren*{\sum_{j = 0}^{q - 1} \enorm{x}^{2j/q} \mscrS(\setR^N)}
    =         \sum_{j = 0}^{q - 1} \enorm{x}^{2j/q} \mscrS(\setR^N).
  \]
\end{remark}

\begin{remark}
  Ivanov~\cite[(1.3), Theorem~1]{MR4684153} considered the space
  \[
    \mscrS_a(\setR)
    = \set{x \mapsto f(\abs{x}^{a/2}) + x g(\abs{x}^{a/2})}{\text{%
      $ f $, $ g \in \mscrS(\setR) $ are even functions%
    }}
  \]
  in his notation, and proved that this space is stable under
  the one-dimensional $ (k, a) $-generalized Fourier transform $ \mscrF_{k, a} $ and
  the differential operator $ \DiffEm{k, a} = (i/a) \enorm{x}^{2 - a} \Laplacian $.

  By a classical result found in \cite[Theorem~1]{MR7783},
  every even function $ f \in \mscrS(\setR) $ can be expressed as
  $ f(x) = u(x^2) $ ($ u \in \mscrS(\setRzp) $).
  Hence, $ \mscrS_a(\setR) $ coincides with
  the one-dimensional $ (k, a) $-generalized Schwartz space
  \[
    \mscrS_{k, a}(\setR)
    = \set*{x \mapsto u(\abs{x}^a) + x v(\abs{x}^a)}{\text{%
      $ u $, $ v \in \mscrS(\setRzp) $%
    }}
  \]
  in our terminology (\cref{thm:schwartz-space-of-rank-one}).

  In \cite{MR4684153}, which focuses on the case $ N = 1 $,
  the space of smooth vectors is not discussed.
  The fact that $ \mscrS_a(\setR) $ coincides with the space $ \mscrS_{k, a}(\setR) $
  of smooth vectors shows stability not only under individual operators
  such as $ \mscrF_{k, a} $ and $ \DiffEm{k, a} $
  but also under $ \Omega_{k, a}(g) $ for any $ g \in \widetilde{SL}(2, \setR) $ and
  $ d\Omega_{k, a}(X) $ for any $ X \in \mathfrak{sl}(2, \setR) $.
\end{remark}

\begin{remark}
  In a recent preprint by Faustino--Negzaoui~\cite{arXiv2507-04064}, 
  another Schwartz-type space for the $ (k, a) $-generalized Fourier transform
  has been proposed.
  In the case $ N = 1 $, they~\cite[Definition~2.1, Theorem~2.1]{arXiv2507-04064}
  considered the space
  \begin{align*}
    & \widetilde{\mscrS}_{k, n}(\setR) \\
    &= \set*{f \in C^\infty(\setR \setminus \setenum{0})}{\text{%
      $ \sup_{x \in \setR \setminus \setenum{0}}
      \abs{(\abs{x}^{2/n})^\alpha (\abs{x}^{2 - 2/n} \Laplacian_k)^\beta (x^l f^{(l)}(x))} < \infty $
      for all $ \alpha $, $ \beta $, $ l \in \setN $%
    }}
  \end{align*}
  ($ \mscrS_{k, n}(\setR) $ in their notation) and
  seem to claim that this space is stable under
  the one-dimensional $ (k, 2/n) $-generalized Fourier transform $ \mscrF_{k, 2/n} $.

  However, this claim does not hold even in the classical case $ (k, 2/n) = (0, 2) $.
  As a counterexample, while the function $ x \mapsto e^{-\abs{x}} $
  belongs to $ \widetilde{\mscrS}_{0, 1}(\setR) $,
  its Fourier transform $ \xi \mapsto \sqrt{2/\pi}\, (1 + \xi^2)^{-1} $
  does not belong to $ \widetilde{\mscrS}_{0, 1}(\setR) $.
  In fact, they use the intertwining properties
  \begin{align*}
    \mscrF_{k, a} \circ \enorm{x}^a
    = -\enorm{x}^{2 - a} \Laplacian_k \circ \mscrF_{k, a}, \qquad
    \mscrF_{k, a} \circ \enorm{x}^{2 - a} \Laplacian_k
    = -\enorm{x}^a \circ \mscrF_{k, a}
  \end{align*}
  proved in \cite[Theorem~5.6]{MR2956043}, which hold on the dense subspace
  $ W_{k, a}(\setR^N) $ (see \cref{thm:unitary-representation} for the definition)
  in the usual sense and
  on $ L^2(\setR^N, w_{k, a}(x) \,dx) $ in the distribution sense.
  However, the proof there~\cite[Proof of Theorem~2.1]{arXiv2507-04064} use these identities
  beyond $ W_{k, a}(\setR^N) $ in the usual sense,
  rather than in the distribution sense.
\end{remark}

\section{The Schwartz space for the minimal representation of the conformal group}
\label{sec:conformal-group}

\subsection{\texorpdfstring{$ L^2 $}{L2}-model for the Segal--Shale--Weil representation}

In this subsection, we recall basic properties of the $ L^2 $-model for the Segal--Shale--Weil representation.
We adopt the same normalization as in Kobayashi--Mano~\cite{MR2401813}.
For details on the Segal--Shale--Weil representation, we refer the reader to Folland~\cite{MR983366}
(a different normalization is adopted there).

Let $ \mathit{Mp}(N, \setR) $ denote the metaplectic group,
or the double covering group of the symplectic group $ \mathit{Sp}(N, \setR) $.
Its Lie algebra of $ \mathit{Mp}(N, \setR) $ is
\[
  \mathfrak{sp}(N, \setR)
  = \set*{
      \begin{pmatrix} A & B \\ C & -\transpose{A} \end{pmatrix}
    }{\text{%
      $ A \in \mathfrak{gl}(N, \setR) $, $ B $ and $ C $ are real symmetric%
    }}.
\]

The \emph{$ L^2 $-model for the Segal--Shale--Weil representation},
for which we write $ \Pi_2 $,
is a unitary representation of $ \mathit{Mp}(N, \setR) $ on $ L^2(\setR^N) $.
The infinitesimal representation $ d\Pi_2 $ is given by
\begin{align*}
  d\Pi_2\Paren*{\begin{pmatrix} A & 0 \\ 0 & -\transpose{A} \end{pmatrix}}
  &= \sum_{j, k = 1}^{N} A_{jk} x_j \Pdif{x_k} + \frac{1}{2} \tr(A), \\
  d\Pi_2\Paren*{\begin{pmatrix} 0 & B \\ 0 & 0 \end{pmatrix}}
  &= \frac{i}{2} \sum_{j, k = 1}^{N} B_{jk} x_j x_k, \\
  d\Pi_2\Paren*{\begin{pmatrix} 0 & 0 \\ C & 0 \end{pmatrix}}
  &= \frac{i}{2} \sum_{j, k = 1}^{N} C_{jk} \frac{\partial^2}{\partial x_j \partial x_k}.
\end{align*}
Here, each differential operator $ D $ on the right-hand side is
understood in the distribution sense,
and its domain consists of all $ F \in L^2(\setR^N) $ such that $ DF \in L^2(\setR^N) $.

Let us consider the $ \mathfrak{sl}_2 $-triple
$ (\begin{psmallmatrix} I_N & 0 \\ 0 & -I_N \end{psmallmatrix},
\begin{psmallmatrix} 0 & I_N \\ 0 & 0 \end{psmallmatrix},
\begin{psmallmatrix} 0 & 0 \\ I_N & 0 \end{psmallmatrix}) $
in $ \mathfrak{sp}(N, \setR) $.
By the above formulas, we have
\begin{alignat*}{3}
  & d\Pi_2 \Paren*{\begin{pmatrix} I_N & 0 \\ 0 & -I_N \end{pmatrix}}
  &\mbox{}=\mbox{} & E_x + \frac{N}{2}
  &\mbox{}=\mbox{} & d\Omega_{0, 2}(\mbfh), \\
  & d\Pi_2 \Paren*{\begin{pmatrix} 0 & I_N \\ 0 & 0 \end{pmatrix}}
  &\mbox{}=\mbox{} & \frac{i}{2} \enorm{x}^2
  &\mbox{}=\mbox{} & d\Omega_{0, 2}(\mbfe^+), \\
  & d\Pi_2 \Paren*{\begin{pmatrix} 0 & 0 \\ I_N & 0 \end{pmatrix}}
  &\mbox{}=\mbox{} & \frac{i}{2} \Laplacian
  &\mbox{}=\mbox{} & d\Omega_{0, 2}(\mbfe^-).
\end{alignat*}
Hence, up to the choice of the covering,
$ \Omega_{0, 2} $ coincides with the restriction of $ \Pi_2 $
to the subgroup generated by
$ \begin{psmallmatrix} I_N & 0 \\ 0 & -I_N \end{psmallmatrix} $,
$ \begin{psmallmatrix} 0 & I_N \\ 0 & 0 \end{psmallmatrix} $ and
$ \begin{psmallmatrix} 0 & 0 \\ I_N & 0 \end{psmallmatrix} $.

\subsection{The classical Schwartz space}

The following theorem is classical,
but we include this for comparison with \cref{thm:schwartz-space-conformal} stated later.

\begin{theorem}
  The following spaces coincide.
  \begin{enumalphp}
    \item the space of smooth vectors for the Segal--Shale--Weil representation $ \Pi_2 $
    \item $ \mscrS_{0, 2}(\setR^N) $,
          or the space of smooth vectors for $ \Omega_{0, 2} $
    \item the space of smooth vectors for $ \restr{\Omega_{0, 2}}{\widetilde{\mathit{SO}}(2)} $
    \item the classical Schwartz space $ \mscrS(\setR^N) $
  \end{enumalphp}
\end{theorem}

\begin{proof}
  \begin{subproof}{$ \text{(a)} \subseteq \text{(b)} $}
    It follows from the fact that
    $ \Omega_{0, 2} $ coincides with the restriction of $ \Pi_2 $
    to a subgroup up to the choice of the covering.
  \end{subproof}

  \begin{subproof}{$ \text{(b)} = \text{(c)} = \text{(d)} $}
    It follows from \cref{thm:schwartz-space-as-infinite-sum,thm:schwartz-space}.
  \end{subproof}

  \begin{subproof}{$ \text{(d)} \subseteq \text{(a)} $}
    By the formula of $ d\Pi_2 $ given in the previous subsection,
    we see that, for $ X \in \mathfrak{sp}(2, \setR) $ of the form
    $ \begin{psmallmatrix} A & 0 \\ 0 & -\transpose{A} \end{psmallmatrix} $,
    $ \begin{psmallmatrix} 0 & B \\ 0 & 0 \end{psmallmatrix} $ or
    $ \begin{psmallmatrix} 0 & 0 \\ C & 0 \end{psmallmatrix} $,
    $ \mscrS(\setR^N) $ is contained in the domain of $ d\Pi_2(X) $ and
    stable under this operator.
    Therefore, the assertion follows from \cref{thm:condition-to-be-a-smooth-vector}.
  \end{subproof}
\end{proof}

\subsection{\texorpdfstring{$ L^2 $}{L2}-model for the minimal representation of the conformal group}
\label{ssec:minimal-representation-of-conformal-group}

In this subsection,
we recall basic properties of the $ L^2 $-model for the minimal representation
of the conformal group studied by Kobayashi--Mano~\cite{MR2401813}.

Let $ N \geq 2 $ and consider the identity component $ \mathit{SO}_0(N + 1, 2) $
of the conformal group $ O(N + 1, 2) $ and
its maximal compact subgroup $ \mathit{SO}(N + 1) \times \mathit{SO}(2) $.
Let $ \widetilde{\mathit{SO}}_0(N + 1, 2) $ denote the double covering group
of $ \mathit{SO}_0(N + 1, 2) $
corresponding to the double covering group
$ K = \mathit{SO}(N + 1) \times \widetilde{\mathit{SO}}(2)^{(2)} $
of $ \mathit{SO}(N + 1) \times \mathit{SO}(2) $.
Here, $ \widetilde{\mathit{SO}}(2)^{(2)} $ denotes the double covering group
of $ \mathit{SO}(2) $.
The Lie algebra of $ \widetilde{\mathit{SO}}_0(N + 1, 2) $ is
\[
  \mfrako(N + 1, 2)
  = \set{X \in \mathfrak{gl}(N + 3, \setR)}{
      \transpose{X} I_{N + 1, 2} + I_{N + 1, 2} X = 0
    },
\]
where $ I_{N + 1, 2} = \begin{psmallmatrix} I_{N + 1} & 0 \\ 0 & -I_2 \end{psmallmatrix} $.
Following Kobayashi--Mano~\cite{MR2401813},
we use indices $ 0 $, $ 1 $, $ \dots $, $ N + 2 $ for coordinates on $ \setR^{N + 3} $,
and set
\begin{align*}
  & \mfrakm = \set*{
      \begin{pmatrix}
        0      & \cdots & 0      \\
        \vdots & X      & \vdots \\
        0      & \cdots & 0
      \end{pmatrix}
    }{
      X \in \mfrako(N, 1)
    }, \\
  & \mfraka = \setR \mbfE, &
    & \mbfE = \mbfE_{0, N + 2} + \mbfE_{N + 2, 0}, \\
  & \overline{\mfrakn} = \lspan_{\setR} \set{\overline{\mbfN}_j}{j \in \setenum{1, \dots, N + 1}}, &
    & \overline{\mbfN}_j = \mbfE_{j, 0} + \mbfE_{j, N + 2} - \epsilon_j \mbfE_{0, j} + \epsilon_j \mbfE_{N + 2, j}, \\
  & \mfrakn = \lspan_{\setR} \set{\mbfN_j}{j \in \setenum{1, \dots, N + 1}}, &
    & \mbfN_j = \mbfE_{j, 0} - \mbfE_{j, N + 2} - \epsilon_j \mbfE_{0, j} - \epsilon_j \mbfE_{N + 2, j},
\end{align*}
where $ \mbfE_{jk} $ denotes the matrix unit and
\[
  \epsilon_j
  = \begin{cases}
      1  & (j \in \setenum{1, \dots, N}) \\
      -1 & (j = N + 1).
    \end{cases}
\]
Then, $ \mfrako(N + 1, 2) = \overline{\mfrakn} \oplus \mfrakm \oplus \mfraka \oplus \mfrakn $,
and $ \mfrakm \oplus \mfraka \oplus \mfrakn $ is the Langlands decomposition
of a maximal parabolic subalgebra of $ \mfrako(N + 1, 2) $.

The \emph{$ L^2 $-model for the minimal representation of $ \widetilde{\mathit{SO}}_0(N + 1, 2) $},
for which we write $ \Pi_1 $,
is a unitary representation of $ \widetilde{\mathit{SO}}_0(N + 1, 2) $
on $ L^2(\setR^N, \enorm{x}^{-1} \,dx) $.
The formulas in \cite[Section~2.4]{MR2401813} and a simple computation show that
the infinitesimal representation $ d\Pi_1 $ is given by
\begin{align*}
  d\Pi_1(X)
  &= \sum_{j = 1}^{N}
    (X_{1, j} x_1 + \dots + X_{N, j} x_N + X_{N + 1, j} \enorm{x}) \Pdif{x_j}
    \quad (X \in \mfrakm), \\
  d\Pi_1(\mbfE)
  &= -E_x - \frac{N - 1}{2}, \\
  d\Pi_1(\overline{\mbfN}_j)
  &= \begin{cases}
      2i x_j       & (j \in \setenum{1, \dots, N}) \\
      2i \enorm{x} & (j = N + 1),
    \end{cases} \\
  d\Pi_1(\mbfN_j)
  &= \begin{dcases}
      \frac{i}{2} \Paren*{x_j \Laplacian - (2E_x + N - 1) \Pdif{x_j}} & (j \in \setenum{1, \dots, N}) \\
      \frac{i}{2} \enorm{x} \Laplacian                                & (j = N + 1)
    \end{dcases}
\end{align*}
on the space $ L^2(\setR^N, \enorm{x}^{-1} \,dx)_K $ of $ K $-finite vectors.
As stated in \cite[Section~3.2]{MR2401813}, the space of $ K $-finite vectors is
\begin{align*}
  L^2(\setR^N, \enorm{x}^{-1} \,dx)_K
  &= \bigoplus_{m \in \setN}
    \mcalH^m(\Sphere{N - 1}) \otimes
    \lspan_{\setC} \set{r \mapsto r^m L_l^{2m + N - 2}(4r) e^{-2r}}{l \in \setN} \\
  &= \Phi(W_{0, 1}(\setR^N)),
\end{align*}
where $ W_{0, 1}(\setR^N) $ is as defined in \cref{thm:unitary-representation},
and $ \Phi $ is the unitary operator defined by
\begin{align}
  & \map{\Phi}{L^2(\setR^N, \enorm{x}^{-1} \,dx)}{L^2(\setR^N, \enorm{x}^{-1} \,dx)}
    \notag \\
  & \Phi F(x) = 2^{N - 1/2} F(2x).
    \label{eq:scaling}
\end{align}

Let us consider the $ \mathfrak{sl}_2 $-triple
$ (-2\mbfE, \overline{\mbfN}_{N + 1}, \mbfN_{N + 1}) $ in $ \mfrako(N + 1, 2) $.
By the above formulas and the same argument as in \cref{thm:infinitesimal-representation},
we have
\begin{alignat*}{3}
  & d\Pi_1(-2\mbfE)
  &\mbox{}=\mbox{} & \closure{\restr*{\Paren*{2E_x + \frac{N - 1}{2}}}{\Phi(W_{0, 1}(\setR^N))}}
  &\mbox{}=\mbox{} & \Phi \circ \closure{\restr{\omega_{0, 1}(\mbfh)}{W_{0, 1}(\setR^N)}} \circ \Phi^{-1}, \\
  & d\Pi_1(\overline{\mbfN}_{N + 1})
  &\mbox{}=\mbox{} & \closure{\restr{(i \enorm{x})}{\Phi(W_{0, 1}(\setR^N))}}
  &\mbox{}=\mbox{} & \Phi \circ \closure{\restr{\omega_{0, 1}(\mbfe^+)}{W_{0, 1}(\setR^N)}} \circ \Phi^{-1}, \\
  & d\Pi_1(\mbfN_{N + 1})
  &\mbox{}=\mbox{} & \closure{\restr{(i \enorm{x} \Laplacian)}{\Phi(W_{0, 1}(\setR^N))}}
  &\mbox{}=\mbox{} & \Phi \circ \closure{\restr{\omega_{0, 1}(\mbfe^-)}{W_{0, 1}(\setR^N)}} \circ \Phi^{-1}.
\end{alignat*}
Comparing this with \cref{thm:infinitesimal-representation}, we obtain
\begin{align*}
  d\Pi_1(-2\mbfE)
  &= \Phi \circ d\Omega_{0, 1}(\mbfh) \circ \Phi^{-1}, \\
  d\Pi_1(\overline{\mbfN}_{N + 1})
  &= \Phi \circ d\Omega_{0, 1}(\mbfe^+) \circ \Phi^{-1}, \\
  d\Pi_1(\mbfN_{N + 1})
  &= \Phi \circ d\Omega_{0, 1}(\mbfe^-) \circ \Phi^{-1}.
\end{align*}
Hence, up to the choice of the covering,
$ \Phi \circ \Omega_{0, 1}(\blank) \circ \Phi^{-1} $ coincides
with the restriction of $ \Pi_2 $
to the subgroup generated by $ -2\mbfE $, $ \overline{\mbfN}_{N + 1} $ and $ \mbfN_{N + 1} $.

\subsection{The skew-adjoint operators \texorpdfstring{$ d\Pi_1(\overline{\mbfN}_j) $}{dΠ1(Nj)} and \texorpdfstring{$ d\Pi_1(\mbfN_j) $}{dΠ1(Nj)}}

The following propositions are used in the next subsection.

\begin{proposition}\label{thm:operator-nbar}
  Let $ N \geq 2 $.
  For $ j \in \setenum{1, \dots, N} $, we have
  \[
    \Dom(d\Pi_1(\overline{\mbfN}_j))
    = \set{F \in L^2(\setR^N, \enorm{x}^{-1} \,dx)}{x_j F \in L^2(\setR^N, \enorm{x}^{-1} \,dx)}
  \]
  and
  \[
    d\Pi_1(\overline{\mbfN}_j) F(x)
    = 2i x_j F(x).
  \]
\end{proposition}

\begin{proof}
  The multiplication operator $ 2i x_j $,
  defined on the domain given in the assertion, is skew-adjoint
  by general theory about operators (see \cite[Example~3.8]{MR2953553}).
  Now the assertion follows from the same argument as in \cref{thm:criterion}.
\end{proof}

\begin{proposition}\label{thm:operator-n}
  Let $ N \geq 2 $.
  For $ j \in \setenum{1, \dots, N} $, we have
  \[
    d\Pi_1(\overline{\mbfN}_j)
    = \closure{\restr*{\frac{i}{2} \Paren*{x_j \Laplacian - (2E_x + N - 1) \Pdif{x_j}}}{\mscrS(\setR^N) + \enorm{x} \mscrS(\setR^N)}}.
  \]
\end{proposition}

\begin{proof}
  Let $ \map{\Phi}{L^2(\setR^N, \enorm{x}^{-1} \,dx)}{L^2(\setR^N, \enorm{x}^{-1} \,dx)} $
  denote the unitary operator defined by \cref{eq:scaling}.

  Let $ D_j $ denote the differential operator
  $ x_j \Laplacian - (2E_x + N - 1) \Pdif{x_j} $.
  Notice that
  \[
    \Phi(W_{0, 1}(\setR^N))
    \subseteq \mscrS(\setR^N) + \enorm{x} \mscrS(\setR^N)
    \subseteq L^2(\setR^N, \enorm{x}^{-1} \,dx) \cap C^\infty(\setR^N \setminus \setenum{0})
  \]
  and $ \mscrS(\setR^N) + \enorm{x} \mscrS(\setR^N) $ is stable under $ D_j $.
  Hence, by the same argument as in \cref{thm:criterion},
  it suffices to show that $ \restr{D_j}{\mscrS(\setR^N) + \enorm{x} \mscrS(\setR^N)} $
  is a symmetric operator on $ L^2(\setR^N, \enorm{x}^{-1} \,dx) $.

  Let $ F $ and $ G $ be elements of $ \mscrS(\setR^N) $ or $ \enorm{x} \mscrS(\setR^N) $.
  By Green's second identity, we have
  \begin{align}
    & \int_{\epsilon \leq \enorm{x} \leq R}
      \Paren*{
        x_j \Laplacian F(x) G(x)
        - F(x) \Paren*{x_j \Laplacian - \frac{x_j}{\enorm{x}^2} (2E_x + N - 1) + 2 \Pdif{x_j}} G(x)
      } \enorm{x}^{-1} \,dx
      \notag \\
    &= \int_{\epsilon \leq \enorm{x} \leq R}
      \Paren*{
        \Laplacian F(x) \frac{x_j G(x)}{\enorm{x}}
        - F(x) \Laplacian \Paren*{\frac{x_j G(x)}{\enorm{x}}}
      } \,dx
      \notag \\
    &= \Brack*{
        \frac{1}{r} \int_{\enorm{x} = r}
        (E_x F(x) G(x) - F(x) E_x G(x)) \frac{x_j}{\enorm{x}} \,dx
      }_{\epsilon}^{R}
      \notag \\
    &= \Brack*{
        r^{N - 2} \int_{\Sphere{N - 1}}
        (E_x F(r\omega) G(r\omega) - F(r\omega) E_x G(r\omega)) \omega_j \,d\omega
      }_{\epsilon}^{R},
      \label{eq:operator-n-1}
  \end{align}
  where $ \omega_j $ denotes the $ j $-th coordinate of $ \omega $.
  Setting $ \widetilde{F} = \Pdiff{F}{x_j} $, by the divergence theorem, we have
  \begin{align}
    & \int_{\epsilon \leq \enorm{x} \leq R}
      \Paren*{
        (2E_x + N - 1) \widetilde{F}(x) G(x)
        + \widetilde{F}(x) (2E_x + N - 1) G(x)
      } \enorm{x}^{-1} \,dx
      \notag \\
    &= \int_{\epsilon \leq \enorm{x} \leq R}
      2 \divg\Paren*{
        \widetilde{F}(x) G(x) \frac{x}{\enorm{x}}
      } \,dx
      \notag \\
    &= \Brack*{
        2 \int_{\enorm{x} = r} \widetilde{F}(x) G(x) \,dx
      }_{\epsilon}^{R}
      \notag \\
    &= \Brack*{
        2r^{N - 1}
        \int_{\Sphere{N - 1}} \widetilde{F}(r\omega) G(r\omega) \,d\omega
      }_{\epsilon}^{R}
      \label{eq:operator-n-2}
  \end{align}
  Setting $ \widetilde{G} = (2E_x + N - 1) G $, by the divergence theorem, we have
  \begin{align}
    & \int_{\epsilon \leq \enorm{x} \leq R}
      \Paren*{
        \Pdiff{F}{x_j}(x) \widetilde{G}(x)
        + F(x) \Paren*{\Pdif{x_j} - \frac{x_j}{\enorm{x}^2}} \widetilde{G}(x)
      } \enorm{x}^{-1} \,dx
      \notag \\
    &= \int_{\epsilon \leq \enorm{x} \leq R}
      \divg\Paren*{F(x) \widetilde{G}(x) \frac{e_j}{\enorm{x}}} \,dx
      \notag \\
    &= \Brack*{
        \int_{\enorm{x} = r} F(x) \widetilde{G}(x) \frac{x_j}{\enorm{x}^2} \,dx
      }_{\epsilon}^{R}
      \notag \\
    &= \Brack*{
        r^{N - 2}
        \int_{\Sphere{N - 1}} F(r\omega) \widetilde{G}(r\omega) \omega_j \,d\omega
      }_{\epsilon}^{R},
      \label{eq:operator-n-3}
  \end{align}
  where $ e_j $ denotes the unit vector in the $ j $-th direction.

  Since $ F $ and $ G $ are rapidly decreasing at infinity,
  the boundary terms \cref{eq:operator-n-1,eq:operator-n-2,eq:operator-n-3}
  for $ r = R $ vanish as $ R \to \infty $.
  Let us consider the boundary terms for $ r = \epsilon $.
  A simple computation shows that $ F(x) = O(1) $, $ \Pdiff{F}{x_j}(x) = O(1) $ and
  $ E_x F(x) = O(\enorm{x}) $ as $ x \to 0 $, and the same estimates hold for $ G $.
  Hence, the boundary terms \cref{eq:operator-n-1,eq:operator-n-2} for $ r = \epsilon $
  vanish as $ \epsilon \to 0 $.
  On the other hand, the boundary term \cref{eq:operator-n-3} for $ r = \epsilon $ is
  \begin{align*}
    & \epsilon^{N - 2} \int_{\Sphere{N - 1}}
      F(\epsilon\omega) \widetilde{G}(\epsilon\omega) \omega_j \,d\omega \\
    &= \epsilon^{N - 2} \int_{\set{\omega \in \Sphere{N - 1}}{\omega_j \geq 0}}
      (F(\epsilon\omega) \widetilde{G}(\epsilon\omega) - F(-\epsilon\omega) \widetilde{G}(-\epsilon\omega))
      \omega_j \,d\omega,
  \end{align*}
  which vanishes as $ \epsilon \to 0 $
  since $ F $ and $ \widetilde{G} = (2E_x + N - 1) G $ are continuous at the origin.
  Therefore, all the boundary terms vanish as $ \epsilon \to 0 $ and $ R \to \infty $.

  Taking the sum of both sides of \cref{eq:operator-n-1,eq:operator-n-2,eq:operator-n-3}
  and passing to the limit $ \epsilon \to 0 $ and $ R \to \infty $,
  we obtain
  \[
    \innprod{D_j F}{G}_{L^2(\setR^N, \enorm{x}^{-1} \,dx)}
      - \innprod{F}{D_j G}_{L^2(\setR^N, \enorm{x}^{-1} \,dx)}
    = 0.
  \]
  This completes the proof.
\end{proof}

\subsection{The Schwartz space for the minimal representation of the conformal group}

In this subsection,
we define and determine the Schwartz space for the minimal representation of the conformal group.
This is one of the main results of this paper.

\begin{definition}\label{def:schwartz-space-conformal}
  Let $ N \geq 2 $.
  We define the \emph{Schwartz space for the minimal representation of the conformal group} as
  \[
    \mscrS_{\mathrm{conf}}(\setR^N)
    = L^2(\setR^N, \enorm{x}^{-1} \,dx)_{\Pi_1}^\infty,
  \]
  or the space of smooth vectors for the minimal representation $ \Pi_1 $.
\end{definition}

\begin{theorem}\label{thm:schwartz-space-conformal}
  Let $ N \geq 2 $.
  The following spaces coincide.
  \begin{enumalphp}
    \item $ \mscrS_{\mathrm{conf}}(\setR^N) $,
          or the space of smooth vectors for the minimal representation $ \Pi_1 $
    \item $ \mscrS_{0, 1}(\setR^N) $,
          or the space of smooth vectors for $ \Omega_{0, 1} $
    \item the space of smooth vectors for $ \restr{\Omega_{0, 1}}{\widetilde{\mathit{SO}}(2)} $
    \item $ \mscrS(\setR^N) + \enorm{x} \mscrS(\setR^N) $
  \end{enumalphp}
\end{theorem}

\begin{proof}
  Let $ \map{\Phi}{L^2(\setR^N, \enorm{x}^{-1} \,dx)}{L^2(\setR^N, \enorm{x}^{-1} \,dx)} $
  denote the unitary operator defined by \cref{eq:scaling}.

  \begin{subproof}{$ \text{(a)} \subseteq \text{(b)} $}
    It follows from the fact that
    $ \Phi \circ \Omega_{0, 1}(\blank) \circ \Phi^{-1} $ coincides
    with the restriction of $ \Pi_1 $
    to a subgroup up to the choice of the covering.
  \end{subproof}

  \begin{subproof}{$ \text{(b)} = \text{(c)} = \text{(d)} $}
    It follows from \cref{thm:schwartz-space-as-infinite-sum,thm:schwartz-space}.
  \end{subproof}
  
  \begin{subproof}{$ \text{(d)} \subseteq \text{(a)} $}
    Notice that $ \overline{\mfrakn} \oplus \mfrakn
    = \set{\overline{\mbfN}_j, \mbfN_j}{j \in \setenum{1, \dots, N + 1}} $ generates
    the Lie algebra $ \mathfrak{o}(N + 1, 2) $.
    Hence, by \cref{thm:infinitesimal-lie-algebra-representation,thm:condition-to-be-a-smooth-vector},
    it suffices to show that $ \mscrS(\setR^N) + \enorm{x} \mscrS(\setR^N) $ is
    contained in the domains of $ d\Pi_1(\overline{\mbfN}_j) $ and
    $ d\Pi_1(\mbfN_j) $ and
    stable under these operators.

    The stability under
    \begin{align*}
      d\Pi_1(\overline{\mbfN}_{N + 1})
      &= \Phi \circ d\Omega_{0, 1}(\mbfe^+) \circ \Phi^{-1}, \\
      d\Pi_1(\mbfN_{N + 1})
      &= \Phi \circ d\Omega_{0, 1}(\mbfe^-) \circ \Phi^{-1}
    \end{align*}
    is included in \cref{thm:schwartz-space}.
    We next prove the stability under $ d\Pi_1(\overline{\mbfN}_j) $ and $ d\Pi_1(\mbfN_j) $.
    A simple computation shows that $ \mscrS(\setR^N) + \enorm{x} \mscrS(\setR^N) $ is
    stable under the differential operators
    \[
      2i x_j, \qquad
      \frac{i}{2} \Paren*{x_j \Laplacian - (2E_x + N - 1) \Pdif{x_j}}
    \]
    for $ j \in \setenum{1, \dots, N} $.
    Moreover, by \cref{thm:operator-nbar,thm:operator-n},
    $ d\Pi_1(\overline{\mbfN}_j) $ and $ d\Pi_1(\mbfN_j) $ are extensions of
    these operators restricted to $ \mscrS(\setR^N) + \enorm{x} \mscrS(\setR^N) $, respectively.
    This completes the proof.
  \end{subproof}
\end{proof}

\appendix
\section{Laguerre polynomials}
\label{sec:laguerre-polynomial}

\subsection{Definition and basic properties}

For $ \lambda \in \setC $ and $ l \in \setN $, we write the Laguerre polynomial as
\[
  L_l^\lambda(t)
  = \sum_{j = 0}^{l} \frac{(-1)^j}{j!} \binom{\lambda + l}{l - j} t^j
  = \sum_{j = 0}^{l} \frac{(-1)^j}{j! (l - j)!} \frac{\Gamma(\lambda + l + 1)}{\Gamma(\lambda + j + 1)} t^j.
\]
For convenience, we set $ L_l^\lambda = 0 $ for negative integers $ l $.

We note some basic properties about Laguerre polynomials.

\begin{fact}[{\cite[8.971~4, 8.971~2]{MR3307944}}]\label{thm:properties-of-laguerre-polynomials}
  Let $ \lambda \in \setC $ and $ l \in \setN $.
  \begin{enumarabicp}
    \item $ t L_l^\lambda(t) = (\lambda + l) L_l^{\lambda - 1}(t) + (l + 1) L_{l + 1}^{\lambda - 1}(t) $.
    \item $ \Odif{t} L_l^\lambda(t) = -L_{l - 1}^{\lambda + 1}(t) $.
  \end{enumarabicp}
\end{fact}

\subsection{Estimates for Laguerre polynomials}

The following estimate for Laguerre polynomials is due to Duran.

\begin{fact}[{\cite[Theorem~1]{MR1121714}}]\label{thm:estimate-1}
  For $ l $, $ \beta $, $ \gamma \in \setN $ and $ \lambda \geq -l - 1 $,
  we have
  \[
    \sup_{t \geq 0} \abs*{t^\beta \Paren*{\Odif{t}}^\gamma (L_l^\lambda(t) e^{-t/2})}
    \leq 2^{\beta + \max\setenum{\beta - \lambda, 0}} (l + \beta)^{\underline{\beta}}
      \binom{\max\setenum{\lambda - \beta, 0} + l + \gamma}{l}.
  \]
\end{fact}

\begin{remark}
  In Duran's paper, it is assumed that $ \lambda \geq 0 $,
  but \cref{thm:estimate-1} holds more generally for $ \lambda \geq -l - 1 $.
  Indeed, the assumption $ \lambda \geq 0 $ is used only in inequality~(6) in the paper,
  which remains valid for $ \lambda \geq -l - 1 $.
\end{remark}

We next show another estimate for Laguerre polynomials,
which we could not find in the literature.
The key to our estimate is the following identity due to Koornwinder.

\begin{fact}[{\cite[Remark~4.1]{MR454106}}]\label{thm:koornwinder}
  For $ \lambda > 0 $, $ l \in \setN $ and $ t \geq 0 $, we have
  \[
    \sum_{i = 0}^{\infty} \sum_{j = 0}^{l}
      \binom{l}{j} \frac{\lambda}{\lambda + i + j}
      \frac{(\lambda + l + 1)^{\overline{i}}}{i! (\lambda + j)^{\overline{i}} (\lambda + i)^{\overline{j}}}
      \Paren*{t^{i + j} \frac{L_{l - j}^{\lambda + i + j}(t^2)}{{L_{l - j}^{\lambda + i + j}(0)}} e^{-t^2 /2}}^2
    = 1.
  \]
\end{fact}

\begin{proposition}\label{thm:estimate-2}
  For $ l $, $ \alpha $, $ \beta $, $ \gamma \in \setN $ and
  $ \lambda \geq \alpha + \beta $, we have
  \begin{align*}
    & \sup_{t \geq 0} \abs*{t^{\alpha/2 + \beta} \Paren*{\Odif{t}}^\gamma (L_l^\lambda(t) e^{-t/2})} \\
    &\leq 2^\beta (\lambda + l)^{\underline{\beta}}
      \binom{\lambda + l + \gamma}{l + \beta}
      (\alpha!)^{1/2} \binom{\lambda - \beta + \gamma}{\alpha}^{1/2} \binom{\lambda + l}{\alpha}^{-1/2}.
  \end{align*}
\end{proposition}

\begin{proof}
  We divide the proof into three steps.

  \begin{subproof}{\textbf{Step~1: the case $ \beta = \gamma = 0 $.}}
    Since all terms on the left-hand side of the formula in \cref{thm:koornwinder}
    are non-negative, each of them is bounded above by $ 1 $.
    In particular, for $ \lambda > 0 $ and $ l $, $ \alpha \in \setN $
    (we replace $ i $ with $ \alpha $), we have
    \[
      \frac{\lambda}{\lambda + \alpha} \frac{(\lambda + l + 1)^{\overline{\alpha}}}{\alpha! \lambda^{\overline{\alpha}}}
        \Paren*{t^\alpha \frac{L_l^{\lambda + \alpha}(t^2)}{{L_l^{\lambda + \alpha}(0)}} e^{-t^2 /2}}^2
      \leq 1.
    \]
    Since $ L_l^{\lambda + \alpha}(0) = \binom{\lambda + \alpha + l}{l} $, we have
    \begin{align*}
      \abs{t^\alpha L_l^{\lambda + \alpha}(t^2) e^{-t^2 /2}}
      &\leq \binom{\lambda + \alpha + l}{l} \Paren*{
        \frac{\lambda + \alpha}{\lambda}
        \frac{\alpha! \lambda^{\overline{\alpha}}}{(\lambda + l + 1)^{\overline{\alpha}}}
      }^{1/2} \\
      &\leq \binom{\lambda + \alpha + l}{l} \Paren*{
        \frac{\alpha! (\lambda + \alpha)^{\underline{\alpha}}}{(\lambda + \alpha + l)^{\underline{\alpha}}}
      }^{1/2} \\
      &\leq \binom{\lambda + \alpha + l}{l} (\alpha!)^{1/2} \binom{\lambda + \alpha}{\alpha}^{1/2} \binom{\lambda + \alpha + l}{\alpha}^{-1/2}.
    \end{align*}
    By replacing $ t^2 $ with $ t $ and $ \lambda + \alpha $ with $ \lambda $
    (so that the new $ \lambda $ is required to be larger than $ \alpha $),
    we obtain
    \[
      \abs{t^{\alpha/2} L_l^\lambda(t) e^{-t/2}}
      \leq \binom{\lambda + l}{l} (\alpha!)^{1/2} \binom{\lambda}{\alpha}^{1/2} \binom{\lambda + l}{\alpha}^{-1/2}.
    \]
    By taking the limit, we see that this inequality also holds when $ \lambda = \alpha $.
  \end{subproof}

  \begin{subproof}{\textbf{Step~2: the case $ \gamma = 0 $.}}
    Using \cref{thm:properties-of-laguerre-polynomials}~(1) recursively, we have
    \[
      t^\beta L_l^\lambda(t)
      = \sum_{j = 0}^{\beta}
        (-1)^j \binom{\beta}{j} (l + 1)^{\overline{j}} (\lambda + l)^{\underline{\beta - j}}
        L_{l + j}^{\lambda - \beta}(t).
    \]
    Since $ \lambda \geq \alpha + \beta \geq \beta $, we have
    $ (l + 1)^{\overline{j}} (\lambda + l)^{\underline{\beta - j}}
    \leq (\lambda + l)^{\underline{k}} $, and hence
    \begin{align*}
      \abs{t^\beta L_l^\lambda(t)}
      &\leq \sum_{j = 0}^{\beta}
        \binom{\beta}{j} (l + 1)^{\overline{j}} (\lambda + l)^{\underline{\beta - j}}
        \abs{L_{l + j}^{\lambda - \beta}(t)} \\
      &\leq 2^\beta (\lambda + l)^{\underline{\beta}}
        \sup_{0 \leq j \leq \beta} \abs{L_{l + j}^{\lambda - \beta}(t)}.
    \end{align*}
    By the above inequality and Step~1 (here, we use the assumption $ \lambda \geq \alpha + \beta $),
    we obtain
    \begin{align*}
      & \abs{t^{\alpha/2 + \beta} L_l^\lambda(t)} \\
      &\leq 2^\beta (\lambda + l)^{\underline{\beta}}
        \sup_{0 \leq j \leq \beta} \abs{t^{\alpha/2} L_{l + j}^{\lambda - \beta}(t)} \\
      &\leq 2^\beta (\lambda + l)^{\underline{\beta}}
        \sup_{0 \leq j \leq \beta}
        \binom{\lambda - \beta + l + j}{l + j}
        (\alpha!)^{1/2} \binom{\lambda - \beta}{\alpha}^{1/2} \binom{\lambda - \beta + l + j}{\alpha}^{-1/2} \\
      &\leq 2^\beta (\lambda + l)^{\underline{\beta}}
        \binom{\lambda + l}{l + \beta}
        (\alpha!)^{1/2} \binom{\lambda - \beta}{\alpha}^{1/2} \binom{\lambda + l}{\alpha}^{-1/2}.
    \end{align*}
    Here, the last inequality holds because the function
    \[
      s \mapsto
      \binom{\lambda - \beta + s}{s} \binom{\lambda - \beta + s}{\alpha}^{-1/2}
      = \frac{(\alpha!)^{1/2}}{(\lambda - \beta)!}
        \frac{((\lambda - \beta + s)!)^{1/2} ((\lambda - \beta - \alpha + s)!)^{1/2}}{s!}
    \]
    is strictly increasing on $ (-1, \infty) $, as can be verified by elementary calculus.
  \end{subproof}

  \begin{subproof}{\textbf{Step~3: the general case.}}
    Using \cref{thm:properties-of-laguerre-polynomials}~(2) recursively, we have
    \begin{align*}
      \Paren*{\Odif{t}}^\gamma (L_l^\lambda(t) e^{-t/2})
      &= \sum_{j = 0}^{\gamma}
        \binom{\gamma}{j}
        \Paren*{\Odif{t}}^j L_l^\lambda(t) \cdot \Paren*{\Odif{t}}^{\gamma - j} e^{-t/2} \\
      &= (-1)^\gamma \sum_{j = 0}^{\gamma}
        2^{-(\gamma - j)} \binom{\gamma}{j}
        L_{l - j}^{\lambda + j}(t) e^{-t/2},
    \end{align*}
    and hence
    \[
      \abs*{\Paren*{\Odif{t}}^\gamma (L_l^\lambda(t) e^{-t/2})}
      \leq \sum_{j = 0}^{\gamma}
        2^{-(\gamma - j)} \binom{\gamma}{j}
        \abs{L_{l - j}^{\lambda + j}(t) e^{-t/2}}.
    \]
    By the above inequality and Step~2, we obtain
    \begin{align*}
      & \abs*{t^{\alpha/2 + \beta} \Paren*{\Odif{t}}^\gamma (L_l^\lambda(t) e^{-t/2})} \\
      &\leq \sum_{j = 0}^{\gamma}
        2^{-(\gamma - j)} \binom{\gamma}{j}
        \abs{t^{\alpha/2 + \beta} L_{l - j}^{\lambda + j}(t) e^{-t/2}} \\
      &\leq \sum_{j = 0}^{\gamma}
        2^{-(\gamma - j)} \binom{\gamma}{j}
        \cdot 2^\beta (\lambda + l)^{\underline{\beta}}
        \binom{\lambda + l}{l - j + \beta}
        (\alpha!)^{1/2} \binom{\lambda + j - \beta}{\alpha}^{1/2} \binom{\lambda + l}{\alpha}^{-1/2} \\
      &\leq 2^\beta (\lambda + l)^{\underline{\beta}}
        (\alpha!)^{1/2} \binom{\lambda - \beta + \gamma}{\alpha}^{1/2} \binom{\lambda + l}{\alpha}^{-1/2}
        \sum_{j = 0}^{\gamma} \binom{\gamma}{j} \binom{\lambda + l}{l - j + \beta} \\
      &= 2^\beta (\lambda + l)^{\underline{\beta}}
        \binom{\lambda + l + \gamma}{l + \beta}
        (\alpha!)^{1/2} \binom{\lambda - \beta + \gamma}{\alpha}^{1/2} \binom{\lambda + l}{\alpha}^{-1/2}.
    \end{align*}
    Here, we use $ 2^{-(\gamma - j)} \leq 1 $ and
    $ \binom{\lambda + j - \beta}{\alpha} \leq \binom{\lambda - \beta + \gamma}{\alpha} $
    in the last inequality.
  \end{subproof}
\end{proof}

\begin{remark}
  When $ \alpha = 0 $, \cref{thm:estimate-2} becomes
  \[
    \sup_{t \geq 0} \abs*{t^\beta \Paren*{\Odif{t}}^\gamma (L_l^\lambda(t) e^{-t/2})}
    \leq 2^\beta (\lambda + l)^{\underline{\beta}}
      \binom{\lambda + l + \gamma}{l + \beta},
  \]
  while Duran's estimate (\cref{thm:estimate-1}) states that
  \[
    \sup_{t \geq 0} \abs*{t^\beta \Paren*{\Odif{t}}^\gamma (L_l^\lambda(t) e^{-t/2})}
    \leq 2^{\beta} (l + \beta)^{\underline{\beta}}
      \binom{\lambda + l - \beta + \gamma}{l}.
  \]
  In this case, Duran's estimate is better since
  \begin{align*}
    (\lambda + l)^{\underline{\beta}} \binom{\lambda + l + \gamma}{l + \beta}
    &= \frac{(\lambda + l)^{\underline{\beta}} (\lambda + l + \gamma)^{\underline{\beta}}}{((l + \beta)^{\underline{\beta}})^2}
      (l + \beta)^{\underline{\beta}} \binom{\lambda + l - \beta + \gamma}{l} \\
    &\geq (l + \beta)^{\underline{\beta}} \binom{\lambda + l - \beta + \gamma}{l}.
  \end{align*}

  On the other hand, if we fix $ M \geq 0 $ and
  let $ \lambda $, $ l $ and $ \alpha $ vary under the condition
  $ \abs{\alpha - \lambda} \leq M $,
  then \cref{thm:estimate-2} yields that
  \begin{align*}
    & \sup_{t \geq 0} \abs*{t^{\alpha/2 + \beta} \Paren*{\Odif{t}}^\gamma (L_l^\lambda(t) e^{-t/2})} \\
    &\leq 2^\beta (\lambda + l)^{\underline{\beta}}
      \binom{\lambda + l + \gamma}{l + \beta}
      (\alpha!)^{1/2} \binom{\lambda - \beta + \gamma}{\alpha}^{1/2} \binom{\lambda + l}{\alpha}^{-1/2} \\
    &\leq (\text{polynomial of $ \lambda $ and $ l $})
      \times \Paren*{\frac{(\lambda + l)!}{l!}}^{1/2},
  \end{align*}
  which does not follow from Duran's estimate.
  This plays a key role in the proof of \cref{thm:laguerre-satisfies-pg}.
\end{remark}

\section{The condition (\texorpdfstring{$ a $}{a}-PG)}
\label{sec:pg}

In this section, we introduce a condition (\texorpdfstring{$ a $}{a}-PG)
on the growth rate of a family of functions,
and prove propositions related to this condition.
They will be used in the proof of \cref{thm:schwartz-space}
to estimate the radial parts of functions.

\subsection{The condition (\texorpdfstring{$ a $}{a}-PG)}

\begin{definition}\label{def:pg}
  Let $ a > 0 $.
  We say that a family $ \faml{h_{m, l}}{m, l \in \setN} $ of smooth functions
  on $ \setRzp $ satisfies \emph{($ a $-PG)} (``PG'' is derived from ``polynomial growth'')
  if, for any $ b \in \setR $ and $ n \in \setN $,
  \[
    \sup_{t \geq 0} t^{\max\setenum{(m + b)/a, 0}} \abs{h_{m, l}^{(n)}(t)}
  \]
  is of polynomial growth with respect to $ m $, $ l \in \setN $.
\end{definition}

\subsection{Decomposition of a family of functions satisfying (\texorpdfstring{$ a $}{a}-PG)}

In this subsection, we show that a certain type of decomposition of
a family of functions satisfying ($ a $-PG) is possible.

Recall that Borel's theorem states that,
for any sequence $ \faml{c^{(n)}}{n \in \setN} $ of complex numbers,
there exists a smooth function $ h $ on $ \setR $
such that $ h^{(n)}(0) = c^{(n)} $ for all $ n \in \setN $.
The following lemma is a generalization of this,
which we call \emph{Borel's theorem for families}.
The proof below is modeled on that of Borel's theorem,
but some additional arguments are required.

\begin{lemma}[Borel's theorem for families]\label{thm:borels-theorem-for-families}
  Let $ \faml{c_{m, l}^{(n)}}{m, l, n \in \setN} $ be a family of complex numbers
  such that $ \faml{c_{m, l}^{(n)}}{l \in \setN} $ is of polynomial growth
  for any $ m $, $ n \in \setN $.
  Then, there exists a family $ \faml{h_{m, l}}{m, l \in \setN} $
  of smooth functions on $ \setRzp $ such that
  \begin{itemize}
    \item $ \faml{h_{m, l}}{m, l \in \setN} $ satisfies ($ a $-PG) for all $ a > 0 $, and
    \item $ h_{m, l}^{(n)}(0) = c_{m, l}^{(n)} $ for all $ m $, $ l $, $ n \in \setN $.
  \end{itemize}
\end{lemma}

\begin{proof}
  Take a smooth function $ \rho $ on $ \setRzp $
  such that $ \rho = 1 $ near the origin and its support is contained in $ [0, 1] $.
  We construct $ \faml{h_{m, l}}{m, l \in \setN} $ satisfying the conditions of the form
  \begin{equation}
    h_{m, l}(t)
    = \sum_{n = 0}^{\infty} c_{m, l}^{(n)} \rho(R_{m, l}^{(n)} t) t^n,
      \label{eq:borels-theorem-for-families-1}
  \end{equation}
  where $ R_{m, l}^{(n)} > 0 $ are chosen later
  so that the infinite sum converges with respect to the $ C^\infty $ topology.
  It is clear that $ h_{m, l}^{(n)}(0) = c_{m, l}^{(n)} $
  for all $ m $, $ l $, $ n \in \setN $.

  For $ \beta \geq 0 $ and $ \gamma \in \setN $, we have
  \[
    t^\beta h_{m, l}^{(\gamma)}(t)
    = \sum_{n = 0}^{\infty} c_{m, l}^{(n)}
      \sum_{j = 0}^{\gamma}
      \binom{\gamma}{j} (R_{m, l}^{(n)})^{\gamma - j} n^{\underline{j}}
      \rho^{(\gamma - j)}(R_{m, l}^{(n)} t) t^{n - j + \beta}.
  \]
  Since $ \rho^{(\gamma - j)}(R_{m, l}^{(n)} t) = 0 $ unless $ t \in [0, (R_{m, l}^{(n)})^{-1}] $,
  it follows that
  \begin{align}
    \abs{t^\beta h_{m, l}^{(\gamma)}(t)}
    &\leq \sum_{n = 0}^{\infty} \abs{c_{m, l}^{(n)}}
      \sum_{j = 0}^{\gamma}
      \binom{\gamma}{j} (R_{m, l}^{(n)})^{\gamma - j} n^{\underline{j}}
      \norm{\rho^{(\gamma - j)}}_\infty (R_{m, l}^{(n)})^{-(n - j + \beta)}
      \notag \\
    &= \sum_{n = 0}^{\infty} C_{n, \gamma} \abs{c_{m, l}^{(n)}} (R_{m, l}^{(n)})^{-n - \beta + \gamma},
      \label{eq:borels-theorem-for-families-2}
  \end{align}
  where $ C_{n, \gamma}
  = \sum_{j = 0}^{\gamma} n^{\underline{j}} \norm{\rho^{(\gamma - j)}}_\infty $.
  We now set
  \[
    R_{m, l}^{(n)}
    = \max \setenum*{
        \max_{\gamma \in \setenum{0, \dots, n - 1}} (2^n C_{n, \gamma} \abs{c_{m, l}^{(n)}})^{\frac{1}{n - \gamma}},\;
        1
      }
  \]
  so that $ (R_{m, l}^{(n)})^{-n - \beta + \gamma}
  \leq (R_{m, l}^{(n)})^{-n + \gamma}
  \leq (2^n C_{n, \gamma} \abs{c_{m, l}^{(n)}})^{-1} $
  for $ n \geq \gamma + 1 $.
  Then, we have
  \begin{align}
    & \sum_{n = 0}^{\infty} C_{n, \gamma} \abs{c_{m, l}^{(n)}} (R_{m, l}^{(n)})^{-n - \beta + \gamma}
      \notag \\
    &\leq \sum_{n = 0}^{\gamma} C_{n, \gamma} \abs{c_{m, l}^{(n)}} (R_{m, l}^{(n)})^{-n - \beta + \gamma}
      + \sum_{n = \gamma + 1}^{\infty} C_{n, \gamma} \abs{c_{m, l}^{(n)}} (2^n C_{n, \gamma} \abs{c_{m, l}^{(n)}})^{-1}
      \notag \\
    &\leq \sum_{n = 0}^{\gamma} C_{n, \gamma} \abs{c_{m, l}^{(n)}} (R_{m, l}^{(n)})^{-n - \beta + \gamma}
      + 2^{-\gamma}
      \notag \\
    &\leq \widetilde{C}_\gamma \sum_{n = 0}^{\gamma} \abs{c_{m, l}^{(n)}} (R_{m, l}^{(n)})^{-\beta + \gamma}
      + 2^{-\gamma},
      \label{eq:borels-theorem-for-families-3}
  \end{align}
  where $ \widetilde{C}_\gamma = \max_{n \in \setenum{0, \dots, \gamma}} C_{n, \gamma} $.
  (This shows the convergence of the infinite sum in \cref{eq:borels-theorem-for-families-1}
  with respect to the $ C^\infty $ topology,
  which justifies the definition of $ h_{m, l} $ and termwise differentiation.)
  We estimate $ \sum_{n = 0}^{\gamma} \abs{c_{m, l}^{(n)}} (R_{m, l}^{(n)})^{-\beta + \gamma} $.
  For any $ \gamma \in \setN $, we have
  \begin{align}
    \sum_{n = 0}^{\gamma} \abs{c_{m, l}^{(n)}} (R_{m, l}^{(n)})^{-\beta + \gamma}
    &\leq \sum_{n = 0}^{\gamma}
      \abs{c_{m, l}^{(n)}}
      \max\setenum{2^\gamma \widetilde{C}_\gamma \abs{c_{m, l}^{(n)}}, 1}^\gamma
      \notag \\
    &= \sum_{n = 0}^{\gamma}
      \max\setenum{2^{\gamma^2} \widetilde{C}_\gamma^\gamma \abs{c_{m, l}^{(n)}}^{\gamma + 1}, \abs{c_{m, l}^{(n)}}}.
      \label{eq:borels-theorem-for-families-4}
  \end{align}
  On the other hand, for $ \gamma \in \setN $ such that $ -\beta + \gamma + 1 \leq 0 $,
  we have
  \begin{align}
    \sum_{n = 0}^{\gamma} \abs{c_{m, l}^{(n)}} (R_{m, l}^{(n)})^{-\beta + \gamma}
    &\leq \sum_{n = 0}^{\gamma}
      \abs{c_{m, l}^{(n)}}
      \max\setenum{2^n C_{n, n - 1} \abs{c_{m, l}^{(n)}}, 1}^{-\beta + \gamma}
      \notag \\
    &\leq \sum_{n = 0}^{\gamma}
      \min\setenum{\abs{c_{m, l}^{(n)}}^{-\beta + \gamma + 1}, \abs{c_{m, l}^{(n)}}}
      \notag \\
    &\leq \gamma + 1.
      \label{eq:borels-theorem-for-families-5}
  \end{align}
  Here, we use $ 2^n $, $ C_{n, n - 1} \geq 1 $ in the second inequality,
  and $ \min\setenum{\abs{c_{m, l}^{(n)}}^{-\beta + \gamma + 1}, \abs{c_{m, l}^{(n)}}} \leq 1 $
  in the last inequality.

  Now let $ a > 0 $, $ b \in \setR $ and $ \gamma \in \setN $,
  and set $ \beta = \max\setenum{(m + b)/a, 0} $.
  Then, the last line of \cref{eq:borels-theorem-for-families-4} is
  of polynomial growth with respect to $ l $ for fixed $ m $.
  Moreover, there are only finitely many $ m \in \setN $
  such that $ -\beta + \gamma + 1 = -\max\setenum{(m + b)/a, 0} + \gamma + 1 > 0 $.
  Therefore, it follows from \cref{eq:borels-theorem-for-families-2},
  \cref{eq:borels-theorem-for-families-3},
  \cref{eq:borels-theorem-for-families-4} and \cref{eq:borels-theorem-for-families-5}
  that $ \faml{h_{m, l}}{m, l \in \setN} $ satisfies ($ a $-PG).
\end{proof}

\begin{lemma}\label{thm:pg-power}
  Let $ a > 0 $ and $ \faml{h_{m, l}}{m, l \in \setN} $ be a family of
  smooth functions on $ \setRzp $ satisfying ($ a $-PG).
  \begin{enumarabicp}
    \item \label{item:positive-integers-c-and-d}
          For any $ c $, $ d \in \setNp $,
          the family $ \faml{t \mapsto t^c h_{m, l}(t^d)}{m, l \in \setN} $ satisfies ($ a/d $-PG).
    \item We further assume that
          $ h_{m, l}^{(n)}(0) = 0 $ for all $ m $, $ l $, $ n \in \setN $.
          Then, for any $ c \in \setR $ and $ d > 0 $,
          the family $ \faml{t \mapsto t^c h_{m, l}(t^d)}{m, l \in \setN} $ satisfies ($ a/d $-PG).
  \end{enumarabicp}
\end{lemma}

\begin{proof}
  \begin{subproof}{(1)}
    Since each derivative of the function $ t \mapsto t^c h_{m, l}(t^d) $ is
    a linear combination of terms of the form $ t^{c'} h_{m, l}^{(n)}(t^d) $
    ($ c' $, $ n \in \setN $),
    it suffices to show that, for any $ b \in \setR $ and $ c' $, $ n \in \setN $,
    \begin{equation}
      \sup_{t \geq 0} t^{\max\setenum{d(m + b)/a, 0} + c'} \abs{h_{m, l}^{(n)}(t^d)}
        \label{eq:pg-power}
    \end{equation}
    is of polynomial growth with respect to $ m $, $ l \in \setN $.
    It follows immediately from the assumption that
    $ \faml{h_{m, l}}{m, l \in \setN} $ satisfies ($ a $-PG).
  \end{subproof}

  \begin{subproof}{(2)}
    The situation is similar to \cref{item:positive-integers-c-and-d}, but here we want to prove that,
    for any $ b $, $ c' \in \setR $ and $ n \in \setN $,
    \cref{eq:pg-power} is of polynomial growth with respect to $ m $, $ l \in \setN $.
    The same argument as in \cref{item:positive-integers-c-and-d} applies
    for $ m \in \setN $ such that $ \max\setenum{d(m + b)/a, 0} + c' \geq 0 $.
    We now consider $ m \in \setN $ such that $ \max\setenum{d(m + b)/a, 0} + c' < 0 $.
    By the assumption and Taylor's theorem, we have
    \[
      t^{-d\alpha} \abs{h_{m, l}^{(n)}(t^d)}
      \leq \frac{1}{\alpha!} \sup_{s \in [0, t^d]} \abs{h_{m, l}^{(n + \alpha)}(s)}
    \]
    for $ \alpha \in \setN $.
    Hence, for fixed $ m \in \setN $ such that $ \max\setenum{d(m + b)/a, 0} + c' < 0 $,
    \cref{eq:pg-power} is of polynomial growth with respect to $ l \in \setN $.
    Since there are only finitely many such $ m $, this completes the proof.
  \end{subproof}
\end{proof}

Using these two lemmas, we prove the following proposition.

\begin{proposition}\label{thm:pg-decomposition}
  Let $ a > 0 $ and $ \faml{h_{m, l}}{m, l \in \setN} $ be a family of
  smooth functions on $ \setRzp $ satisfying ($ a $-PG).
  Then, for any $ q \in \setNp $, there exist
  $ \faml{h_{m, l; 0}}{m, l \in \setN} $, $ \dots $, $ \faml{h_{m, l; q - 1}}{m, l \in \setN} $
  such that
  \begin{itemize}
    \item each $ \faml{h_{m, l; j}}{m, l \in \setN} $ satisfies ($ qa $-PG), and
    \item $ h_{m, l}(t) = \sum_{j = 0}^{q - 1} t^j h_{m, l; j}(t^q) $
          for $ t \geq 0 $.
  \end{itemize}
\end{proposition}

\begin{proof}
  By Borel's theorem for families (\cref{thm:borels-theorem-for-families}), there exist
  $ \faml{h_{m, l; 0}}{m, l \in \setN} $, $ \dots $, $ \faml{h_{m, l; q - 1}}{m, l \in \setN} $
  satisfying ($ qa $-PG)
  such that $ (h_{m, l; j})^{(n)}(0)/n! = h_{m, l}^{(nq + j)}(0)/(nq + j)! $
  for all $ m $, $ l \in \setN $ and $ j \in \setenum{0, \dots, q - 1} $.
  We set
  \[
    \widetilde{h}_{m, l}(t)
    = h_{m, l}(t) - \sum_{j = 0}^{q - 1} t^j h_{m, l; j}(t^q).
  \]
  Then, $ \faml{\widetilde{h}_{m, l}}{m, l \in \setN} $ satisfies ($ a $-PG)
  by \cref{thm:pg-power}~(1), and
  $ \widetilde{h}_{m, l}^{(n)}(0) = 0 $ for all $ m $, $ l \in \setN $.
  Hence, by \cref{thm:pg-power}~(2),
  $ \faml{t \mapsto \widetilde{h}_{m, l}(t^{1/q})}{m, l \in \setN} $ satisfies ($ qa $-PG).
  By redefining $ h_{m, l; 0}(t) $ as $ h_{m, l; 0}(t) + \widetilde{h}_{m, l}(t^{1/q}) $,
  we obtain the desired decomposition.
\end{proof}

\subsection{Laguerre polynomials and the condition (\texorpdfstring{$ a $}{a}-PG)}

By using the estimates of Laguerre polynomials in \cref{sec:laguerre-polynomial},
we prove the following proposition.

\begin{proposition}\label{thm:laguerre-satisfies-pg}
  Let $ a > 0 $ and $ \faml{\lambda_m}{m \in \setN} $ be a sequence on $ [-1, \infty) $
  such that $ \sup_{m \in \setN} \abs{\lambda_m - 2m/a} < \infty $.
  Then, the family
  \[
    \faml*{
      t \mapsto \Paren*{\frac{\Gamma(l + 1)}{\Gamma(\lambda_m + l + 1)}}^{1/2} L_l^{\lambda_m}(t) e^{-t/2}
    }{m, l \in \setN}
  \]
  satisfies ($ a $-PG).
\end{proposition}

\begin{proof}
  By \cref{thm:estimate-1}, for $ m $, $ l $, $ \beta $, $ \gamma \in \setN $,
  we have
  \begin{align*}
    & \sup_{t \geq 0} \abs*{
        t^\beta \Paren*{\Odif{t}}^\gamma \Paren*{
          \Paren*{\frac{\Gamma(l + 1)}{\Gamma(\lambda_m + l + 1)}}^{1/2} L_l^{\lambda_m}(t) e^{-t/2}
        }
      } \\
    &\leq  \Paren*{\frac{\Gamma(l + 1)}{\Gamma(\lambda_m + l + 1)}}^{1/2} \cdot
      2^{\beta + \max\setenum{\beta - \lambda_m, 0}} (l + \beta)^{\underline{\beta}}
      \binom{\max\setenum{\lambda_m - \beta, 0} + l + \gamma}{l}
  \end{align*}
  Hence, for any $ \beta $, $ \gamma \in \setN $ and fixed $ m \in \setN $,
  the left-hand side of the above inequality is of polynomial growth
  with respect to $ l $.
  Since there are only finitely many $ m \in \setN $ such that $ \lambda_m \leq 0 $
  by the assumption $ \sup_{m \in \setN} \abs{\lambda_m - 2m/a} < \infty $,
  now it suffices to consider the case $ \lambda_m > 0 $.

  By \cref{thm:estimate-2}, for $ m $, $ l $, $ \alpha $, $ \beta $, $ \gamma \in \setN $
  such that $ \lambda_m > 0 $ and $ \lambda_m \geq \alpha + \beta $, we have
  \begin{align*}
    & \sup_{t \geq 0} \abs*{
        t^{\alpha/2 + \beta} \Paren*{\Odif{t}}^\gamma \Paren*{
          \Paren*{\frac{\Gamma(l + 1)}{\Gamma(\lambda_m + l + 1)}}^{1/2} L_l^{\lambda_m}(t) e^{-t/2}
        }
      } \\
    &\leq \Paren*{\frac{\Gamma(l + 1)}{\Gamma(\lambda_m + l + 1)}}^{1/2} \cdot
      2^\beta (\lambda_m + l)^{\underline{\beta}}
      \binom{\lambda_m + l + \gamma}{l + \beta}
      (\alpha!)^{1/2} \binom{\lambda_m - \beta + \gamma}{\alpha}^{1/2} \binom{\lambda_m + l}{\alpha}^{-1/2} \\
    &= \frac{2^\beta (\lambda_m + l)^{\underline{\beta}}}{((\lambda_m - \alpha - \beta + \gamma)!)^{1/2}}
      \frac{(\alpha!)^{1/2}}{((\lambda_m - \beta + \gamma)!)^{1/2}}
      \frac{(l!)^{1/2} ((\lambda_m + l - \alpha)!)^{1/2}}{(l + \beta)!}
      \frac{(\lambda_m + l + \gamma)!}{(\lambda_m + l)!}.
  \end{align*}
  If we fix $ M \geq 0 $ and let $ m $, $ l $, $ \alpha $ and $ \beta $ vary
  under the conditions $ \abs{\lambda_m - \alpha} \leq M $ and $ \beta \leq M $,
  then each fraction in the last line of the above inequality is bounded above
  by a polynomial of $ m $ and $ l $
  (whose degree and coefficients depend on $ M $ and $ \gamma $).
  Note that the condition $ \abs{\lambda_m - \alpha} \leq M $ is satisfied
  as long as $ \abs{\alpha/2 - m/a} \leq (M - C)/2 $,
  where we set $ C = \sup_{m' \in \setN} \abs{\lambda_{m'} - 2m'/a} < \infty $.
  Because it holds for any $ M \geq 0 $, we conclude that
  $ \faml{t \mapsto (\frac{\Gamma(l + 1)}{\Gamma(\lambda_m + l + 1)})^{1/2} L_l^{\lambda_m}(t) e^{-t/2}}{m, l \in \setN} $
  satisfies ($ a $-PG).
\end{proof}

\section{Inequalities for harmonic polynomials}
\label{sec:harmonic-polynomial}

In this section, we prove inequalities for harmonic polynomials,
which will be used in the proof of \cref{thm:schwartz-space}
to estimate the spherical parts of functions.

For $ p $, $ q \in \mcalP(\setR^N) $, we define the \emph{Fischer inner product} by
\[
  \innprod{p}{q}_{\mathrm{F}}
  = (p(\partial) \conj{q})(0),
\]
where $ p(\partial) $ denotes the differential operator obtained
by replacing the variables in $ p $ with $ \Pdif{x_1} $, $ \dots $, $ \Pdif{x_N} $.
We write $ \norm{\blank}_{\mathrm{F}} $ for the \emph{Fischer norm},
or the norm defined by the Fischer inner product.
The Fischer inner product is related to the $ L^2 $ inner product
(with respect to the standard measure on $ \Sphere{N - 1} $) as follows.

\begin{fact}[{\cite[Lemma~3.12]{MR1130821}}]\label{thm:fischer-and-l2}
  For $ p \in \mcalH^m(\setR^N) $ and $ q \in \mcalH^n(\setR^N) $
  ($ m $, $ n \in \setN $),
  we have
  \[
    \innprod{p}{q}_{\mathrm{F}}
    = \begin{dcases}
        \frac{2^{m - 1} \Gamma(m + N/2)}{\pi^{N/2}} \innprod{p}{q}_{L^2(\Sphere{N - 1})}
        & (m = n) \\
        0
        & (m \neq n).
      \end{dcases}
  \]
\end{fact}

\begin{lemma}\label{thm:derivatives-of-harmonic-polynomial}
  For $ p \in \mcalH^m(\setR^N) $ ($ m \in \setN $) and $ \gamma \in \setN^N $,
  we have
  \[
    \norm{\partial^\gamma p}_{L^2(\Sphere{N - 1})}
    \leq
    \begin{dcases}
      \Paren*{(2m)^{\abs{\gamma}} \Paren*{m + \frac{N}{2} - 1}^{\underline{\abs{\gamma}}}}^{1/2}
        \norm{p}_{L^2(\Sphere{N - 1})}
      & (\abs{\gamma} \leq m) \\
      0
      & (\abs{\gamma} > m)
    \end{dcases}
  \]
\end{lemma}

\begin{proof}
  Since it is clear that $ \partial^\gamma p = 0 $ when $ \abs{\gamma} > m $,
  we consider the case $ \abs{\gamma} \leq m $ in what follows.

  The monomials of degree $ m $ form an orthogonal basis of $ \mcalP^m(\setR^N) $
  with respect to the Fischer inner product.
  Since $ \partial^\gamma x^\alpha
  = \alpha_1^{\underline{\gamma_1}} \dotsm \alpha_N^{\underline{\gamma_N}} x^{\alpha - \gamma} $,
  we have
  \[
    \norm{\partial^\gamma p}_{\mathrm{F}}
    \leq \Paren*{
        \max_{\abs{\alpha} = m}
        \alpha_1^{\underline{\gamma_1}} \dotsm \alpha_N^{\underline{\gamma_N}}
      }
      \norm{p}_{\mathrm{F}}
    \leq m^{\abs{\gamma}} \norm{p}_{\mathrm{F}}
  \]
  for $ p \in \mcalP^m(\setR^N) $.
  Now assume that $ p \in \mcalH^m(\setR^N) $,
  so that $ \partial^\gamma p \in \mcalH^{m - \abs{\gamma}}(\setR^N) $.
  The above inequality, together with \cref{thm:fischer-and-l2},
  yields the desired inequality.
\end{proof}

We then prove pointwise estimates for the derivatives of harmonic polynomials.
The following fact is needed for this purpose.

\begin{fact}[{\cite[Corollary~IV.2.9~(b)]{MR304972}}]\label{thm:uniform-norm-of-spherical-harmonics}
  For $ p \in \mcalH^m(\Sphere{N - 1}) $ ($ m \in \setN $),
  we have
  \[
    \sup_{\omega \in \Sphere{N - 1}} \abs{p(\omega)}
    \leq \Paren*{\frac{\dim \mcalH^m(\setR^N)}{\vol(\Sphere{N - 1})}}^{1/2}
      \norm{p}_{L^2(\Sphere{N - 1})}.
  \]
\end{fact}

\begin{proposition}\label{thm:pointwise-estimates-for-derivatives-of-harmonic-polynomials}
  For $ p \in \mcalH^m(\setR^N) $ ($ m \in \setN $) and $ \gamma \in \setN^N $,
  we have
  \begin{align*}
    & \abs{\partial^\gamma p(x)} \\
    &\leq
    \begin{dcases}
      \Paren*{
        (2m)^{\abs{\gamma}}
        \Paren*{m + \frac{N}{2} - 1}^{\underline{\abs{\gamma}}} \,
        \frac{\dim \mcalH^{m - \abs{\gamma}}(\setR^N)}{\vol(\Sphere{N - 1})}
      }^{1/2}
      \norm{p}_{L^2(\Sphere{N - 1})}
      \enorm{x}^{m - \abs{\gamma}}
      & (\abs{\gamma} \leq m) \\
      0
      & (\abs{\gamma} > m)
    \end{dcases}
  \end{align*}
  for $ x \in \setR^N $.
\end{proposition}

\begin{proof}
  It follows from \cref{thm:uniform-norm-of-spherical-harmonics} and
  \cref{thm:derivatives-of-harmonic-polynomial},
  together with the fact that $ \partial^\gamma p \in \mcalH^{m - \abs{\gamma}}(\setR^N) $.
\end{proof}

\section*{Acknowledgements}
\addcontentsline{toc}{section}{Acknowledgements}

The author is deeply grateful to Professor Toshiyuki Kobayashi for helpful
comments and warm encouragement.
The author also thanks Dr.\ Yugo Takanashi for valuable conversations.

This work was supported by JSPS KAKENHI Grant Number JP25KJ0914.

\bibliographystyle{amsalpha}
\addcontentsline{toc}{section}{References}
\bibliography{back_references}

@article {MR2566988,
    AUTHOR = {Ben Sa{\"i}d, Salem and Kobayashi, Toshiyuki and {\O}rsted, Bent},
     TITLE = {Generalized {F}ourier transforms {$\mathscr{F}_{k,a}$}},
   JOURNAL = {C. R. Math. Acad. Sci. Paris},
  FJOURNAL = {Comptes Rendus Math{\'e}matique. Acad{\'e}mie des Sciences. Paris},
    VOLUME = {347},
      YEAR = {2009},
    NUMBER = {19-20},
     PAGES = {1119--1124},
      ISSN = {1631-073X,1778-3569},
   MRCLASS = {43A32 (42A16)},
  MRNUMBER = {2566988},
MRREVIEWER = {Karl-Hermann\ Neeb},
       DOI = {10.1016/j.crma.2009.07.015},
       URL = {https://doi.org/10.1016/j.crma.2009.07.015},
}

@article {MR2956043,
    AUTHOR = {Ben Sa{\"i}d, Salem and Kobayashi, Toshiyuki and {\O}rsted, Bent},
     TITLE = {Laguerre semigroup and {D}unkl operators},
   JOURNAL = {Compos. Math.},
  FJOURNAL = {Compositio Mathematica},
    VOLUME = {148},
      YEAR = {2012},
    NUMBER = {4},
     PAGES = {1265--1336},
      ISSN = {0010-437X,1570-5846},
   MRCLASS = {33C52 (22E46 43A32)},
  MRNUMBER = {2956043},
MRREVIEWER = {Hendrik\ De Bie},
       DOI = {10.1112/S0010437X11007445},
       URL = {https://doi.org/10.1112/S0010437X11007445},
}

@article {MR917849,
    AUTHOR = {Dunkl, Charles F.},
     TITLE = {Reflection groups and orthogonal polynomials on the sphere},
   JOURNAL = {Math. Z.},
  FJOURNAL = {Mathematische Zeitschrift},
    VOLUME = {197},
      YEAR = {1988},
    NUMBER = {1},
     PAGES = {33--60},
      ISSN = {0025-5874,1432-1823},
   MRCLASS = {42C05 (22E45 33A65 33A75)},
  MRNUMBER = {917849},
MRREVIEWER = {Walter\ Schempp},
       DOI = {10.1007/BF01161629},
       URL = {https://doi.org/10.1007/BF01161629},
}

@article {MR951883,
    AUTHOR = {Dunkl, Charles F.},
     TITLE = {Differential-difference operators associated to reflection
              groups},
   JOURNAL = {Trans. Amer. Math. Soc.},
  FJOURNAL = {Transactions of the American Mathematical Society},
    VOLUME = {311},
      YEAR = {1989},
    NUMBER = {1},
     PAGES = {167--183},
      ISSN = {0002-9947,1088-6850},
   MRCLASS = {33A45 (20H15 33A65 42C10 51F15)},
  MRNUMBER = {951883},
MRREVIEWER = {W.\ A.\ Al-Salam},
       DOI = {10.2307/2001022},
       URL = {https://doi.org/10.2307/2001022},
}

@article {MR1145585,
    AUTHOR = {Dunkl, Charles F.},
     TITLE = {Integral kernels with reflection group invariance},
   JOURNAL = {Canad. J. Math.},
  FJOURNAL = {Canadian Journal of Mathematics. Journal Canadien de
              Math{\'e}matiques},
    VOLUME = {43},
      YEAR = {1991},
    NUMBER = {6},
     PAGES = {1213--1227},
      ISSN = {0008-414X,1496-4279},
   MRCLASS = {33C80 (20F55)},
  MRNUMBER = {1145585},
MRREVIEWER = {Eric\ M.\ Opdam},
       DOI = {10.4153/CJM-1991-069-8},
       URL = {https://doi.org/10.4153/CJM-1991-069-8},
}

@incollection {MR1199124,
    AUTHOR = {Dunkl, Charles F.},
     TITLE = {Hankel transforms associated to finite reflection groups},
 BOOKTITLE = {Hypergeometric functions on domains of positivity, {J}ack
              polynomials, and applications ({T}ampa, {FL}, 1991)},
    SERIES = {Contemp. Math.},
    VOLUME = {138},
     PAGES = {123--138},
 PUBLISHER = {Amer. Math. Soc., Providence, RI},
      YEAR = {1992},
      ISBN = {0-8218-5159-4},
   MRCLASS = {33C80 (42B10 43A32 44A15)},
  MRNUMBER = {1199124},
MRREVIEWER = {Luis\ Verde-Star},
       DOI = {10.1090/conm/138/1199124},
       URL = {https://doi.org/10.1090/conm/138/1199124},
}

@article {MR1121714,
    AUTHOR = {Duran, Antonio J.},
     TITLE = {A bound on the {L}aguerre polynomials},
   JOURNAL = {Studia Math.},
  FJOURNAL = {Polska Akademia Nauk. Instytut Matematyczny. Studia
              Mathematica},
    VOLUME = {100},
      YEAR = {1991},
    NUMBER = {2},
     PAGES = {169--181},
      ISSN = {0039-3223,1730-6337},
   MRCLASS = {33C45},
  MRNUMBER = {1121714},
MRREVIEWER = {J.\ W.\ Brown},
       DOI = {10.4064/sm-100-2-169-181},
       URL = {https://doi.org/10.4064/sm-100-2-169-181},
}

@book {MR1446489,
    AUTHOR = {Faraut, Jacques and Kor{\'a}nyi, Adam},
     TITLE = {Analysis on symmetric cones},
    SERIES = {Oxford Mathematical Monographs},
      NOTE = {Oxford Science Publications},
 PUBLISHER = {The Clarendon Press, Oxford University Press, New York},
      YEAR = {1994},
     PAGES = {xii+382},
      ISBN = {0-19-853477-9},
   MRCLASS = {17C37 (22E30 33C80 35A30 43A85 46-02)},
  MRNUMBER = {1446489},
MRREVIEWER = {Hong\ Ming\ Ding},
}

@misc {arXiv2507-04064,
    AUTHOR = {Faustino, Nelson and Negzaoui, Selma},
     TITLE = {A novel {S}chwartz space for the {$ \left(k, \frac{2}{n}\right) $}-generalized {F}ourier transform},
      YEAR = {2025},
       DOI = {10.48550/arXiv.2507.04064},
       URL = {https://arxiv.org/abs/2507.04064v2},
      NOTE = {preprint},
}

@book {MR983366,
    AUTHOR = {Folland, Gerald B.},
     TITLE = {Harmonic analysis in phase space},
    SERIES = {Annals of Mathematics Studies},
    VOLUME = {122},
 PUBLISHER = {Princeton University Press, Princeton, NJ},
      YEAR = {1989},
     PAGES = {x+277},
      ISBN = {0-691-08527-7; 0-691-08528-5},
   MRCLASS = {22E30 (43A80 58G15 81S30)},
  MRNUMBER = {983366},
       DOI = {10.1515/9781400882427},
       URL = {https://doi.org/10.1515/9781400882427},
}

@article {MR4629458,
    AUTHOR = {Gorbachev, Dmitry and Ivanov, Valerii and Tikhonov, Sergey},
     TITLE = {On the kernel of the {$(\kappa, a)$}-generalized {F}ourier
              transform},
   JOURNAL = {Forum Math. Sigma},
  FJOURNAL = {Forum of Mathematics. Sigma},
    VOLUME = {11},
      YEAR = {2023},
     PAGES = {Paper No. e72, 25},
      ISSN = {2050-5094},
   MRCLASS = {42B10 (33C45)},
  MRNUMBER = {4629458},
MRREVIEWER = {Vitaly\ V.\ Volchkov},
       DOI = {10.1017/fms.2023.69},
       URL = {https://doi.org/10.1017/fms.2023.69},
}

@book {MR1130821,
    AUTHOR = {Gilbert, John E. and Murray, Margaret A. M.},
     TITLE = {Clifford algebras and {D}irac operators in harmonic analysis},
    SERIES = {Cambridge Studies in Advanced Mathematics},
    VOLUME = {26},
 PUBLISHER = {Cambridge University Press, Cambridge},
      YEAR = {1991},
     PAGES = {viii+334},
      ISBN = {0-521-34654-1},
   MRCLASS = {42B20 (15A66 30E10 30G35 31B15 35C15 58G03)},
  MRNUMBER = {1130821},
MRREVIEWER = {John\ Ryan},
       DOI = {10.1017/CBO9780511611582},
       URL = {https://doi.org/10.1017/CBO9780511611582},
}

@book {MR3307944,
    AUTHOR = {Gradshteyn, I. S. and Ryzhik, I. M.},
     TITLE = {Table of integrals, series, and products},
   EDITION = {Eighth},
      NOTE = {Translated from the Russian,
              Translation edited and with a preface by Daniel Zwillinger and
              Victor Moll},
 PUBLISHER = {Elsevier/Academic Press, Amsterdam},
      YEAR = {2015},
     PAGES = {xlvi+1133},
      ISBN = {978-0-12-384933-5},
   MRCLASS = {00A22 (33-00)},
  MRNUMBER = {3307944},
}

@article {MR3201818,
    AUTHOR = {Hilgert, Joachim and Kobayashi, Toshiyuki and M{\"o}llers, Jan},
     TITLE = {Minimal representations via {B}essel operators},
   JOURNAL = {J. Math. Soc. Japan},
  FJOURNAL = {Journal of the Mathematical Society of Japan},
    VOLUME = {66},
      YEAR = {2014},
    NUMBER = {2},
     PAGES = {349--414},
      ISSN = {0025-5645,1881-1167},
   MRCLASS = {22E45 (17C30 33C10)},
  MRNUMBER = {3201818},
MRREVIEWER = {Zhanqiang\ Bai},
       DOI = {10.2969/jmsj/06620349},
       URL = {https://doi.org/10.2969/jmsj/06620349},
}

@article {MR4684153,
    AUTHOR = {Ivanov, V. I.},
     TITLE = {One-dimensional {$(k,a)$}-generalized {F}ourier transform},
   JOURNAL = {Tr. Inst. Mat. Mekh.},
  FJOURNAL = {Trudy Instituta Matematiki i Mekhaniki},
    VOLUME = {29},
      YEAR = {2023},
    NUMBER = {4},
     PAGES = {92--108},
      ISSN = {0134-4889,2658-4786},
   MRCLASS = {42A38 (33C45)},
  MRNUMBER = {4684153},
MRREVIEWER = {Branko\ Sari{\'c}},
}

@article {MR2134314,
    AUTHOR = {Kobayashi, Toshiyuki and Mano, Gen},
     TITLE = {Integral formulas for the minimal representation of {$\mathrm{O}(p,2)$}},
   JOURNAL = {Acta Appl. Math.},
  FJOURNAL = {Acta Applicandae Mathematicae},
    VOLUME = {86},
      YEAR = {2005},
    NUMBER = {1-2},
     PAGES = {103--113},
      ISSN = {0167-8019,1572-9036},
   MRCLASS = {22E30 (22E46 43A80)},
  MRNUMBER = {2134314},
MRREVIEWER = {Daniel\ Belti{\c t}{\u a}},
       DOI = {10.1007/s10440-005-0464-2},
       URL = {https://doi.org/10.1007/s10440-005-0464-2},
}

@article {MR2317306,
    AUTHOR = {Kobayashi, Toshiyuki and Mano, Gen},
     TITLE = {Integral formula of the unitary inversion operator for the
              minimal representation of {${\rm O}(p,q)$}},
   JOURNAL = {Proc. Japan Acad. Ser. A Math. Sci.},
  FJOURNAL = {Japan Academy. Proceedings. Series A. Mathematical Sciences},
    VOLUME = {83},
      YEAR = {2007},
    NUMBER = {3},
     PAGES = {27--31},
      ISSN = {0386-2194},
   MRCLASS = {22E30 (22E46 43A80)},
  MRNUMBER = {2317306},
MRREVIEWER = {Daniel\ Belti{\c t}{\u a}},
       URL = {http://projecteuclid.org/euclid.pja/1176126886},
}

@incollection {MR2401813,
    AUTHOR = {Kobayashi, Toshiyuki and Mano, Gen},
     TITLE = {The inversion formula and holomorphic extension of the minimal
              representation of the conformal group},
 BOOKTITLE = {Harmonic analysis, group representations, automorphic forms
              and invariant theory},
    SERIES = {Lect. Notes Ser. Inst. Math. Sci. Natl. Univ. Singap.},
    VOLUME = {12},
     PAGES = {151--208},
 PUBLISHER = {World Sci. Publ., Hackensack, NJ},
      YEAR = {2007},
      ISBN = {978-981-277-078-3; 981-277-078-X},
   MRCLASS = {22E45 (33C80 43A80 43A85)},
  MRNUMBER = {2401813},
MRREVIEWER = {Tomasz\ Przebinda},
       DOI = {10.1142/9789812770790\_0006},
       URL = {https://doi.org/10.1142/9789812770790_0006},
}

@article {MR2858535,
    AUTHOR = {Kobayashi, Toshiyuki and Mano, Gen},
     TITLE = {The {S}chr{\"o}dinger model for the minimal representation of
              the indefinite orthogonal group {${\rm O}(p,q)$}},
   JOURNAL = {Mem. Amer. Math. Soc.},
  FJOURNAL = {Memoirs of the American Mathematical Society},
    VOLUME = {213},
      YEAR = {2011},
    NUMBER = {1000},
     PAGES = {vi+132},
      ISSN = {0065-9266,1947-6221},
      ISBN = {978-0-8218-4757-2},
   MRCLASS = {22E30 (22E46 43A80)},
  MRNUMBER = {2858535},
MRREVIEWER = {Takeshi\ Kawazoe},
       DOI = {10.1090/S0065-9266-2011-00592-7},
       URL = {https://doi.org/10.1090/S0065-9266-2011-00592-7},
}

@article {MR454106,
    AUTHOR = {Koornwinder, Tom},
     TITLE = {The addition formula for {L}aguerre polynomials},
   JOURNAL = {SIAM J. Math. Anal.},
  FJOURNAL = {SIAM Journal on Mathematical Analysis},
    VOLUME = {8},
      YEAR = {1977},
    NUMBER = {3},
     PAGES = {535--540},
      ISSN = {0036-1410},
   MRCLASS = {33A65},
  MRNUMBER = {454106},
MRREVIEWER = {M.\ Mikol{\'a}s},
       DOI = {10.1137/0508041},
       URL = {https://doi.org/10.1137/0508041},
}

@article {MR1755901,
    AUTHOR = {Kostant, Bertram},
     TITLE = {On {L}aguerre polynomials, {B}essel functions, {H}ankel
              transform and a series in the unitary dual of the
              simply-connected covering group of {${\rm Sl}(2,{\bf R})$}},
   JOURNAL = {Represent. Theory},
  FJOURNAL = {Representation Theory. An Electronic Journal of the American
              Mathematical Society},
    VOLUME = {4},
      YEAR = {2000},
     PAGES = {181--224},
      ISSN = {1088-4165},
   MRCLASS = {22E46 (22D10 33C10 33C45 33C80 42C05 43A65)},
  MRNUMBER = {1755901},
MRREVIEWER = {Sergei\ S.\ Platonov},
       DOI = {10.1090/S1088-4165-00-00096-0},
       URL = {https://doi.org/10.1090/S1088-4165-00-00096-0},
}

@incollection {MR2022853,
    AUTHOR = {R{\"o}sler, Margit},
     TITLE = {Dunkl operators: theory and applications},
 BOOKTITLE = {Orthogonal polynomials and special functions ({L}euven, 2002)},
    SERIES = {Lecture Notes in Math.},
    VOLUME = {1817},
     PAGES = {93--135},
 PUBLISHER = {Springer, Berlin},
      YEAR = {2003},
      ISBN = {3-540-40375-2},
   MRCLASS = {33C47 (20F55 37J35 39A70)},
  MRNUMBER = {2022853},
MRREVIEWER = {Walter\ Schempp},
       DOI = {10.1007/3-540-44945-0\_3},
       URL = {https://doi.org/10.1007/3-540-44945-0_3},
}

@book {MR2953553,
    AUTHOR = {Schm{\"u}dgen, Konrad},
     TITLE = {Unbounded self-adjoint operators on {H}ilbert space},
    SERIES = {Graduate Texts in Mathematics},
    VOLUME = {265},
 PUBLISHER = {Springer, Dordrecht},
      YEAR = {2012},
     PAGES = {xx+432},
      ISBN = {978-94-007-4752-4},
   MRCLASS = {47-01 (47B25 47E05)},
  MRNUMBER = {2953553},
MRREVIEWER = {G.\ V.\ Rozenblum},
       DOI = {10.1007/978-94-007-4753-1},
       URL = {https://doi.org/10.1007/978-94-007-4753-1},
}

@book {MR2279709,
    AUTHOR = {Sepanski, Mark R.},
     TITLE = {Compact {L}ie groups},
    SERIES = {Graduate Texts in Mathematics},
    VOLUME = {235},
 PUBLISHER = {Springer, New York},
      YEAR = {2007},
     PAGES = {xiv+198},
      ISBN = {978-0-387-30263-8; 0-387-30263-8},
   MRCLASS = {22Exx},
  MRNUMBER = {2279709},
MRREVIEWER = {David\ A.\ Renard},
       DOI = {10.1007/978-0-387-49158-5},
       URL = {https://doi.org/10.1007/978-0-387-49158-5},
}

@book {MR304972,
    AUTHOR = {Stein, Elias M. and Weiss, Guido},
     TITLE = {Introduction to {F}ourier analysis on {E}uclidean spaces},
    SERIES = {Princeton Mathematical Series},
    VOLUME = {No. 32},
 PUBLISHER = {Princeton University Press, Princeton, NJ},
      YEAR = {1971},
     PAGES = {x+297},
   MRCLASS = {42A92 (31B99 32A99 46F99 47G05)},
  MRNUMBER = {304972},
MRREVIEWER = {Edwin\ Hewitt},
}

@article {MR2328257,
    AUTHOR = {Vershik, A. M. and Graev, M. I.},
     TITLE = {The structure of complementary series and special
              representations of the groups {${\rm O}(n,1)$} and {${\rm
              U}(n,1)$}},
   JOURNAL = {Uspekhi Mat. Nauk},
  FJOURNAL = {Uspekhi Matematicheskikh Nauk},
    VOLUME = {61},
      YEAR = {2006},
    NUMBER = {5(371)},
     PAGES = {3--88},
      ISSN = {0042-1316,2305-2872},
   MRCLASS = {22E46},
  MRNUMBER = {2328257},
       DOI = {10.1070/RM2006v061n05ABEH004356},
       URL = {https://doi.org/10.1070/RM2006v061n05ABEH004356},
}

@book {MR929683,
    AUTHOR = {Wallach, Nolan R.},
     TITLE = {Real reductive groups. {I}},
    SERIES = {Pure and Applied Mathematics},
    VOLUME = {132},
 PUBLISHER = {Academic Press, Inc., Boston, MA},
      YEAR = {1988},
     PAGES = {xx+412},
      ISBN = {0-12-732960-9},
   MRCLASS = {22E46 (17B10 22-02 22E30)},
  MRNUMBER = {929683},
MRREVIEWER = {Roberto\ J.\ Miatello},
}

@article {MR7783,
    AUTHOR = {Whitney, Hassler},
     TITLE = {Differentiable even functions},
   JOURNAL = {Duke Math. J.},
  FJOURNAL = {Duke Mathematical Journal},
    VOLUME = {10},
      YEAR = {1943},
     PAGES = {159--160},
      ISSN = {0012-7094,1547-7398},
   MRCLASS = {27.0X},
  MRNUMBER = {7783},
MRREVIEWER = {A.\ E.\ Taylor},
       URL = {http://projecteuclid.org/euclid.dmj/1077471799},
}

\end{document}